\RequirePackage{fix-cm}
\documentclass[smallextended,envcountsect]{svjour3}       
\smartqed  
\usepackage{amssymb,amsfonts,amsmath}
\usepackage[mathscr]{eucal}
\usepackage{bm}
\usepackage{xcolor}  
\usepackage{algorithm,algpseudocode} 
\usepackage{xparse}
\usepackage{extarrows} 
\usepackage{subfig}
\usepackage{booktabs} 
\usepackage{graphicx} 
\usepackage{siunitx} 
\usepackage{tablefootnote} 
\usepackage{tikz}
\usetikzlibrary {quotes} 
\usetikzlibrary {arrows.meta} 
\usepackage{hyperref}
\usepackage[nameinlink]{cleveref} 

\DeclareMathOperator{\Trace}{Trace}
\DeclareMathOperator{\TR}{TR}
\NewDocumentCommand \tensor {O{}m} {\boldsymbol{#1\mathscr{\MakeUppercase{#2}}}} 
\newcommand{\mat}[1]{\mathbf{#1}}

\newcommand{\bb}[1]{\mathbb{#1}}
\newcommand{\bigO}[1]{\mathcal{O}\left( #1 \right)}
\newcommand{\Lref}[1]{\hyperref[#1]{\Cref{#1}}}
\newcommand{\tabincell}[2]{\begin{tabular}{@{}#1@{}}#2\end{tabular}}

\begin{document}
\begin{sloppypar}

\title{Practical Alternating Least Squares for Tensor Ring Decomposition
}

\titlerunning{Practical TR-ALS}        

\author{Yajie Yu         \and
        Hanyu Li 
}


\institute{Yajie Yu \at
              College of Mathematics and Statistics, Chongqing University, Chongqing 401331, P.R. China \\
              \email{zqyu@cqu.edu.cn}           
           \and
           Hanyu Li \at
              College of Mathematics and Statistics, Chongqing University, Chongqing 401331, P.R. China\\
              \email{lihy.hy@gmail.com or hyli@cqu.edu.cn}
}

\date{Received: date / Accepted: date}


\maketitle

\begin{abstract}
Tensor ring (TR) decomposition has been widely applied as an effective approach in a variety of applications to discover the hidden low-rank patterns in multidimensional data. 
A well-known method for TR decomposition is the alternating least squares (ALS). However, it often suffers from the notorious intermediate data explosion issue, especially for large-scale tensors.
In this paper, we provide two strategies to tackle this issue and design three ALS-based algorithms. 
Specifically, the first strategy is used to simplify the calculation of the coefficient matrices of the normal equations for the ALS subproblems, which takes full advantage of the structure of the coefficient matrices of the subproblems and hence makes the corresponding algorithm perform much better than the regular ALS method in terms of computing time. The second strategy is to stabilize the ALS subproblems by QR factorizations on TR-cores, and hence the corresponding algorithms are more numerically stable compared with our first algorithm. Extensive numerical experiments on synthetic and real data are given to illustrate and confirm the above results. In addition, we also present the complexity analyses of the proposed algorithms.

\end{abstract}

\section{Introduction}
\label{sec:introduction}
\emph{Tensor ring} (TR) decomposition is a simple but powerful tensor network for analyzing and interpreting latent patterns for multidimensional data, i.e., tensors, due to its excellent compression and data representation capacities \cite{zhao2016TensorRing}. 
It is also referred to as the tensor chain format in previous mathematics literature \cite{khoromskij2011DlogQuantics}, or matrix product states format with periodic boundaries in physics literature \cite{affleck1988ValenceBond,perez-garcia2007MatrixProduct}. 
Specifically, the decomposition aims to represent a high-order tensor by a sequence of 3rd-order tensors that are multiplied circularly, and the specific format of element-wise form for the tensor $\tensor{X} \in \bb{R}^{I_1 \times I_2 \times \cdots \times I_N}$ is:
\begin{align*}
	\tensor{X}(i_1, \cdots, i_N)&={\Trace} \left(\mat{G}_1(i_1) \mat{G}_2(i_2) \cdots \mat{G}_N(i_N)\right)
	={\Trace} \left(\prod_{n=1}^N \mat{G}_n(i_n)\right),
\end{align*}
where $\mat{G}_n(i_n) = \tensor{G}_n(:,i_n,:) \in \bb{R}^{R_n \times R_{n+1}}$ is the $i_n$-th \emph{lateral slice} of the \emph{core tensor (TR-core)} $\tensor{G}_n \in \bb{R}^{R_n \times I_n \times R_{n+1}}$. Note that a \emph{slice} is a 2nd-order section, i.e., a matrix, of a tensor obtained by fixing all the tensor indices but two, and $ R_{N+1}=R_{1} $. The sizes of TR-cores, i.e., $R_k$ with $k=1,\cdots,N$, are called \emph{TR-ranks}. Additionally, we use the notation $\TR \left( \{\tensor{G}_n\}_{n=1}^N \right)$ to denote the TR decomposition of a tensor.

Besides TR decomposition, there exist a number of other tensor decompositions, such as CANDECOMP/PARAFAC (CP) decomposition \cite{carroll1970AnalysisIndividual,harshman1970FoundationsPARAFAC}, Tucker decomposition \cite{tucker1966MathematicalNotes} and tensor train (TT) decomposition \cite{oseledets2011TensorTrainDecomposition}.
Among them, finding the optimal CP decomposition of a tensor is an NP-hard problem, Tucker decomposition always suffers from the curse of dimensionality, and TT decomposition requires the rank constraint and strict order.
TR decomposition overcomes these deficiencies well and has become popular in recent years.

The problem of fitting $\TR \left( \{\tensor{G}_n\}_{n=1}^N \right)$ to a tensor $\tensor{X}$ can be written as the following minimization problem:
\begin{equation}
	\label{eq:trmin}
	\mathop{\arg\min}_{\tensor{G}_1, \cdots, \tensor{G}_N} \left\| \TR \left( \{\tensor{G}_n\}_{n=1}^N \right) - \tensor{X} \right\|_F,
\end{equation} 
where $\| \cdot \|_F$ denotes the Frobenius norm of a matrix or tensor. 
Deterministic algorithms for computing the above TR decomposition can be mainly divided into two categories \cite{zhao2016TensorRing}. One is SVD-based, and another is the alternating least squares (ALS). 
We use the abbreviation TR-SVD and TR-ALS for these two methods in the following text. 
The former involves the computation of a series of SVDs of unfolding matrices and the latter needs to compute the pseudoinverses of the coefficient matrices in the ALS subproblems within iterations. 
Clearly, both of them are computationally prohibitive for large-scale tensors. 
See \cite{zhao2016TensorRing} and \cite{mickelin2020AlgorithmsComputing} for further details on deterministic algorithms for TR decomposition.

Randomized algorithms for TR decomposition are effective methods to deal with the problem mentioned above, which can reduce the computational complexities of deterministic algorithms and the communications between different levels of the memory hierarchy, and have been developed in a huge number of works.
For more specific, Yuan et al. \cite{yuan2019RandomizedTensor} devised a method that first applies a randomized Tucker decomposition, followed by a TR decomposition of the core tensor using either TR-SVD or TR-ALS. Then, the TR cores are contracted appropriately. 
Later, Ahmadi-Asl et al. \cite{ahmadi-asl2020RandomizedAlgorithms} developed several randomized variants of the deterministic TR-SVD by replacing the SVDs with their randomized counterparts.
As for TR-ALS, Malik and Becker \cite{malik2021SamplingBasedMethod} proposed the TR-ALS-Sampled, which uses the leverage score sampling to reduce the size of the ALS subproblems.
Moreover, Malik \cite{malik2022MoreEfficient} also provided a new approach to approximate the leverage scores and devised the corresponding sampling algorithm. 
Very recently, we proposed two sketching-based randomized algorithms for TR decomposition in \cite{yu2022PracticalSketchingBased} and also suggested an algorithmic framework based on random projection.

Although the ALS-based randomized algorithms can reduce the computational complexities greatly, like deterministic algorithms, they also usually have to directly solve the ALS subproblems eventually and hence may suffer from the intermediate data explosion issue. In \cite{yu2022PracticalSketchingBased}, we found the special structure of the coefficient matrices of the ALS subproblems by defining a new tensor product called \emph{subchain product}. It offers a possibility to tackle the above issue. 
With the structure and a property of subchain product, we first propose a strategy to accelerate the calculation of the coefficient matrices of the normal equations for the ALS subproblems. This is similar to what is done in CP-ALS for CP decomposition, where the property of Khatri–Rao product is employed to simplify the calculations \cite{carroll1970AnalysisIndividual,harshman1970FoundationsPARAFAC,kolda2009TensorDecompositions}. Then, we investigate the QR factorizations of the coefficient matrices of the ALS subproblems via the structure and defining the QR factorization for the 3rd-order tensor to stabilize and facilitate the subproblems. This strategy is similar to the one for CP decomposition adopted in \cite{minster2021CPDecomposition}. Tracing the source, it is the famous idea for stabilizing the least squares (LS) problem. As a result, the algorithm with this strategy is more stable than the one with the first strategy. Furthermore, we also combine these 
two algorithms to balance the running time and stability.

The remainder of this paper is organized as follows. 
\Cref{sec:preliminaries} introduces some necessary definitions, propositions, and TR-ALS. 
In \Cref{sec:algorithms}, we propose our main algorithms and analyze their computational complexities. 
\Cref{sec:experiments} is devoted to numerical experiments to illustrate and confirm our
methods. 
Finally, the concluding remarks of the whole paper are presented.

\section{Preliminaries}
\label{sec:preliminaries}

Before presenting the necessary definitions regarding tensors, we 
denote $[I] \xlongequal{def} \{ 1, \cdots, I \}$ for a positive integer $I$, and set  $\overline{i_1 i_2 \cdots i_N} \xlongequal{def} 1 + \sum_{n=1}^N(i_n-1)\prod_{j=1}^{n-1}I_j$ for the indices $i_1 \in [I_1], \cdots, i_N \in [I_N]$. 

\begin{definition}[Classical Mode-$n$ Unfolding]
	The \textbf{classical mode-$n$ unfolding} of a tensor $\tensor{X} \in \bb{R}^{I_1 \times I_2 \cdots \times I_N}$ is the matrix $\mat{X}_{(n)}$ of size $I_n \times \prod_{j \ne n} I_j$ defined element-wise via
	\begin{equation*}
	\mat{X}_{(n)}(i_n, \overline{i_1 \cdots i_{n-1} i_{n+1} \cdots i_N})=\tensor{X}(i_1, \cdots, i_N).
	\end{equation*}
\end{definition}

\begin{definition}[Mode-$n$ Unfolding]
	The \textbf{mode-$n$ unfolding} of a tensor $\tensor{X} \in \bb{R}^{I_1 \times I_2 \cdots \times I_N}$ is the matrix $\mat{X}_{[n]}$ of size $I_n \times \prod_{j \ne n} I_j$ defined element-wise via
	\begin{equation*}
	\mat{X}_{[n]}(i_n, \overline{i_{n+1} \cdots i_N i_1 \cdots i_{n-1}})=\tensor{X}(i_1, \cdots, i_N).
	\end{equation*}
\end{definition}

\begin{definition}[$n$-Unfolding]
\label{def:n_unfolding}
	The \textbf{$n$-unfolding} of a tensor $\tensor{X} \in \bb{R}^{I_1 \times I_2 \cdots \times I_N}$ is the matrix $\mat{X}_{<n>}$ of size $\prod_{j=1}^{n} I_j \times \prod_{j=n+1}^{N} I_j$ defined element-wise via
	\begin{equation*}
	\mat{X}_{<n>}(\overline{i_1, \cdots, i_n}, \overline{i_{n+1} \cdots i_N})=\tensor{X}(i_1, \cdots, i_N).
	\end{equation*}
\end{definition}

\begin{definition}[Outer Product]
\label{def:outer_product}
	For $\tensor{A} \in \bb{R}^{I_1 \times \cdots \times I_N}$ and $\tensor{B} \in \bb{R}^{J_1 \times \cdots \times J_M}$, 
	their \textbf{outer product} 
	is a tensor of size $I_1 \times \cdots \times I_N \times J_1 \times \cdots \times J_M$ denoted by $\tensor{A} \circ \tensor{B}$ and defined element-wise via
	\begin{equation*}
		(\tensor{A} \circ \tensor{B})(i_1, \cdots, i_N, j_1, \cdots, j_M) = \tensor{A}(i_1, \cdots, i_N) \tensor{B}(j_1, \cdots, j_M). 
	\end{equation*}
\end{definition}

\begin{definition}[Contracted Tensor Product]\label{def:ct_product}
	For $\tensor{A} \in \bb{R}^{I_1 \times \cdots \times I_N}$ and $\tensor{B} \in \bb{R}^{J_1 \times \cdots \times J_M}$ 
	with $I_n = J_m$, their \textbf{contracted tensor product} 
	is a 
	tensor of size $I_1 \times \cdots \times I_{n-1} \times I_{n+1} \times \cdots \times I_N \times J_1 \times \cdots \times J_{m-1} \times J_{m+1} \times \cdots \times J_M$ denoted by $\tensor{A} \times_n^m \tensor{B}$ and defined element-wise via
	\begin{align*}
		&(\tensor{A} \times_n^m \tensor{B})(i_1, \cdots, i_{n-1}, i_{n+1}, \cdots, i_N, j_1, \cdots, j_{m-1}, j_{m+1}, \cdots, j_M) \\
		&\qquad = \sum_{i=1}^{I_n}\tensor{A}(i_1, \cdots, i_{n-1}, i, i_{n+1}, \cdots, i_N) \tensor{B}(j_1, \cdots, j_{m-1}, i, j_{m+1}, \cdots, j_M).
	\end{align*}
\end{definition}

\begin{definition}[TTM]
    For $\tensor{X} \in \bb{R}^{I_1 \times I_2 \cdots \times I_N}$ and $\mat{U} \in \bb{R}^{J \times I_n}$, their \textbf{tensor-times-matrix (TTM) multiplication} 
    is a tensor of size $I_1 \times \cdots \times I_{n-1} \times J \times I_{n+1} \times \cdots \times I_N$ denoted by $\tensor{X} \times_n \mat{U}$ and defined element-wise via
	\begin{equation*}
	(\tensor{X} \times_n \mat{U})(i_1, \cdots, i_{n-1}, j, i_{n+1}, \cdots, i_N) = \sum_{i_n = 1}^{I_n} \tensor{X}(i_1, \cdots, i_n, \cdots, i_N) \mat{U}(j, i_n).
	\end{equation*}
\end{definition}

It is easy to find that TTM is a special case of the contracted tensor product. Furthermore, multiplying an $N$th-order tensor by multiple matrices on distinct modes is known as \emph{Multi-TTM}, whose computation can be performed by using a sequence of individual mode TTMs, and can be done in any order. In particular, multiplying an $N$th-order tensor by matrices $\mat{U}_j$ with $j=1,\cdots, N$ in each mode implies $\tensor{Y} = \tensor{X} \times_1 \mat{U}_1 \times_2 \mat{U}_2 \cdots \times_N \mat{U}_N$. Its 
mode-$n$ unfolding can be presented as follows.

\begin{proposition}
\label{prop:ttm}
	Let the tensor $\tensor{Y}$ have the form $\tensor{Y} = \tensor{X} \times_1 \mat{U}_1 \times_2 \mat{U}_2 \cdots \times_N \mat{U}_N$, where $\tensor{X} \in \bb{R}^{I_1 \times I_2 \cdots \times I_N}$, and $\mat{U}_n \in \bb{R}^{J_n \times I_n}$ for $n = {1, \cdots, N}$. Then
	\begin{align*}
	\mat{Y}_{[n]} = \mat{U}_n \mat{X}_{[n]} \left( \mat{U}_{n-1} \otimes \cdots \otimes \mat{U}_{1} \otimes \mat{U}_{N} \otimes \cdots \otimes \mat{U}_{n+1} \right)^\intercal.
	\end{align*}
\end{proposition}

We now detail the TR-ALS mentioned in \Cref{sec:introduction}, which is the most popular algorithm for TR decomposition. 
To achieve this, we need the following definition.
\begin{definition}
\label{def:subchain_tensor}
	Let $\tensor{X} = \TR \left( \{\tensor{G}_n\}_{n=1}^N \right) \in \bb{R}^{I_1 \times I_2 \cdots \times I_N}$. The \textbf{subchain tensor} $\tensor{G}^{\ne n} \in \bb{R}^{R_{n+1} \times \prod_{j \ne n} I_j \times R_n}$ is the merging of all TR-cores expect the $n$-th one and can be written slice-wise via
	\begin{equation*}
	\mat{G}^{\ne n}(\overline{i_{n+1} \cdots i_N i_1 \cdots i_{n-1}})=\prod_{j=n+1}^{N} \mat{G}_j(i_j) \prod_{j=1}^{n-1} \mat{G}_j(i_j).
	\end{equation*}
\end{definition}

Thus, according to Theorem 3.5 in \cite{zhao2016TensorRing}, the objective in \eqref{eq:trmin} can be rewritten as the following $N$ subproblems
\begin{equation}
\label{eq:tr_als}
\mathop{\arg\min}_{\mat{G}_{n(2)}} \left\|\mat{G}_{[2]}^{\ne n} \mat{G}_{n(2)}^\intercal-\mat{X}_{[n]}^\intercal \right\|_F,\ n=1,\cdots, N.
\end{equation}
The so-called TR-ALS is a method that keeps all cores fixed except the $n$-th one and finds the solution to the LS problem \eqref{eq:tr_als} with respect to it. We summarize the method in \Cref{alg:tr_als}.

\begin{algorithm}
\caption{TR-ALS \cite{zhao2016TensorRing}}
\label{alg:tr_als}
	\textbf{Input:} $\tensor{X} \in \bb{R}^{I_1 \times \cdots \times I_N}$, TR-ranks $R_1, \cdots, R_N$
	
	\textbf{Output:} TR-cores $\{ \tensor{G}_n \in \bb{R}^{R_n \times I_n \times R_{n+1}} \}_{n=1}^N$
	\begin{algorithmic}[1]\footnotesize
		\State Initialize TR-cores $\tensor{G}_1, \cdots, \tensor{G}_N$ \label{line:als_init}
		\Repeat
		\For{$n = 1, \cdots, N$}
		\State Compute $\mat{G}_{[2]}^{\ne n}$ from cores \label{line:als_subchain}
		\State Update $\tensor{G}_n = \mathop{\arg\min}_{\tensor{Z}} \left\| \mat{G}_{[2]}^{\ne n} \mat{Z}_{(2)}^\intercal - \mat{X}_{[n]}^\intercal \right\|_F$ \label{line:als_ls}
		\EndFor
		\Until{termination criteria met}
	\end{algorithmic}
\end{algorithm}

The matrix $\mat{G}_{[2]}^{\ne n}$ in \eqref{eq:tr_als}, i.e., the subchain tensor $\tensor{G}^{\ne n}$, has a special structure, which can be revealed elegantly with the subchain product defined as follows. 
\begin{definition}[Subchain Product \cite{yu2022PracticalSketchingBased}] 
\label{def:subchain_product}
	For $\tensor{A} \in \bb{R}^{I_1 \times J_1 \times K}$ and $\tensor{B} \in \bb{R}^{K \times J_2 \times I_2}$, 
	their mode-2 \textbf{subchain product} 
	is a tensor of size $I_1 \times J_1 J_2 \times I_2$ denoted by $\tensor{A} \boxtimes_2 \tensor{B}$ and  defined as 
	\begin{equation*}
	(\tensor{A} \boxtimes_2 \tensor{B})(\overline{j_1 j_2}) = \tensor{A}(j_1)\tensor{B}(j_2),
	\end{equation*}
	where $\mat{A}(j_1)$ and $\mat{B}(j_2)$ are the $j_1$-th and $j_2$-th lateral slices of $\tensor{A}$ and $\tensor{B}$, respectively.
	That is, with respect to the correspondence on indices, the lateral slices of $\tensor{A} \boxtimes_2 \tensor{B}$ are the classical matrix multiplications of the lateral slices of $\tensor{A}$ and $\tensor{B}$. The mode-1 and mode-3 subchain products can be defined similarly.
\end{definition}

Thus, $\tensor{G}^{\ne n}$ can be expressed as \cite{yu2022PracticalSketchingBased}:
\begin{equation}
\label{eq:subchain_new}
	\tensor{G}^{\ne n} = \tensor{G}_{n+1} \boxtimes_2 \cdots \boxtimes_2 \tensor{G}_{N} \boxtimes_2 \tensor{G}_{1} \boxtimes_2 \cdots \boxtimes_2 \tensor{G}_{n-1}.
\end{equation}
With this expression and the property of subchain product given in \Cref{prop:slice_kron}, we devised some randomized algorithms
for TR decomposition based on the Kronecker
sub-sampled randomized Fourier transform and TensorSketch \cite{yu2022PracticalSketchingBased}.
\begin{proposition}[\cite{yu2022PracticalSketchingBased}]
\label{prop:slice_kron}
	Let $\tensor{A} \in \bb{R}^{I_1 \times J_1 \times K}$ and $\tensor{B} \in \bb{R}^{K \times J_2 \times I_2}$ be two 3rd-order tensors, and $\mat{A} \in \bb{R}^{R_1 \times J_1}$ and $\mat{B} \in \bb{R}^{R_2 \times J_2}$ be two matrices. Then
	\begin{equation*}
	(\tensor{A} \times_2 \mat{A}) \boxtimes_2 (\tensor{B} \times_2 \mat{B}) = (\tensor{A} \boxtimes_2 \tensor{B}) \times_2 (\mat{B} \otimes \mat{A}).
	\end{equation*}
\end{proposition}

\section{Proposed Methods}
\label{sec:algorithms} 
In this section, we will present three practical ALS-based algorithms for TR decomposition. One is built on normal equation and another one is based on QR factorization. The third one is the combination of the preceding two algorithms. In addition, we also present a by-product on inner product.

\subsection{TR-ALS based on normal equation}
Recall that the LS problem from \eqref{eq:tr_als} is typically solved by using the normal equation, i.e., 
\begin{equation*}
	\mat{X}_{[n]} \mat{G}_{[2]}^{\ne n} =  \mat{G}_{n(2)} \left( (\mat{G}_{[2]}^{\ne n})^\intercal \mat{G}_{[2]}^{\ne n} \right).
\end{equation*}
From \eqref{eq:subchain_new}, it follows that $(\mat{G}_{[2]}^{\ne n})^\intercal \mat{G}_{[2]}^{\ne n}$ can be written as $(\tensor{G}_{n+1} \boxtimes_2 \cdots \boxtimes_2 \tensor{G}_{N} \boxtimes_2 \tensor{G}_{1} \boxtimes_2 \cdots \boxtimes_2 \tensor{G}_{n-1})_{[2]}^\intercal (\tensor{G}_{n+1} \boxtimes_2 \cdots \boxtimes_2 \tensor{G}_{N} \boxtimes_2 \tensor{G}_{1} \boxtimes_2 \cdots \boxtimes_2 \tensor{G}_{n-1})_{[2]}$. To compute it efficiently, we now propose a property of subchain product using the outer product, contracted tensor product and $n$-unfolding. 

\begin{proposition}
\label{prop:subchain_gram}
	Let $\tensor{A} \in \bb{R}^{I_1 \times J \times K_1}$, $\tensor{B} \in \bb{R}^{K_1 \times R \times L_1}$, $\tensor{C} \in \bb{R}^{I_2 \times J \times K_2}$ and $\tensor{D} \in \bb{R}^{K_2 \times R \times L_2}$ be 3rd-order tensors. Then
	\begin{equation*}
		(\tensor{A} \boxtimes_2 \tensor{B})_{[2]}^\intercal (\tensor{C} \boxtimes_2 \tensor{D})_{[2]} = \left( (\sum_{r=1}^{R} \tensor{B}(r)^\intercal \circ \tensor{D}(r)^\intercal) \times_{2,4}^{1,3} (\sum_{j=1}^{J} \tensor{A}(j)^\intercal \circ \tensor{C}(j)^\intercal) \right)_{<2>}.
	\end{equation*}
\end{proposition}

\begin{proof}
	The proof is just to examine both sides of the equation directly by some algebraic operations. 
	The idea is simple, but the process is tedious. So we omit the proof here.
\end{proof}

\begin{remark}
    The contracted tensor product used in \Cref{prop:subchain_gram} is a little different from the one in \Cref{def:ct_product}. It can be regarded as a \emph{general contracted tensor product}. For $\tensor{A} \in \bb{R}^{I_1 \times J \times R_1 \times K}$ and $\tensor{B} \in \bb{R}^{J \times I_2 \times K \times R_2}$, their general product is a 4th-order tensor of size $I_1 \times I_2 \times R_1 \times R_2$ defined as
	\begin{equation*}
		(\tensor{A} \times_{2,4}^{1,3} \tensor{B})(i_1, i_2, r_1, r_2) = \sum_{j,k} \tensor{A}(i_1, j, r_1, k) \tensor{B}(j, i_2, k,r_2).
	\end{equation*}
	The graphical illustration of the product is shown in \Cref{fig:gengral_contract}.

\begin{figure}
\centering
	\scalebox{1}{
	\begin{tikzpicture}
		\node[circle, minimum width =20pt, minimum height =20pt, shading=ball, ball color=red!30, "{$\tensor{A}$}" left] (a) at(2,1){};
		\draw [densely dashed, fill=purple!10] (0.2,0.5) -| (1.9,-0.5) -| cycle;
		\draw [densely dashed, fill=teal!10] (2.1,0.5) -| (3.8,-0.5) -| cycle;
		\node (i1) at (0.5,0) {$I_1$};
		\node (j1) at (1.5,0) {$J$};
		\node (r1) at (2.5,0) {$R_1$};
		\node (k1) at (3.5,0) {$K$};
		\draw[thick] (a) -- (i1);
		\draw[thick] (a) -- (j1);
		\draw[thick] (a) -- (r1);
		\draw[thick] (a) -- (k1);
		
		\node[circle, minimum width =20pt, minimum height =20pt, shading=ball, ball color=red!30, "{$\tensor{B}$}" left] (b) at(6,1){};
		\draw [densely dashed, fill=purple!10] (4.2,0.5) -| (5.9,-0.5) -| cycle;
		\draw [densely dashed, fill=teal!10] (6.1,0.5) -| (7.8,-0.5) -| cycle;
		\node (j2) at (4.5,0) {$J$};
		\node (i2) at (5.5,0) {$I_2$};
		\node (k2) at (6.5,0) {$K$};
		\node (r2) at (7.5,0) {$R_2$};
		\draw[thick] (b) -- (j2);
		\draw[thick] (b) -- (i2);
		\draw[thick] (b) -- (k2);
		\draw[thick] (b) -- (r2);
		
		\node[circle, minimum width =20pt, minimum height =20pt, shading=ball, ball color=blue!30, "{$\tensor{A} \times_{2,4}^{1,3} \tensor{B}$}" left] (b) at(4,-3){};
		\draw [densely dashed, fill=purple!10] (2.2,-1.5) -| (3.9,-2.5) -| cycle;
		\draw [densely dashed, fill=teal!10] (4.1,-1.5) -| (5.8,-2.5) -| cycle;
		\node (ii1) at (2.5,-2) {$I_1$};
		\node (ii2) at (3.5,-2) {$I_2$};
		\node (rr1) at (4.5,-2) {$R_1$};
		\node (rr2) at (5.5,-2) {$R_2$};
		\draw[thick] (b) -- (ii1);
		\draw[thick] (b) -- (ii2);
		\draw[thick] (b) -- (rr1);
		\draw[thick] (b) -- (rr2);
		
		\draw[-Stealth] (1,-0.5) -- (3,-1.5);
		\draw[-Stealth] (5,-0.5) -- (3,-1.5);
		\draw[-Stealth] (3,-0.5) -- (5,-1.5);
		\draw[-Stealth] (7,-0.5) -- (5,-1.5);
	\end{tikzpicture}
	}
\caption{Illustration of the general contracted tensor product $\tensor{A} \times_{2,4}^{1,3} \tensor{B}$.}
\label{fig:gengral_contract}
\end{figure}
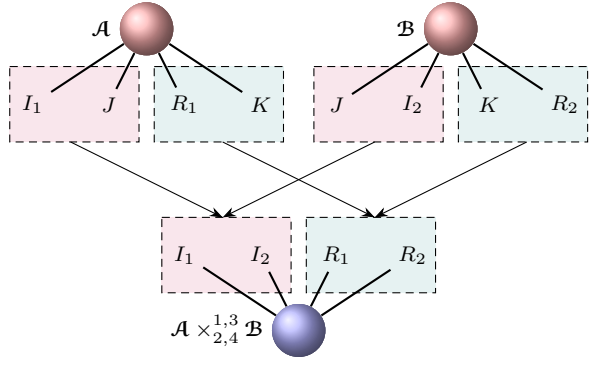
\end{remark}

 According to \Cref{prop:subchain_gram}, the coefficient matrix $(\mat{G}_{[2]}^{\ne n})^\intercal \mat{G}_{[2]}^{\ne n}$ can be computed efficiently as $(\mat{G}_{[2]}^{\ne n})^\intercal \mat{G}_{[2]}^{\ne n} = \mat{S}_{n<2>}$, 
 where 
 $$\tensor{S}_{n} =  \tensor{P}_{n-1} \times_{2,4}^{1,3} \cdots \times_{2,4}^{1,3} \tensor{P}_{1} \times_{2,4}^{1,3} \tensor{P}_{N} \times_{2,4}^{1,3} \cdots \times_{2,4}^{1,3} \tensor{P}_{n+1},$$ 
 and 
 $$\tensor{P}_{j} = \sum_{i_j = 1}^{I_j} \mat{G}_{j}(i_j)^\intercal \circ \mat{G}_{j}(i_j)^\intercal, ~~j \ne n.$$
 Putting the above together, we can devise an algorithm called TR-ALS-NE, whose details are listed in \Cref{alg:tr_als_ne}.
 
\begin{algorithm}
\caption{TR-ALS-NE (Proposal)}
\label{alg:tr_als_ne}
	\textbf{Input:} $\tensor{X} \in \bb{R}^{I_1 \times \cdots \times I_N}$, TR-ranks $R_1, \cdots, R_N$
	
	\textbf{Output:} TR-cores $\{ \tensor{G}_n \in \bb{R}^{R_n \times I_n \times R_{n+1}} \}_{n=1}^N$
	\begin{algorithmic}[1]
		\State Initialize TR-cores $\tensor{G}_{1}, \cdots, \tensor{G}_{N}$ \label{line:ne_init}
		\State Compute the Gram tensors $\tensor{P}_{1} = \sum_{i_1 = 1}^{I_1} \mat{G}_{1}(i_1)^\intercal \circ \mat{G}_{1}(i_1)^\intercal, \cdots, \tensor{P}_{N} = \sum_{i_N = 1}^{I_N} \mat{G}_{N}(i_N)^\intercal \circ \mat{G}_{N}(i_N)^\intercal$ \label{line:ne_gramn}
		\Repeat
		\For{$n = 1,\cdots,N$}
		\State $\tensor{S}_{n} \leftarrow \tensor{P}_{n-1} \times_{2,4}^{1,3} \cdots \times_{2,4}^{1,3} \tensor{P}_{1} \times_{2,4}^{1,3} \tensor{P}_{N} \times_{2,4}^{1,3} \cdots \times_{2,4}^{1,3} \tensor{P}_{n+1}$ \label{line:ne_gram}
		\State $\tensor{G}^{\ne n} \leftarrow \tensor{G}_{n+1} \boxtimes_2 \cdots \boxtimes_2 \tensor{G}_{N} \boxtimes_2 \tensor{G}_{1} \boxtimes_2 \cdots \boxtimes_2 \tensor{G}_{n-1}$ \label{line:ne_subchain}
		\State $\mat{M}_{n} \leftarrow \mat{X}_{[n]} \mat{G}_{[2]}^{\ne n}$ \label{line:ne_mttsp}
		\Comment matricized-tensor times subchain product (MTTSP)
		\State Solve $\mat{G}_{n(2)} \mat{S}_{n<2>} = \mat{M}_{n}$ \label{line:ne_normal}
		\Comment normal equation
		\State Recompute the Gram tensor $\tensor{P}_{n} = \sum_{i_n = 1}^{I_n} \mat{G}_{n}(i_n)^\intercal \circ \mat{G}_{n}(i_n)^\intercal$ for the updated TR-core $\tensor{G}_{n}$ \label{line:ne_regramn}
		\EndFor
		\Until termination criteria met
	\end{algorithmic}
\end{algorithm}

\begin{remark}
    It can be found that the methods TR-ALS and TR-ALS-NE are equivalent in mathematics. Hence, they have the same conclusions on convergence, which is also confirmed by the numerical results in {Experiment A-\uppercase\expandafter{\romannumeral2}} in \Cref{ssec:syn-A}. That is, they 
    require almost the same number of iterations to achieve the same error. Since a new way is adopted to compute $(\mat{G}_{[2]}^{\ne n})^\intercal \mat{G}_{[2]}^{\ne n}$, the cost of our method is cheaper than that of TR-ALS, which is supported by the discussions on the complexities of various methods in \Cref{ssec:tr_als_qene} and the numerical results in Experiments A-\uppercase\expandafter{\romannumeral1} and A-\uppercase\expandafter{\romannumeral2} in \Cref{ssec:syn-A}. For the convergence analyses of TR-ALS, refer to, e.g., \cite{espig2015ConvergenceAlternatinga,chen2020TensorRing} and references therein.
\end{remark}

In \Cref{alg:tr_als_ne}, we use a term called Gram tensor, whose formal definition is as follows. 
Note that we do not use it for the derivation of \Cref{alg:tr_als_ne}.
\begin{definition}[Gram Tensor]
	For $\tensor{A}_k \in \bb{R}^{I_1 \times \cdots \times I_N}$ with $k \in [K]$, 
	their \textbf{Gram tensor} 
	is an $(N+2)$nd-order tensor defined as
	\begin{equation*}
		\Delta(\tensor{A}_1, \tensor{A}_2, \cdots, \tensor{A}_K)=
		\begin{bmatrix}
			\langle \tensor{A}_1,\tensor{A}_1 \rangle & \langle \tensor{A}_1,\tensor{A}_2 \rangle & \cdots & \langle \tensor{A}_1,\tensor{A}_K \rangle \\
			\langle \tensor{A}_2,\tensor{A}_1 \rangle & \langle \tensor{A}_2,\tensor{A}_2 \rangle & \cdots & \langle \tensor{A}_2,\tensor{A}_K \rangle \\
			\vdots & \vdots & \cdots & \vdots \\
			\langle \tensor{A}_K,\tensor{A}_1 \rangle & \langle \tensor{A}_K,\tensor{A}_2 \rangle & \cdots & \langle \tensor{A}_K,\tensor{A}_K \rangle
		\end{bmatrix}.
	\end{equation*}
\end{definition}

Now, we analyze the computational complexity of TR-ALS-NE shown in \Cref{alg:tr_als_ne}. Recall that $\tensor{X}$ has dimensions $I_1 \times \cdots \times I_N$ and its TR-ranks are $R_1, \cdots, R_N$. To simplify notation, we make the assumption throughout the analysis that $I_n = I$ and $R_n = R$ for all $n \in [N]$, and that $N < I$ and $R^2< I$.
And, we ignore any cost associated with, e.g., checking termination conditions.

Upfront costs of TR-ALS-NE:

\textbf{\Lref{line:ne_init}: Initializing cores.} This depends on how to initiate the cores. We assume that they are randomly drawn, e.g., from a Gaussian distribution, resulting in a cost $\bigO{NIR^2}$.

\textbf{\Lref{line:ne_gramn}: Computing the Gram tensor.} It costs $\bigO{NIR^4}$ according to the definition of the outer product.

The costs of per outer loop iteration in TR-ALS-NE:

\textbf{\Lref{line:ne_gram}: Computing the general contracted tensor product.} It costs $\bigO{(N-2)R^6}$. Doing this for each of the $N$ cores in the inner loop brings the cost $\bigO{N(N-2)R^6}$.

\textbf{\Lref{line:ne_subchain}: Computing the unfolding subchain tensor.} If the $N-1$ cores are dense and contracted in sequence, the cost is $$R^3(I^2+I^3+ \cdots +I^{N-1}) \leq R^3(NI^{N-2}+I^{N-1}) \leq 2R^3I^{N-1} = \bigO{I^{N-1}R^3}.$$ 
Doing this for each of the $N$ cores in the inner loop brings the cost $\bigO{NI^{N-1}R^3}$.

\textbf{\Lref{line:ne_mttsp}: Computing MTTSP.} It costs $\bigO{I^{N}R^2}$ per inner loop iteration, i.e., $\bigO{NI^{N}R^2}$ per outer loop iteration.

\textbf{\Lref{line:ne_normal}: Solving the normal equation.} We consider the cost of the method for normal equation described in \cite[Section 5.3.2]{golub2013MatrixComputations}. It costs $\bigO{IR^6}$. Doing this for each of the $N$ cores in the inner loop brings the cost $\bigO{NIR^6}$.

\textbf{\Lref{line:ne_regramn}: Updating the Gram tensor.} It costs $\bigO{IR^4}$ per inner loop iteration, i.e., $\bigO{NIR^4}$ per outer loop iteration.
 
Putting them all together, we have that the leading order complexity of TR-ALS-NE is 
$$\bigO{NIR^4 + it \cdot N I^N R^2},$$ 
where ``$it$'' denotes the number of outer loop iterations.
It is easy to see that, in TR-ALS-NE, the cost is dominated by MTTSP.

\subsection{TR-ALS based on QR factorization}

We begin with a definition of QR factorization for the 3rd-order tensor.
\begin{definition}[Mode-$n$ QR Factorization]
	For $\tensor{A} \in \bb{R}^{I_1 \times I_2 \times I_3}$, 
	its \textbf{mode-$n$ QR factorization} 
	is defined as follows:
	
	(1) If $I_n \geq \prod_{j \ne n}I_j, n = 1,2,3$,
	\begin{equation*}
	\tensor{A} = \tensor{R} \times_n \mat{Q}, ~~n = 1,2,3,
	\end{equation*}
	where $\mat{Q} \in \bb{R}^{I_n \times \prod_{j \ne n}I_j}$ is an orthogonal matrix, and $\tensor{R}$ is a 3rd-order tensor whose mode-$n$ unfolding matrix is a upper triangular matrix of size $\prod_{j \ne n}I_j \times \prod_{j \ne n}I_j$.
	
	(2) If $I_n \leq \prod_{j \ne n}I_j, n = 1,2,3$,
	\begin{equation*}
	\tensor{A} = \tensor{R} \times_n \mat{Q}, ~~n = 1,2,3,
	\end{equation*}
	where $\mat{Q} \in \bb{R}^{I_n \times I_n}$ is an orthogonal matrix, and $\tensor{R}$ is a 3rd-order tensor whose mode-$n$ unfolding matrix is a upper triangular matrix of size $I_n \times \prod_{j \ne n}I_j$.
\end{definition}

Thus, noting \eqref{eq:subchain_new}, to form the QR factorization of $\mat{G}^{\ne n}_{[2]}$, a first step is to compute the mode-2 QR factorization of each individual TR-core, i.e., $\tensor{G}_j = \tensor{R}_j \times_2 \mat{Q}_j$ for $j \in [N]$. Then
\begin{align*}
\tensor{G}^{\ne n}
&= (\tensor{R}_{n+1} \times_2 \mat{Q}_{n+1}) \boxtimes_2 \cdots \boxtimes_2 (\tensor{R}_{N} \times_2 \mat{Q}_{N}) \boxtimes_2 (\tensor{R}_{1} \times_2 \mat{Q}_{1}) \boxtimes_2 \\
&\qquad \cdots \boxtimes_2 (\tensor{R}_{n-1} \times_2 \mat{Q}_{n-1}) \\
&= (\tensor{R}_{n+1} \boxtimes_2 \cdots \boxtimes_2 \tensor{R}_{N} \boxtimes_2 \tensor{R}_{1} \boxtimes_2 \cdots \boxtimes_2 \tensor{R}_{n-1}) \\
&\qquad \times_2 (\mat{Q}_{n-1} \otimes \cdots \otimes \mat{Q}_{1} \otimes \mat{Q}_{N} \otimes \cdots \otimes \mat{Q}_{n+1}),
\end{align*}
where the last equality comes from \Cref{prop:slice_kron}. Further, we compute the mode-2 QR factorization of the subchain product $\tensor{V}_{n} = \tensor{R}_{n+1} \boxtimes_2 \cdots \boxtimes_2 \tensor{R}_{N} \boxtimes_2 \tensor{R}_{1} \boxtimes_2 \cdots \boxtimes_2 \tensor{R}_{n-1} = \tensor{R}_{0} \times_2 \mat{Q}_{0}$. This allows us to express the mode-2 QR factorization of $\tensor{G}^{\ne n}$ as
\begin{align*}
	\tensor{G}^{\ne n} 
	&= (\tensor{R}_{0} \times_2 \mat{Q}_{0}) \times_2 (\mat{Q}_{n-1} \otimes \cdots \otimes \mat{Q}_{1} \otimes \mat{Q}_{N} \otimes \cdots \otimes \mat{Q}_{n+1}) \\
	&= \tensor{R}_{0} \times_2 \left( (\mat{Q}_{n-1} \otimes \cdots \otimes \mat{Q}_{1} \otimes \mat{Q}_{N} \otimes \cdots \otimes \mat{Q}_{n+1}) \mat{Q}_{0} \right) = \tensor{R} \times_2 \mat{Q},
\end{align*}
where the second equality comes from the property of TTM detailed in \cite{kolda2009TensorDecompositions}. Thus, from the above representation, we can get the QR factorization of $\mat{G}^{\ne n}_{[2]}$ as $\mat{G}^{\ne n}_{[2]} = \mat{Q} \mat{R}_{[2]}$ using \Cref{prop:ttm} eventually.

Once the above QR factorization is derived, the problem \eqref{eq:tr_als} can be written as 
\begin{equation*}
	\mathop{\arg\min}_{\mat{G}_{n(2)}} \left\| \mat{X}_{[n]} - \mat{G}_{n(2)} \mat{R}_{[2]}^\intercal \mat{Q}_{0}^\intercal (\mat{Q}_{n-1} \otimes \cdots \otimes \mat{Q}_{1} \otimes \mat{Q}_{N} \otimes \cdots \otimes \mat{Q}_{n+1})^\intercal \right\|_F,
\end{equation*}
where $\mat{R}_{[2]} \in \bb{R}^{R_{n+1}R_{n} \times R_{n+1}R_{n}}$, $\mat{Q}_{0} \in \bb{R}^{\frac{\prod_{n=1}^{N}R_n}{R_{n+1}R_{n}} \times R_{n+1}R_{n}}$, and $\mat{Q}_j \in \bb{R}^{I_j \times R_jR_{j+1} }$ for $j \in [N] \backslash n$.
Further, let $\tensor{Y}$ have the format of Multi-TTM, i.e., $\tensor{Y} = \tensor{X} \times_1 \mat{Q}_{1}^\intercal \cdots \times_{n-1} \mat{Q}_{n-1}^\intercal \times_{n+1} \mat{Q}_{n+1}^\intercal \cdots \times_{N} \mat{Q}_{N}^\intercal$. Then, using \Cref{prop:ttm}, we have
\begin{equation}
\label{eq:qrhalf_ls}
	 \mathop{\arg\min}_{\mat{G}_{n(2)}} \left\| \mat{Y}_{[n]} - \mat{G}_{n(2)} \mat{V}_{n[2]}^\intercal \right\|_F,
\end{equation}
where $\mat{V}_{n[2]} = \mat{Q}_{0} \mat{R}_{[2]}$ due to $\tensor{V}_{n} = \tensor{R}_{0} \times_2 \mat{Q}_{0}$. 
Finally, by forming $\mat{W}_{n} = \mat{Y}_{[n]} \mat{Q}_{0}$, we obtain a smaller LS problem
\begin{equation}
\label{eq:tri_ls}
	\mathop{\arg\min}_{\mat{G}_{n(2)}} \left\| \mat{W}_{n} - \mat{G}_{n(2)} \mat{R}_{[2]}^\intercal \right\|_F,
\end{equation}
from which we can compute $\mat{G}_{n(2)}$ and hence the TR-core $\tensor{G}_n$. We call this method TR-ALS-QR and detail it in \Cref{alg:tr_als_qr}.

\begin{algorithm}
\caption{TR-ALS-QR (Proposal)}
\label{alg:tr_als_qr}
	\textbf{Input:} $\tensor{X} \in \bb{R}^{I_1 \times \cdots \times I_N}$, TR-ranks $R_1, \cdots, R_N$
	
	\textbf{Output:} TR-cores $\{ \tensor{G}_n \in \bb{R}^{R_n \times I_n \times R_{n+1}} \}_{n=1}^N$
	\begin{algorithmic}[1]
		\State Initialize TR-cores $\tensor{G}_1, \cdots, \tensor{G}_N$ \label{line:qr_init}
		\State Compute the mode-2 QR factorizations $\tensor{R}_{1} \times_2 \mat{Q}_{1}, \cdots, \tensor{R}_{N} \times_2 \mat{Q}_{N}$ of TR-cores \label{line:qr_comput_qrdN}
		\Repeat
		\For{$n = 1,\cdots,N$}
		\State $\tensor{V}_{n} \leftarrow \tensor{R}_{n+1} \boxtimes_2 \cdots \boxtimes_2 \tensor{R}_{N} \boxtimes_2 \tensor{R}_{1} \boxtimes_2 \cdots \boxtimes_2 \tensor{R}_{n-1}$ \label{line:qr_subchain_prod}
		\State Compute mode-2 QR factorization $\tensor{V}_n = \tensor{R} \times_2 \mat{Q}_{0}$ \label{line:qr_comput_qrd0}
		\State $\tensor{Y} \leftarrow \tensor{X} \times_1 \mat{Q}_{1}^\intercal \cdots \times_{n-1} \mat{Q}_{n-1}^\intercal \times_{n+1} \mat{Q}_{n+1}^\intercal \cdots \times_{N} \mat{Q}_{N}^\intercal$ \label{line:qr_ttm}
		\State $\mat{W}_{n} \leftarrow \mat{Y}_{[n]} \mat{Q}_{0}$ \label{line:qr_wn}
		\State Solve $\mat{G}_{n(2)} \mat{R}_{[2]}^\intercal = \mat{W}_{n}$ by substitution \label{line:qr_solve}
		\State Recompute the mode-2 QR factorization $\tensor{R}_{n} \times_2 \mat{Q}_{n}$ for the updated TR-core $\tensor{G}_{n}$ \label{line:qr_recomput_qrd}
		\EndFor
		\Until termination criteria met
	\end{algorithmic}
\end{algorithm}

\begin{remark}
	If $\mat{G}^{\ne n}_{[2]}$ is rank deficient, another more stable way of solving \eqref{eq:tri_ls} is to use the SVD of $\mat{R}_{[2]}$ further \cite[Section 5.5]{golub2013MatrixComputations}.
\end{remark}

\begin{remark}
\label{rem:tr-als-qr}
    The method TR-ALS-QR is also equivalent to TR-ALS 
    in mathematics. As pointed out in \cite[Section 5.3]{golub2013MatrixComputations}, QR factorization can stabilize the LS problem and the methods based on normal equation are more sensitive. Hence, for ill-conditioned problems, TR-ALS-QR may need fewer iterations compared with TR-ALS-NE. This also suggests that TR-ALS-QR usually performs better than TR-ALS-NE in computing time though the former is a little more expensive than the latter. These results are supported by the comparisons on the complexities between these two methods in \Cref{ssec:tr_als_qene} and Experiments B-\uppercase\expandafter{\romannumeral2} and B-\uppercase\expandafter{\romannumeral3} in \Cref{ssec:syn-B}. 
\end{remark}

Now, we consider the complexity analysis of \Cref{alg:tr_als_qr} with the same assumptions as done for TR-ALS-NE. Here, we mainly treat $\tensor{V}_n$ as a dense tensor without discussing the possibility of exploiting sparsity. 

Upfront costs of TR-ALS-QR:

\textbf{\Lref{line:qr_init}: Initializing cores.} This is the same as the one for TR-ALS-NE. That is, we assume that the cores are randomly drawn, e.g., from a Gaussian distribution, resulting in a cost $\bigO{NIR^2}$.

\textbf{\Lref{line:qr_comput_qrdN}: Compute mode-2 QR factorization of TR-cores.} 
This is equivalent to computing the QR factorization of $\mat{G}_{n(2)}$ for $ n = 1,\cdots,N$ resulting in a cost $\bigO{NIR^4}$

The costs of per outer loop iteration in TR-ALS-QR:

\textbf{\Lref{line:qr_subchain_prod}: Compute $\tensor{V}_{n}$.} If $\tensor{R}_{j}$ for $ j \ne n$ are dense and the product is implemented in sequence, the cost is 
$$\bigO{R^3(R^4+R^6+ \cdots +R^{2N-2})} \leq \bigO{R^{2N+1}}.$$
Doing this for each of the $N$ cores in the inner loop brings the cost $\bigO{NR^{2N+1}}$.

\textbf{\Lref{line:qr_comput_qrd0}: Compute mode-2 QR factorization of $\tensor{V}_{n}$.} It costs $\bigO{R^{2N+2}}$ because $\mat{V}_{n}$ has dimensions $R \times R^{2N-2} \times R$, i.e., $\bigO{NR^{2N+2}}$ per outer loop iteration.

\textbf{\Lref{line:qr_ttm}: Multi-TTM.} We compute the resulting tensor $\tensor{Y}$, which has dimensions $R^2 \times \cdots \times I \times \cdots \times R^2$, by performing the single TTMs in sequence. Thus, the overall cost of the Multi-TTM is
\begin{equation*}
\bigO{R^2I^N(1+\frac{R^2}{I}+ \frac{R^4}{I^2} + \cdots + \frac{R^{2N-4}}{I^{N-2}})}
\end{equation*}
per inner loop iteration. Hence, under the assumption $I>R^2$, the Multi-TTM costs $\bigO{NI^NR^2}$ per outer loop iteration.

\textbf{\Lref{line:qr_wn}: Compute $\mat{W}_{n}$.} It costs $\bigO{IR^{2N}}$ per inner loop iteration, i.e., $\bigO{NIR^{2N}}$ per outer loop iteration.

\textbf{\Lref{line:qr_solve}: Solve triangular system.} It costs $\bigO{IR^4}$ per inner loop iteration, i.e., $\bigO{NIR^4}$ per outer loop iteration.

\textbf{\Lref{line:qr_recomput_qrd}: Recompute mode-2 QR factorization of the $n$-th TR-core.} It costs $\bigO{IR^4}$ per inner loop iteration, i.e., $\bigO{NIR^4}$ per outer loop iteration.

Putting them all together, we have that the leading order complexity of TR-ALS-QR is 
$$\bigO{NIR^4 + it \cdot N I^N R^2},$$ 
where ``$it$'' denotes the number of outer loop iterations.
So, in TR-ALS-QR, the cost is dominated by the Multi-TTM in \Lref{line:qr_ttm}.

\subsection{TR-ALS based on QR factorization and normal equation}
\label{ssec:tr_als_qene}
Recalling the derivation of TR-ALS-QR, we find that $\tensor{V}_n$ also has the subchain product structure. Moreover, the complexity analysis for TR-ALS-QR shows that computing the mode-2 QR factorization of $\tensor{V}_n$ is expensive. Thus, we propose to solve \eqref{eq:qrhalf_ls} using TR-ALS-NE. We call this method TR-ALS-QRNE and detail it in \Cref{alg:tr_als_qrne}.

\begin{algorithm}
\caption{TR-ALS-QRNE (Proposal)}
\label{alg:tr_als_qrne}
	\textbf{Input:} $\tensor{X} \in \bb{R}^{I_1 \times \cdots \times I_N}$, TR-ranks $R_1, \cdots, R_N$
	
	\textbf{Output:} TR-cores $\{ \tensor{G}_n \in \bb{R}^{R_n \times I_n \times R_{n+1}} \}_{n=1}^N$
	\begin{algorithmic}[1]
		\State Initialize TR-cores $\tensor{G}_1, \cdots, \tensor{G}_N$
		\State Compute the mode-2 QR factorizations $\tensor{R}_{1} \times_2 \mat{Q}_{1}, \cdots, \tensor{R}_{N} \times_2 \mat{Q}_{N}$ of TR-cores 
		\State Compute the Gram tensors $\tensor{P}_{1} = \sum_{i_1 = 1}^{I_1} \mat{R}_{1}(i_1)^\intercal \circ \mat{R}_{1}(i_1)^\intercal, \cdots, \tensor{P}_{N} = \sum_{i_N = 1}^{I_N} \mat{R}_{N}(i_N)^\intercal \circ \mat{R}_{N}(i_N)^\intercal$ 
		\Repeat
		\For{$n = 1,\cdots,N$}
		\State $\tensor{S}_{n} \leftarrow \tensor{P}_{n-1} \times_{2,4}^{1,3} \cdots \times_{2,4}^{1,3} \tensor{P}_{1} \times_{2,4}^{1,3} \tensor{P}_{N} \times_{2,4}^{1,3} \cdots \times_{2,4}^{1,3} \tensor{P}_{n+1}$
		\State $\tensor{V}_{n} \leftarrow \tensor{R}_{n+1} \boxtimes_2 \cdots \boxtimes_2 \tensor{R}_{N} \boxtimes_2 \tensor{R}_{1} \boxtimes_2 \cdots \boxtimes_2 \tensor{R}_{n-1}$
		\State $\tensor{Y} \leftarrow \tensor{X} \times_1 \mat{Q}_{1}^\intercal \cdots \times_{n-1} \mat{Q}_{n-1}^\intercal \times_{n+1} \mat{Q}_{n+1}^\intercal \cdots \times_{N} \mat{Q}_{N}^\intercal$
		\State $\mat{M}_{n} \leftarrow \mat{Y}_{[n]} \mat{V}_{n[2]}$ 
		\State Solve $\mat{G}_{n(2)} \mat{S}_{n<2>} = \mat{M}_{n}$ 
		\State Recompute the mode-2 QR factorization $ \tensor{R}_{n} \times_2 \mat{Q}_{n}$ for the updated TR-core $\tensor{G}_{n}$
		\State Recompute the Gram tensor $\tensor{P}_{n} = \sum_{i_n = 1}^{I_n} \mat{R}_{n}(i_n)^\intercal \circ \mat{R}_{n}(i_n)^\intercal$ for the updated $\tensor{R}_{n}$
		\EndFor
		\Until termination criteria met
	\end{algorithmic}
\end{algorithm}

The complexity analysis of \Cref{alg:tr_als_qrne} is similar to those of TR-ALS-NE and TR-ALS-QR when treating $\tensor{V}_n$ as a dense tensor, so we won't go into details here. The following \Cref{tab:complexity} summarizes the computational complexities of the algorithms involved in this paper in detail.  

\begin{table}[htbp]
	\centering
	\caption{Comparison of each part of computational complexities for different algorithms (only in the dense case).} 
	\label{tab:complexity}
	\resizebox{1\linewidth}{!}{
	\begin{tabular}{lllllll}  
		\toprule
		& \multicolumn{2}{c}{Upfront costs} & \multicolumn{4}{c}{Per outer loop iteration} \\
		\cmidrule(lr){2-3}
		\cmidrule(lr){4-7}
		Method 		 & {Initializing} & {Gram/QR}   & {MTTSP/TTM} & {Solving} & {Gram/QR} & {Others} \\
		\midrule
		TR-ALS 	     & $\bigO{NIR^2}$ & ---   & --- & \tabincell{l}{$\bigO{NI^{N-1}R^4}$\\$+\bigO{NI^NR^2}$\\$+\bigO{NIR^4}$} & --- & $\bigO{NI^{N-1}R^3}$ \\
		\midrule
		TR-ALS-NE    & $\bigO{NIR^2}$ & $\bigO{NIR^4}$   & $\bigO{NI^{N}R^2}$ & $\bigO{NIR^6}$ & $\bigO{NIR^4}$ & \tabincell{l}{$\bigO{N(N-2)R^6}$\\$+\bigO{NI^{N-1}R^3}$} \\
		TR-ALS-QR    & $\bigO{NIR^2}$ & $\bigO{NIR^4}$   & $\bigO{NR^2I^N(1+\frac{R^2}{I}+ \frac{R^4}{I^2} + \cdots + \frac{R^{2N-4}}{I^{N-2}})}$ & $\bigO{NIR^4}$ & $\bigO{NIR^4}$ & \tabincell{l}{$\bigO{NR^{2N+1}}$\\$+\bigO{NR^{2N+2}}$\\$+\bigO{NIR^{2N}}$} \\
		TR-ALS-QRNE  & $\bigO{NIR^2}$ & \tabincell{l}{$\bigO{NIR^4}$\\$+\bigO{NR^6}$}   & \tabincell{l}{$\bigO{NR^2I^N(1+\frac{R^2}{I}+ \frac{R^4}{I^2} + \cdots + \frac{R^{2N-4}}{I^{N-2}})}$\\$+\bigO{IR^{2N}}$} & $\bigO{NIR^6}$ & \tabincell{l}{$\bigO{NIR^4}$\\$+\bigO{NR^6}$} & \tabincell{l}{$\bigO{N(N-2)R^6}$\\$+\bigO{NR^{2N+1}}$} \\
		\bottomrule
	\end{tabular}
	}
\end{table}	
From \Cref{tab:complexity}, we can find the following results.
\begin{itemize}
	\item \textbf{TR-ALS and TR-ALS-NE.}  
	The domain cost of TR-ALS is $\bigO{NI^{N-1}R^4+NI^NR^2}$, which appears in solving the ALS subproblems. Accordingly, the domain cost of TR-ALS-NE appears in computing MTTSP, which costs $\bigO{NI^NR^2}$. They have the same leading order, but TR-ALS-NE is much closer to twice as fast as TR-ALS. This is mainly because we use a new method to compute the coefficient matrices of the normal equations for the ALS subproblems.
	
	\item \textbf{TR-ALS-NE and TR-ALS-QR.} 
	The domain cost of TR-ALS-QR, appearing in computing Multi-TTM, is $$\bigO{NR^2I^N(1+\frac{R^2}{I}+ \frac{R^4}{I^2} + \cdots + \frac{R^{2N-4}}{I^{N-2}})},$$ which is larger than $\bigO{NI^NR^2}$ for TR-ALS-NE, especially when $R^2$ is significantly large. In addition, for the aforementioned case, the complexity in the `Others' part for TR-ALS-QR is not less than that for TR-ALS-NE either.  
	Therefore, a wrap-up is that TR-ALS-QR is slower than TR-ALS-NE. However, in the case of $I>R^2$, the two algorithms have the identical leading order computational complexity. Considering that TR-ALS-QR is more stable than TR-ALS-NE, the former may perform better in practice, especially for ill-conditioned problems.
	
	\item \textbf{TR-ALS-QR and TR-ALS-QRNE.} Compared to TR-ALS-QR, TR-ALS-QRNE mainly reduces the computational complexity in the `Others' part. So a wrap-up is that TR-ALS-QRNE is faster than TR-ALS-QR. 
\end{itemize}

\subsection{A by-product}
Based on the previous findings, we can obtain an expression of the inner product of tensors with TR format.
\begin{proposition}
	For $\tensor{A}, \tensor{B} \in \bb{R}^{I_1 \times \cdots \times I_N}$ 
	with TR decompositions being $\TR \left( \{\tensor{G}_n\}_{n=1}^N \right)$, where $\tensor{G}_n \in \bb{R}^{R_n \times I_n \times R_{n+1}}$, and $\TR \left( \{\tensor{Z}_n\}_{n=1}^N \right)$, where $\tensor{Z}_n \in \bb{R}^{S_n \times I_n \times S_{n+1}}$, respectively, 
	their \textbf{inner product} 
	can be expressed as
	\begin{equation*}
		\langle \tensor{A}, \tensor{B} \rangle = \Trace \left( (\tensor{P}_{N} \times_{2,4}^{1,3} \cdots \times_{2,4}^{1,3} \tensor{P}_{1})_{<2>} \right),
	\end{equation*}
	where $\tensor{P}_{n} = \sum_{i_n = 1}^{I_n} \mat{G}_{n}(i_n)^\intercal \circ \mat{Z}_{n}(i_n)^\intercal$ for $ n \in [N]$.
\end{proposition}

\begin{proof}
	According to Theorem 3.5 in \cite{zhao2016TensorRing} and the definition of inner product, we have
	\begin{align*}
		\langle \tensor{A}, \tensor{B} \rangle &= \langle \mat{A}_{[n]}, \mat{B}_{[n]} \rangle = \langle \mat{G}_{n(2)} (\mat{G}^{\ne n}_{[2]})^\intercal, \mat{Z}_{n(2)} (\mat{Z}^{\ne n}_{[2]})^\intercal \rangle \\ 
		&= \Trace \left( \mat{G}^{\ne n}_{[2]} \mat{G}_{n(2)}^\intercal \mat{Z}_{n(2)} (\mat{Z}^{\ne n}_{[2]})^\intercal \right) 
		= \Trace \left( \mat{G}_{n(2)}^\intercal \mat{Z}_{n(2)} (\mat{Z}^{\ne n}_{[2]})^\intercal \mat{G}^{\ne n}_{[2]} \right),
	\end{align*}
	where the last equation is from the property of trace.
	Using \Cref{prop:subchain_gram}, we know that 
	\begin{equation*}
	    (\mat{Z}^{\ne n}_{[2]})^\intercal \mat{G}^{\ne n}_{[2]} = \left( \tensor{P}_{n-1} \times_{2,4}^{1,3} \cdots \times_{2,4}^{1,3} \tensor{P}_{1} \times_{2,4}^{1,3} \tensor{P}_{N} \times_{2,4}^{1,3} \cdots \times_{2,4}^{1,3} \tensor{P}_{n+1} \right)_{<2>}.
	\end{equation*}
	On the other hand, we can get $\mat{G}_{n(2)}^\intercal \mat{Z}_{n(2)} = (\tensor{P}_n^\intercal)_{<2>}$ by carefully examining the difference between the classical mode-$n$ unfolding and the mode-$n$ unfolding, and a definition of transpose for $\tensor{P}_n$, i.e., $\tensor{P}_n^\intercal \xlongequal{def} \textsc{Permute}(\tensor{P}_n,[2,1,4,3])$. 
	Thus, combining these together and using the correspondence between elements, the desired result can be obtained.
\end{proof}

\begin{remark}
	Considering that  $\| \tensor{A} \|^2_F = \sqrt{\langle \tensor{A}, \tensor{A} \rangle}$, we can obtain the corresponding expression of the \textbf{Frobenius norm} of a tensor in TR representation. These expressions are equivalent to the counterparts in \cite{zhao2016TensorRing} but with a more concise format.
\end{remark}

\section{Numerical Experiments}
\label{sec:experiments}
To test our proposed methods, we choose TR-ALS as the main baseline. All experiments are run on Matlab R2020b on a computer with an Intel Xeon W-2255 3.7 GHz CPU and 256 GB memory. 
Additionally, we also use the MATLAB Tensor Toolbox \cite{kolda2006TensorToolbox}.

All the synthetic tensors have the same dimensions in all modes and they are generated by creating $N$ TR-cores of size $R_{true} \times I \times R_{true}$ firstly.  
Note that for the target rank of all algorithms, we denote it as $R$. These TR-cores may be generated in different ways, which will be detailed in subsequent experiments.
Then, we form the tensor by $\tensor{X}_{true} = \TR \left( \{\tensor{G}_n\}_{n=1}^N \right)$. 
Finally, the noise is added to obtain the observed tensor:  
\begin{equation*}
\tensor{X} = \tensor{X}_{true} + \eta \left( \frac{\| \tensor{X}_{true} \|_F}{\| \tensor{N} \|_F}\right) \tensor{N},
\end{equation*}
where the entries of $\tensor{N} \in \bb{R}^{I \times \cdots \times I}$ are drawn from a standard normal distribution and the parameter $\eta$ is the amount of noise.

As stated in the discussions of computational complexities, we use the random Gaussian tensors to initiate the TR-cores for all the related algorithms. 
For the termination criterion, all algorithms are terminated only after the maximum number of iterations being reached, and, unless otherwise stated, we set the maximum number to be 20. Then,  the running time and the relative errors via the formula 
\begin{equation*}
	\frac{\left\|\tensor{X} -\hat{\tensor{X}}\right\|_F}{\|  \tensor{X} \|_F}=\frac{\left\| \tensor{X}-\TR \left( \{\hat{\tensor{G}}_n\}_{n=1}^N \right)  \right\|_F}{\| \tensor{X} \|_F},
\end{equation*}
where the TR-cores $\hat{\tensor{G}}_n$ are computed by various algorithms, can be reported and compared. Again, unless otherwise stated, the numerical results are the averages over 10 runs.

Next, we consider the computation of the error $\|\tensor{X}- \hat{\tensor{X}} \|_F$. In our specific experiments, we compute it by forming the explicit representation of $\hat{\tensor{X}}$ using the TR-cores output by any algorithms. 
This is an accurate but less efficient method.  
We now introduce an approach to approximating the error, 
which exploits the identity $\| \tensor{X} - \hat{\tensor{X}} \|^2_F = \| \tensor{X} \|^2_F - 2 \langle \tensor{X}, \hat{\tensor{X}} \rangle + \| \hat{\tensor{X}} \|^2_F$ and computes $\langle \tensor{X}, \hat{\tensor{X}} \rangle$ and $\| \hat{\tensor{X}} \|_F$ cheaply by using the temporary quantities already computed by the ALS iterations. Note that $\| \tensor{X} \|_F$ is pre-computed and does not change over iterations. More specifically,
in the case of \textbf{TR-ALS-NE}, we have
\begin{equation*}
	\langle \tensor{X}, \hat{\tensor{X}} \rangle = \langle \mat{X}_{[N]}, \hat{\mat{G}}_{N(2)} (\hat{\mat{G}}^{\ne N}_{[2]})^\intercal \rangle = \langle \mat{X}_{[N]} \hat{\mat{G}}^{\ne N}_{[2]}, \hat{\mat{G}}_{N(2)} \rangle = \langle \mat{M}_{N}, \hat{\mat{G}}_{N(2)} \rangle,
\end{equation*}
where $\mat{M}_{N}$ is the result of the MTTSP computation in the mode $N$, i.e., the mode of the last subiteration.
Likewise, we have 
\begin{align*}
\| \hat{\tensor{X}} \|^2_F &= \langle \hat{\mat{G}}_{N(2)} (\hat{\mat{G}}^{\ne N}_{[2]})^\intercal, \hat{\mat{G}}_{N(2)} (\hat{\mat{G}}^{\ne N}_{[2]})^\intercal \rangle \\
&= \langle (\hat{\mat{G}}^{\ne N}_{[2]})^\intercal \hat{\mat{G}}^{\ne N}_{[2]}, (\hat{\mat{G}}_{N(2)})^\intercal \hat{\mat{G}}_{N(2)} \rangle = \langle \mat{S}_{n<2>}, (\tensor{P}_N^\intercal)_{<2>} \rangle,
\end{align*}
where $\tensor{P}_N^\intercal =\textsc{Permute}(\tensor{P}_N,[2,1,4,3])$. 
Similarly, in the case of \textbf{TR-ALS-QR}, we have 
\begin{equation*}
\langle \tensor{X}, \hat{\tensor{X}} \rangle = \langle \mat{W}_{N}, \hat{\mat{G}}_{N(2)} \mat{R}_{[2]}^\intercal \rangle ~~\text{and}~~\| \hat{\tensor{X}} \|^2_F  = \langle \mat{R}_{[2]}^\intercal \mat{R}_{[2]}, \mat{R}_{N(2)}^\intercal \mat{R}_{N(2)} \rangle,
\end{equation*}
and in the case of \textbf{TR-ALS-QRNE}, we have \begin{equation*}
\langle \tensor{X}, \hat{\tensor{X}} \rangle = \langle \mat{Y}_{N} \mat{V}_{N[2]}, \hat{\mat{G}}_{N(2)} \rangle ~~\text{and}~~\| \hat{\tensor{X}} \|^2_F = \langle  \mat{V}_{N[2]}^\intercal \mat{V}_{N[2]}, \mat{R}_{N(2)}^\intercal \mat{R}_{N(2)} \rangle.
\end{equation*}
The main reason why we don't employ the above method in our experiments is that that part of time is not what we want to focus on for the comparison of various algorithms. Our main concern is the valid differences of methods. Meanwhile, we need the exact error more in the experiments for ill-conditioned datasets. 


\subsection{Efficiency of TR-ALS-NE}
\label{ssec:syn-A}
We use two experiments to test the speedup of TR-ALS-NE over TR-ALS. 
\paragraph{Experiment A-\uppercase\expandafter{\romannumeral1}.}
Our first experiment ignores the convergence of TR-ALS and TR-ALS-NE and merely compares the computational time for each iteration of these two methods. 
 
We consider 3rd- and 5th-order tensors of various sizes without noise. 
Each TR-core is generated by a random Gaussian tensor with entries drawn independently from a standard normal distribution. \Cref{fig:a1_per_iter} shows how much cheaper each iteration of the ALS is when using TR-ALS-NE. Obviously, TR-ALS-NE runs faster than TR-ALS in both the 3rd- and 5th-order tensors. 

\begin{figure}[htbp] 
	\centering 
	\subfloat[3rd-order tensor, $R_{true}$ = $R$ = 10]{\includegraphics[scale=0.32]{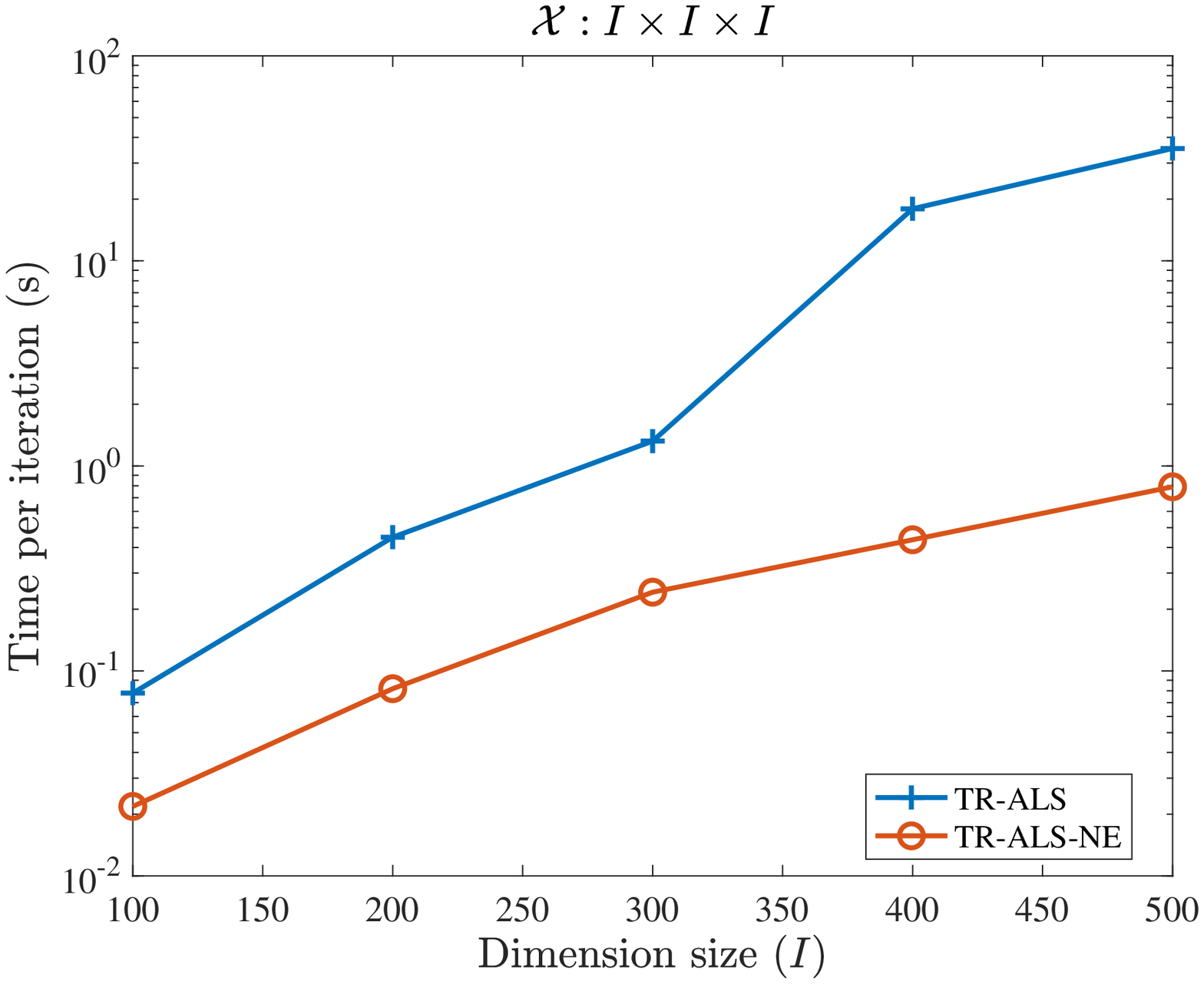}} 
	\subfloat[5th-order tensor, $R_{true}$ = $R$ = 4]{\includegraphics[scale=0.32]{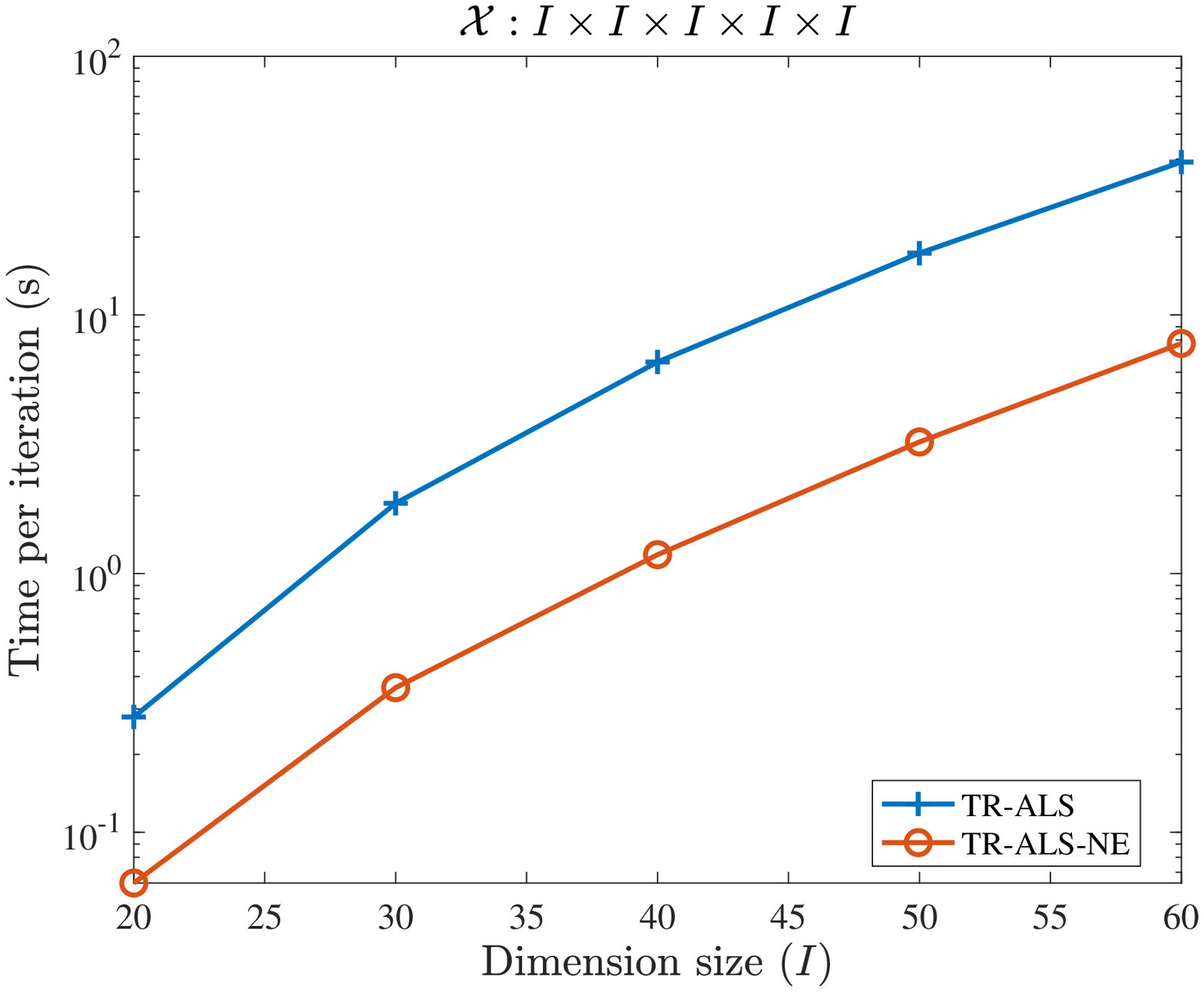}} 
	\caption{Mean time per iteration of TR-ALS and TR-ALS-NE for 3rd- and 5th-order tensors. Each dot represents the mean iteration time over 20 iterations (no checks for convergence).}
	\label{fig:a1_per_iter}
\end{figure}

\paragraph{Experiment A-\uppercase\expandafter{\romannumeral2}.}
Our second experiment considers both the convergence and computational cost of TR-ALS-NE. We still use the data in Experiment A-\uppercase\expandafter{\romannumeral1}. \Cref{fig:a2_conv} shows the numerical results on decreasing trend of the relative errors as the number of iterations and time increase. We can see that TR-ALS-NE can achieve almost the same convergence errors as TR-ALS but with much less computing time. And, the time gap becomes more obvious as the tensor order, i.e., $N$, and the dimensionality, i.e., $I$, increase. This means that our method is more effective for large-scale data.

\begin{figure}[htbp] 
	\centering 
	\subfloat[$\tensor{X}: 300 \times 300 \times 300$, $R_{true}$ = $R$ = 10]{\includegraphics[scale=0.15]{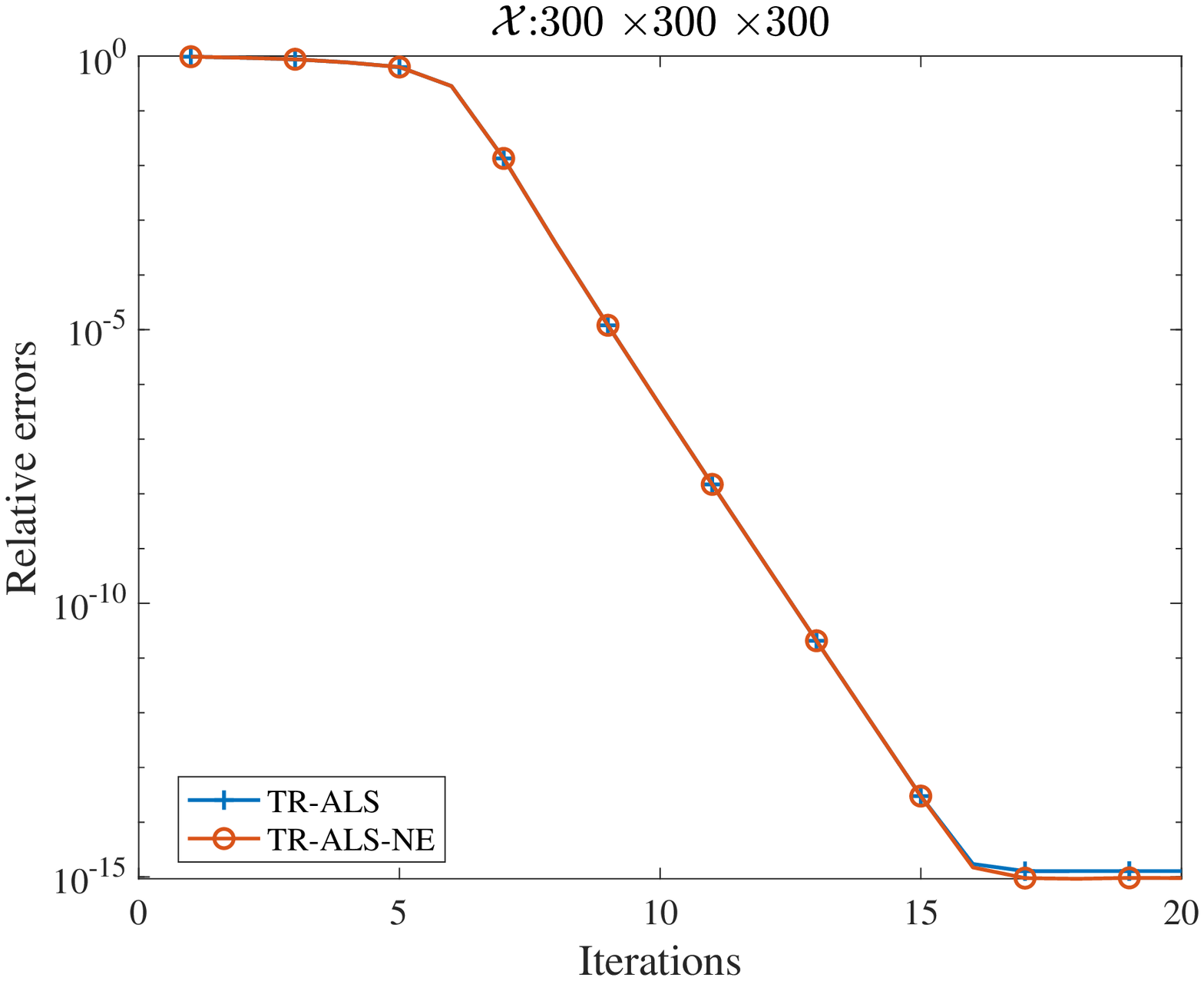}} 
	\subfloat[$\tensor{X}: 300 \times 300 \times 300$, $R_{true}$ = $R$ = 10]{\includegraphics[scale=0.15]{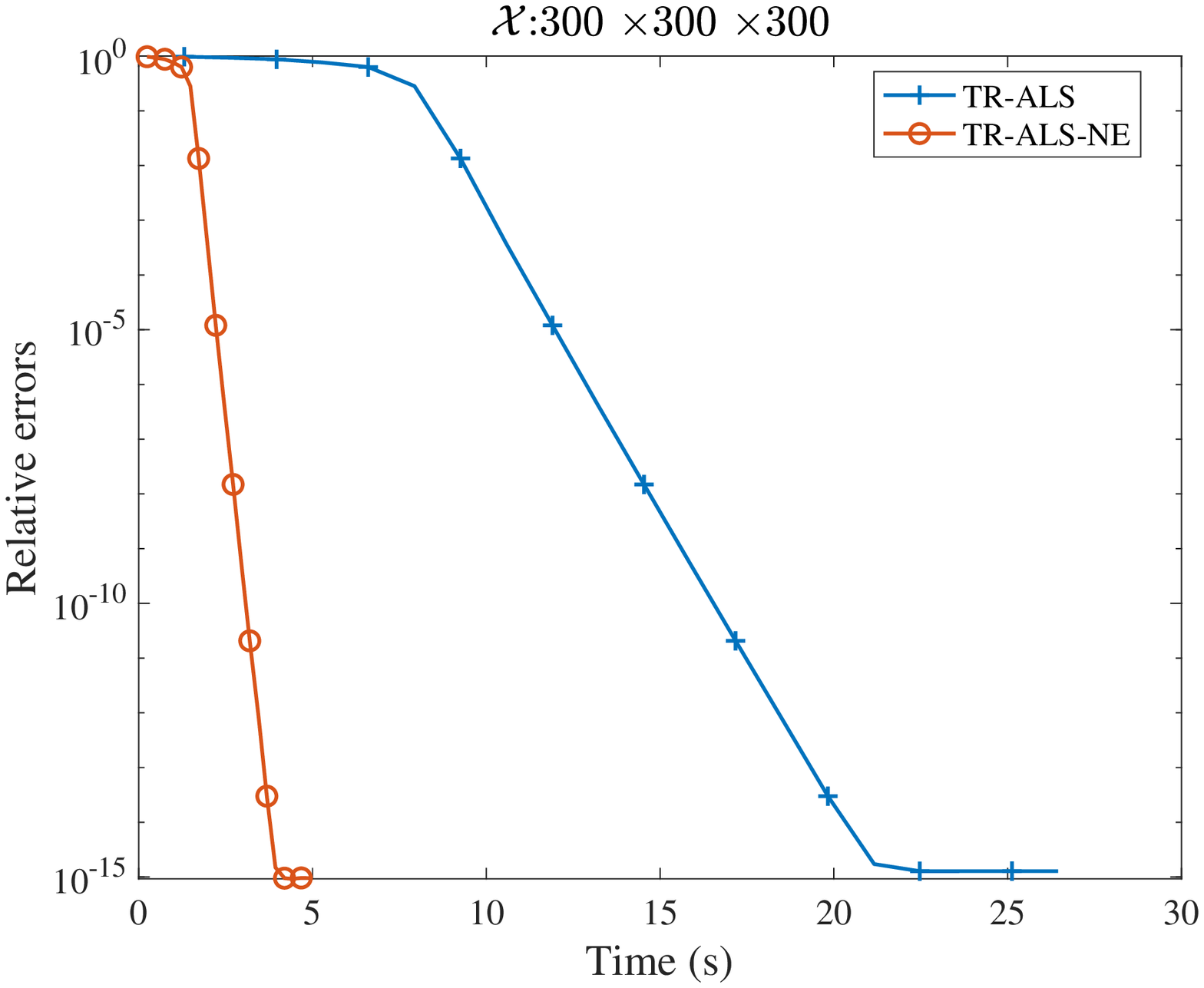}} 
	\subfloat[$\tensor{X}: 40 \times 40 \times 40 \times 40 \times 40$,  $R_{true}$ = $R$ = 5]{\includegraphics[scale=0.15]{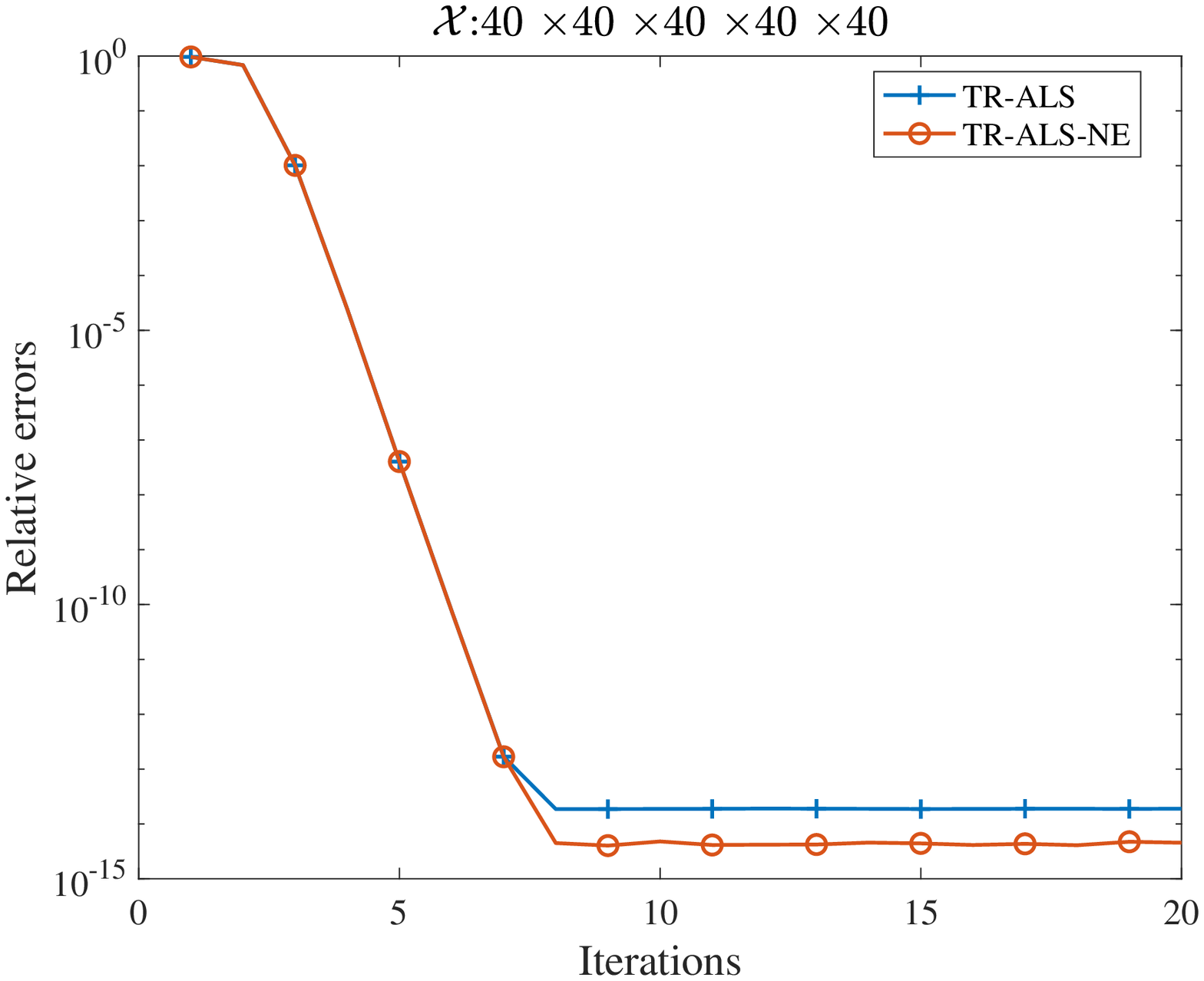}} 
	\subfloat[$\tensor{X}: 40 \times 40 \times 40 \times 40 \times 40$, $R_{true}$ = $R$ = 5]{\includegraphics[scale=0.15]{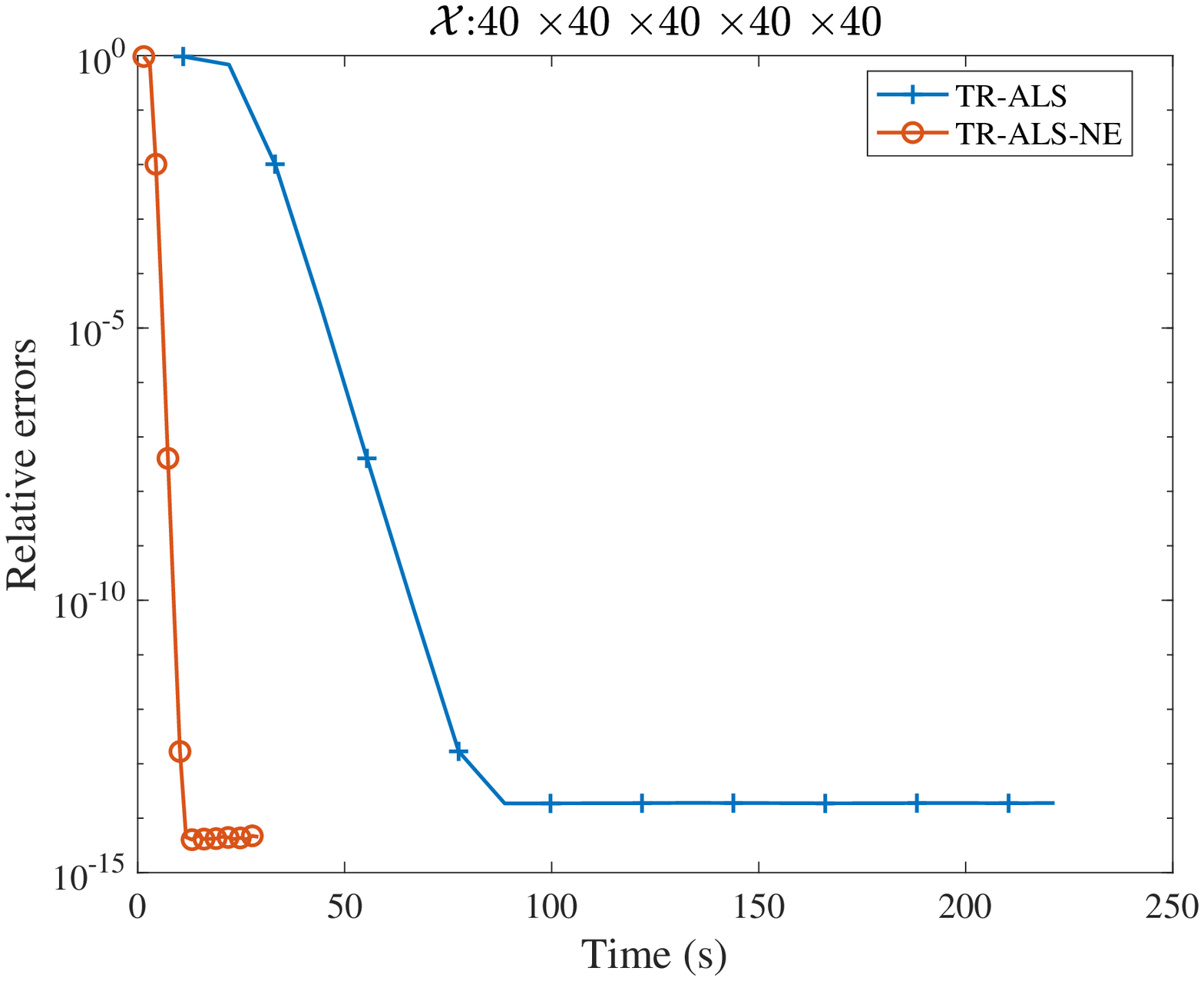}} 
	\quad
	\subfloat[$\tensor{X}: 500 \times 500 \times 500$, $R_{true}$ = $R$ = 10]{\includegraphics[scale=0.15]{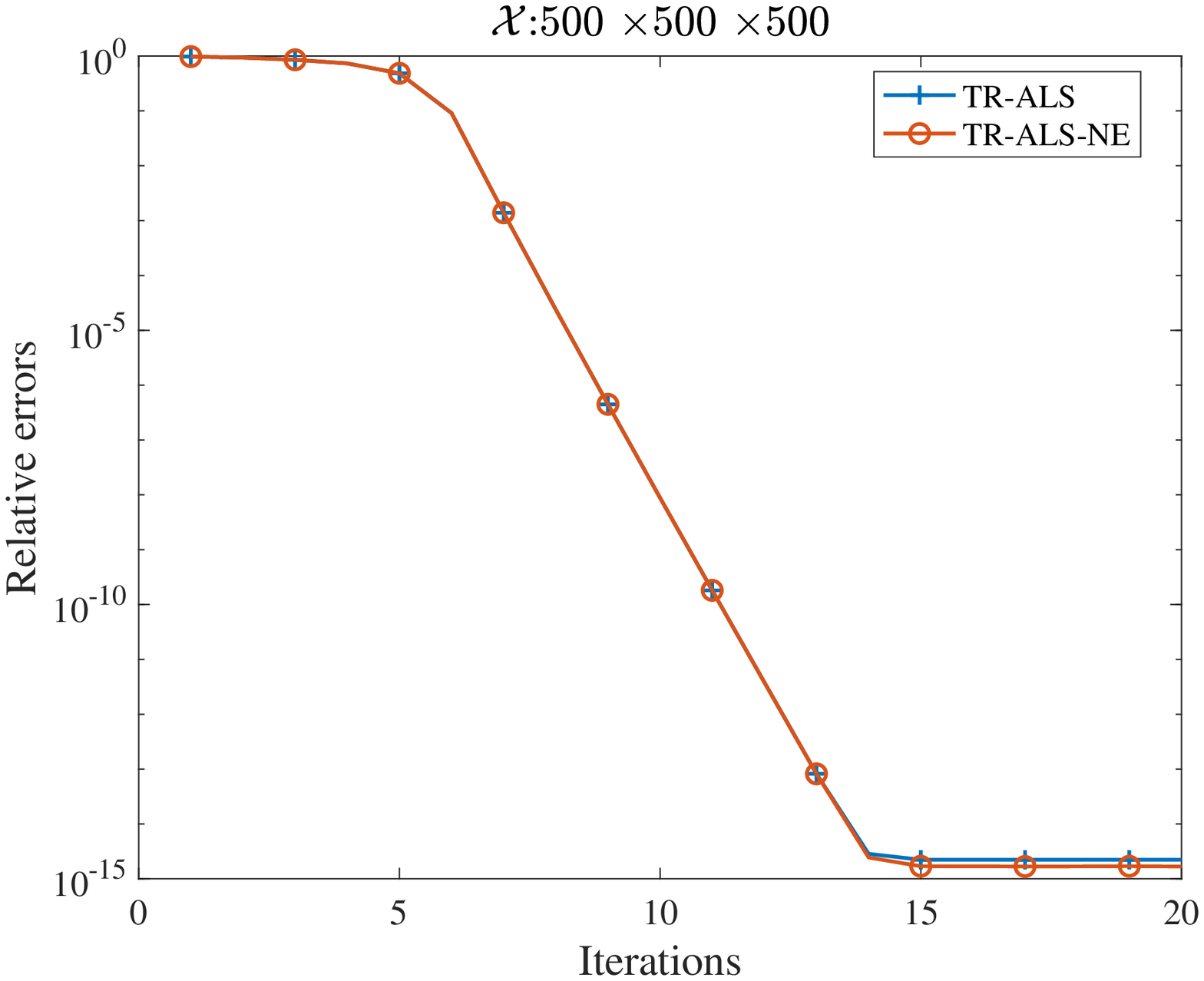}} 
	\subfloat[$\tensor{X}: 500 \times 500 \times 500$, $R_{true}$ = $R$ = 10]{\includegraphics[scale=0.15]{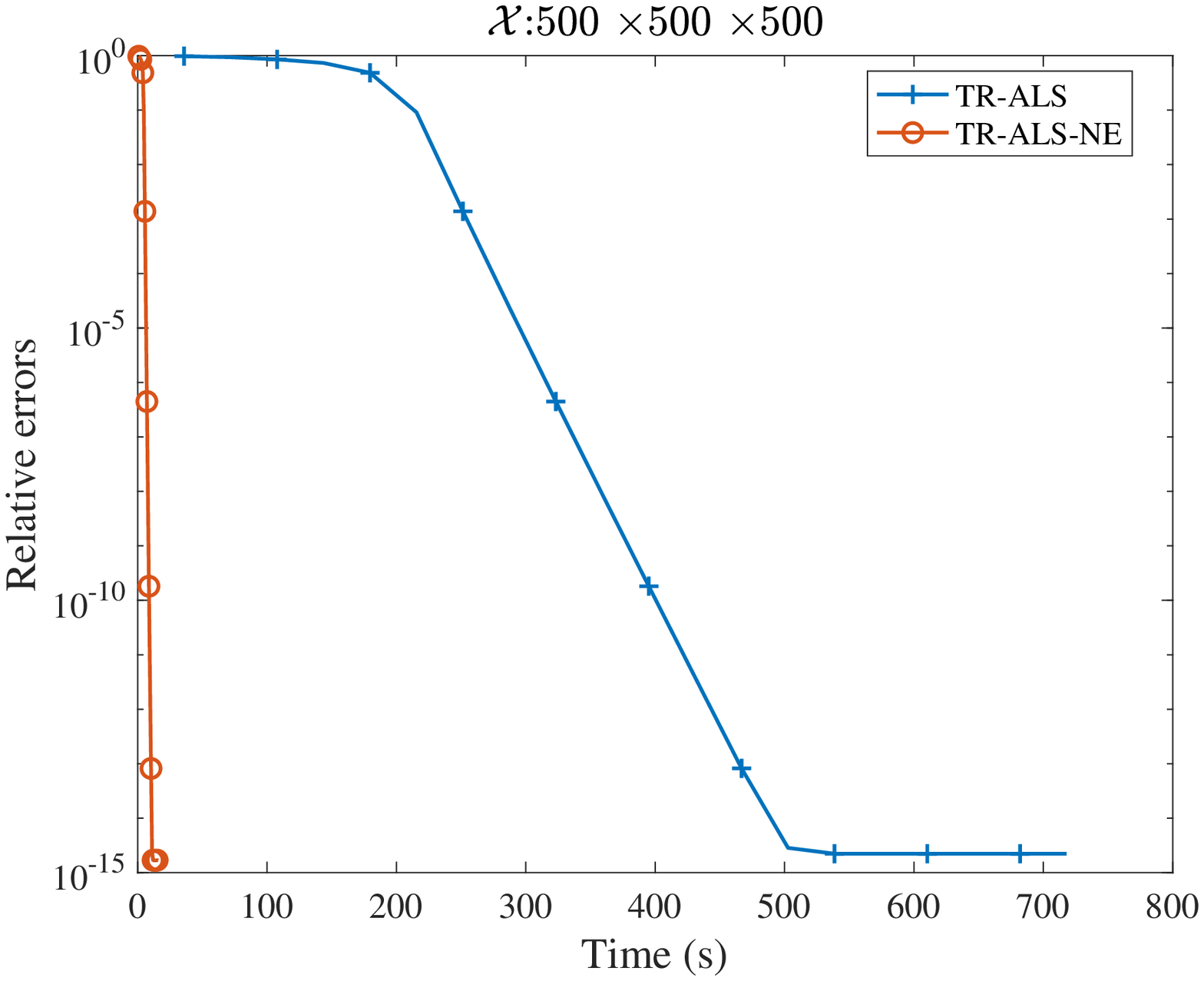}} 
	\subfloat[$\tensor{X}: 60 \times 60 \times 60 \times 60 \times 60$, $R_{true}$ = $R$ = 5]{\includegraphics[scale=0.15]{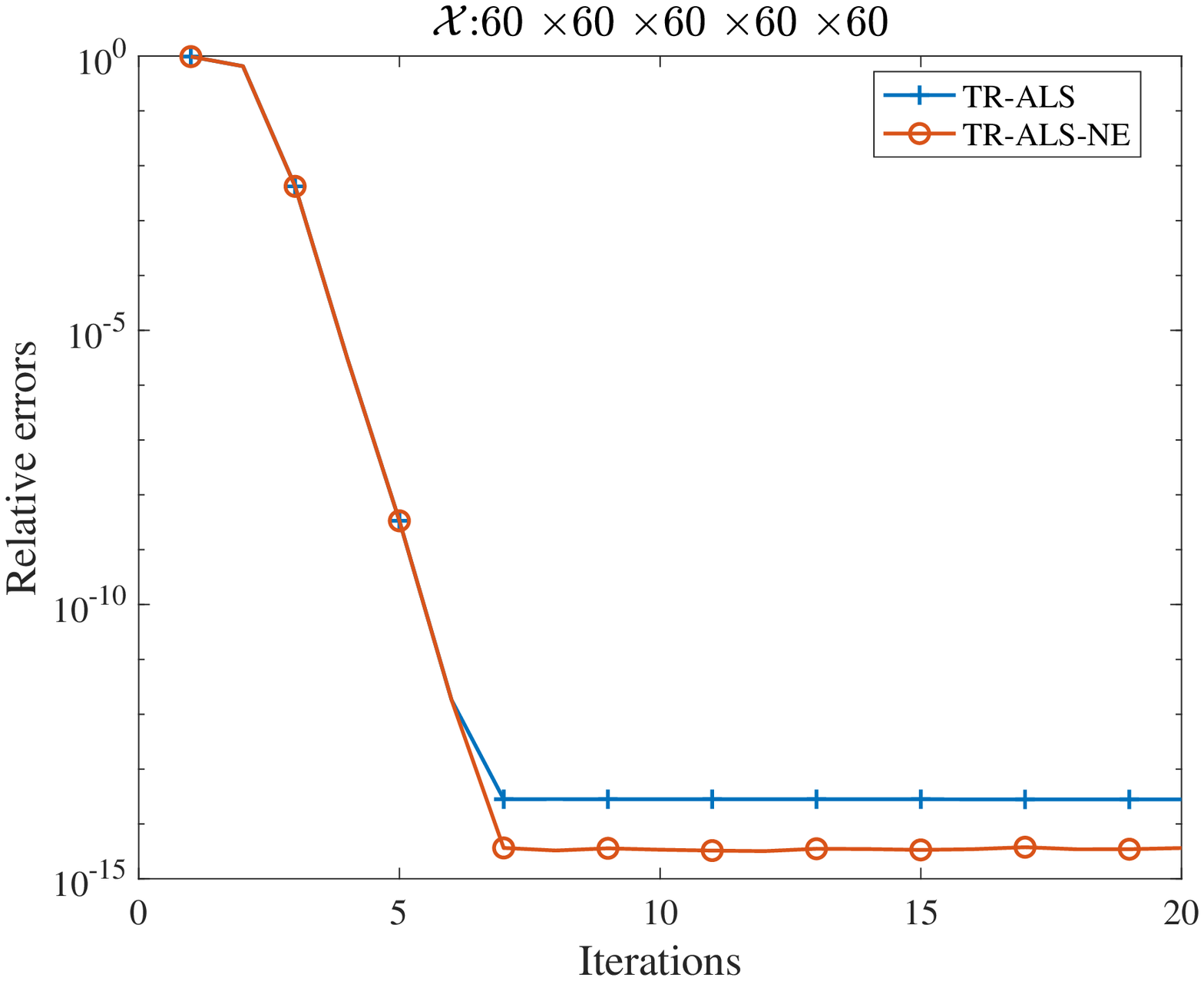}} 
	\subfloat[$\tensor{X}: 60 \times 60 \times 60 \times 60 \times 60$, $R_{true}$ = $R$ = 5]{\includegraphics[scale=0.15]{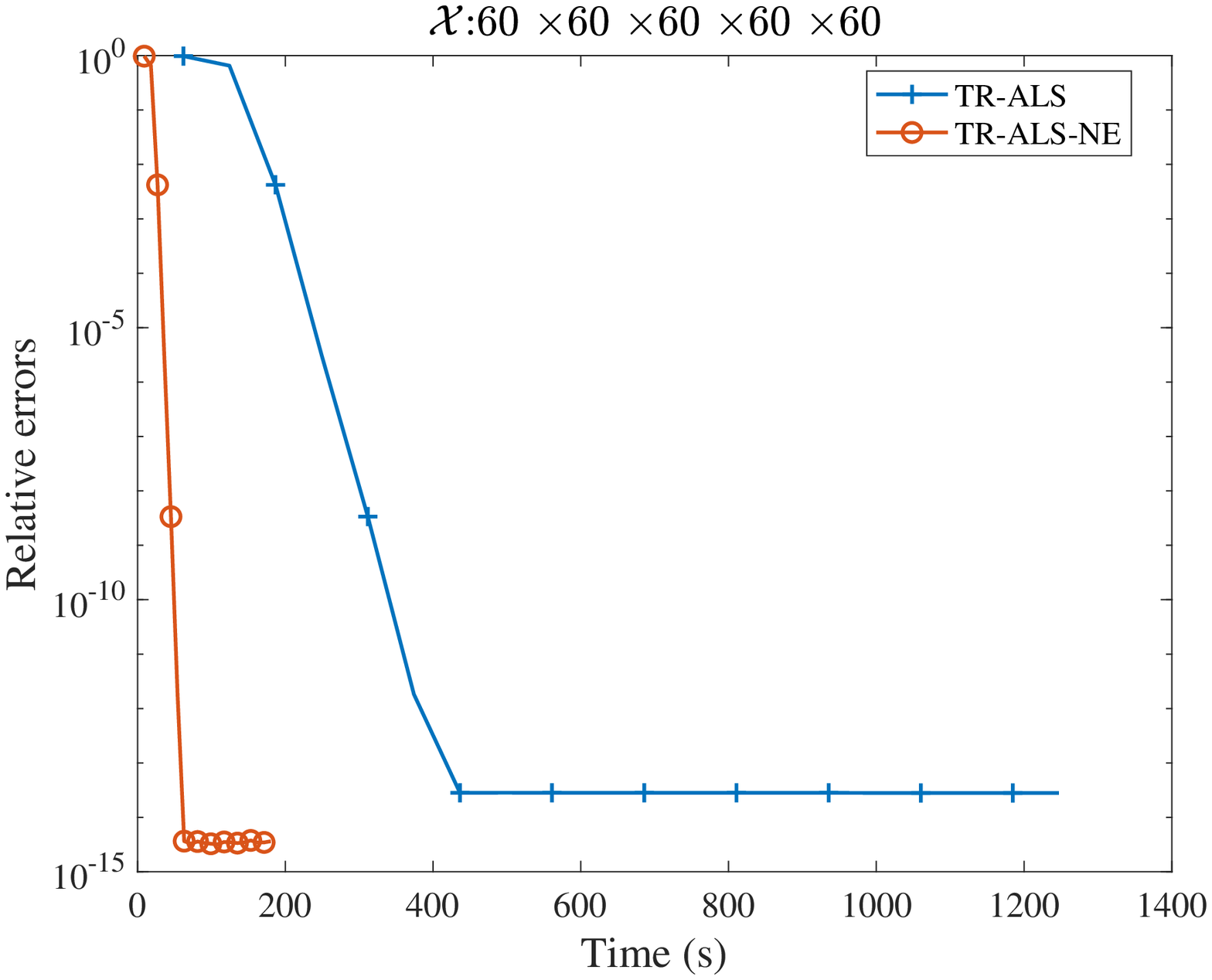}} 
	\caption{Number of iterations v.s. Relative errors and Time v.s. Relative errors output by algorithms for 3rd- and 5th-order tensors.}
	\label{fig:a2_conv}
\end{figure}

\subsection{Stability of TR-ALS-QR and TR-ALS-QRNE}
\label{ssec:syn-B}
Three experiments on different datasets are presented to show the stronger stability of TR-ALS-QR and TR-ALS-QRNE compared with TR-ALS-NE.

\paragraph{Experiment B-\uppercase\expandafter{\romannumeral1}.}
Our first experiment 
is run on the data used in Experiment A-\uppercase\expandafter{\romannumeral1}, which is well-conditioned.
The numerical results 
are presented in \Cref{fig:b1_conv}, from which we can see that, in this case, the above three methods have similar performance in accuracy but a little difference in running time. More detailedly, for the two lower-order tensors, the two QR-based algorithms perform a little better, but for the 40-dimensional 5th-order tensor, TR-ALS-NE is a little faster. When switching to the 60-dimensional 5th-order tensor, the performance of all the methods is almost the same. This is mainly because, for well-conditioned data, the advantage of the stability of QR-based methods is not remarkable. In addition, for TR-ALS-QRNE and TR-ALS-QR, the former always runs a little faster than the latter as expected. 


\begin{figure}[htbp] 
	\centering 
	\subfloat[$\tensor{X}: 300 \times 300 \times 300$, $R_{true}$ = $R$ = 10]{\includegraphics[scale=0.15]{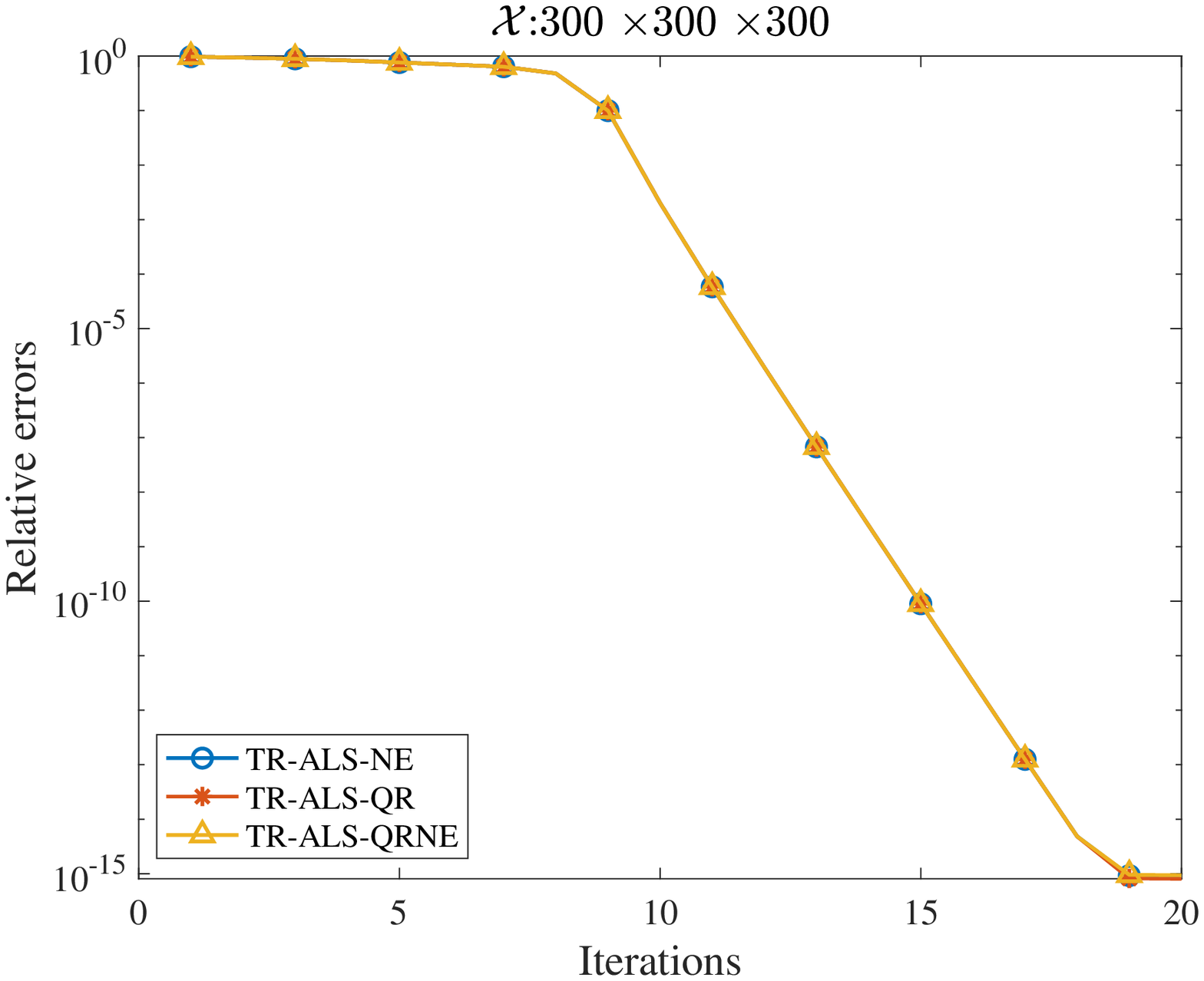}} 
	\subfloat[$\tensor{X}: 300 \times 300 \times 300$, $R_{true}$ = $R$ = 10]{\includegraphics[scale=0.15]{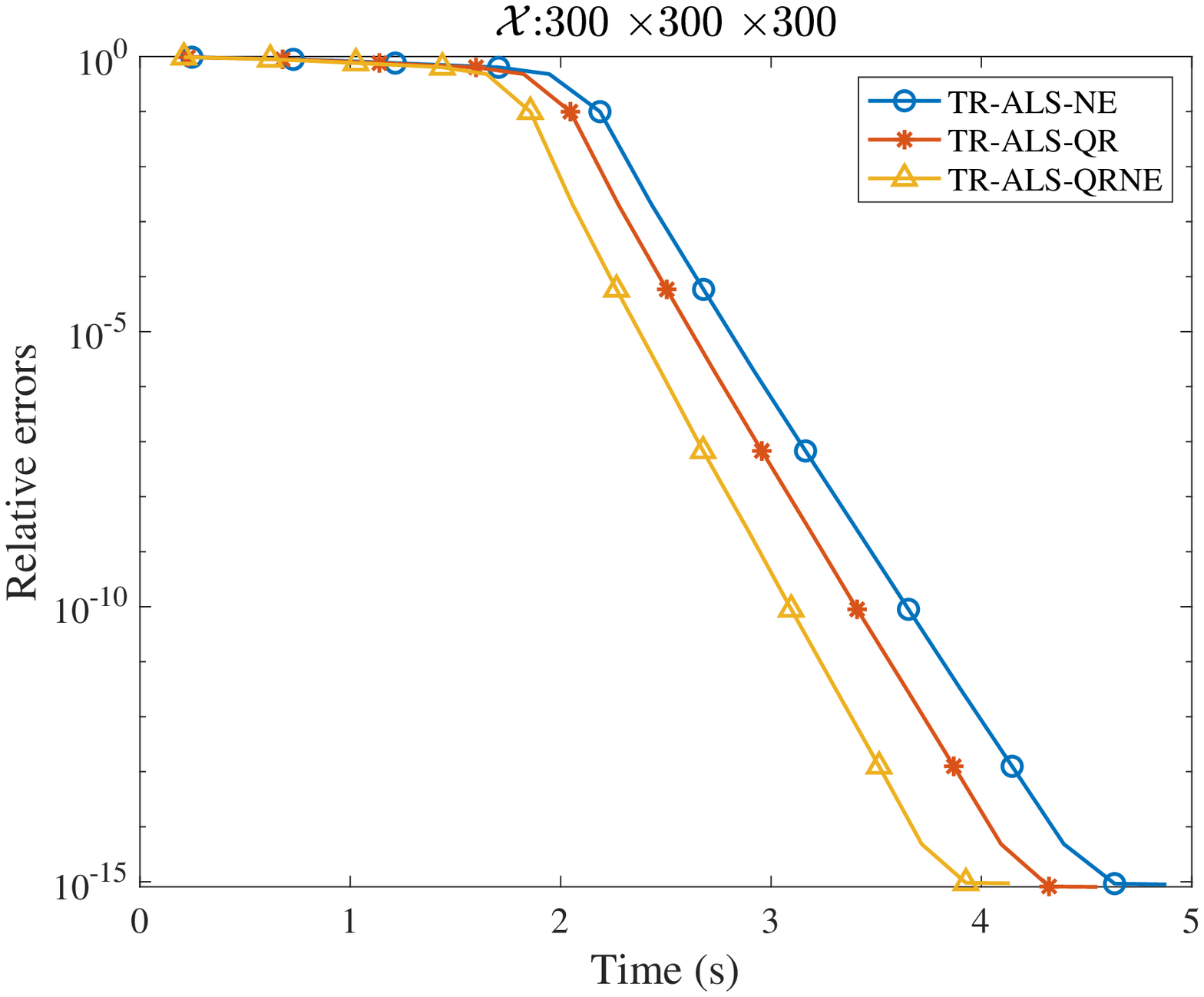}} 
	\subfloat[$\tensor{X}: 40 \times 40 \times 40 \times 40 \times 40$, $R_{true}$ = $R$ = 5]{\includegraphics[scale=0.15]{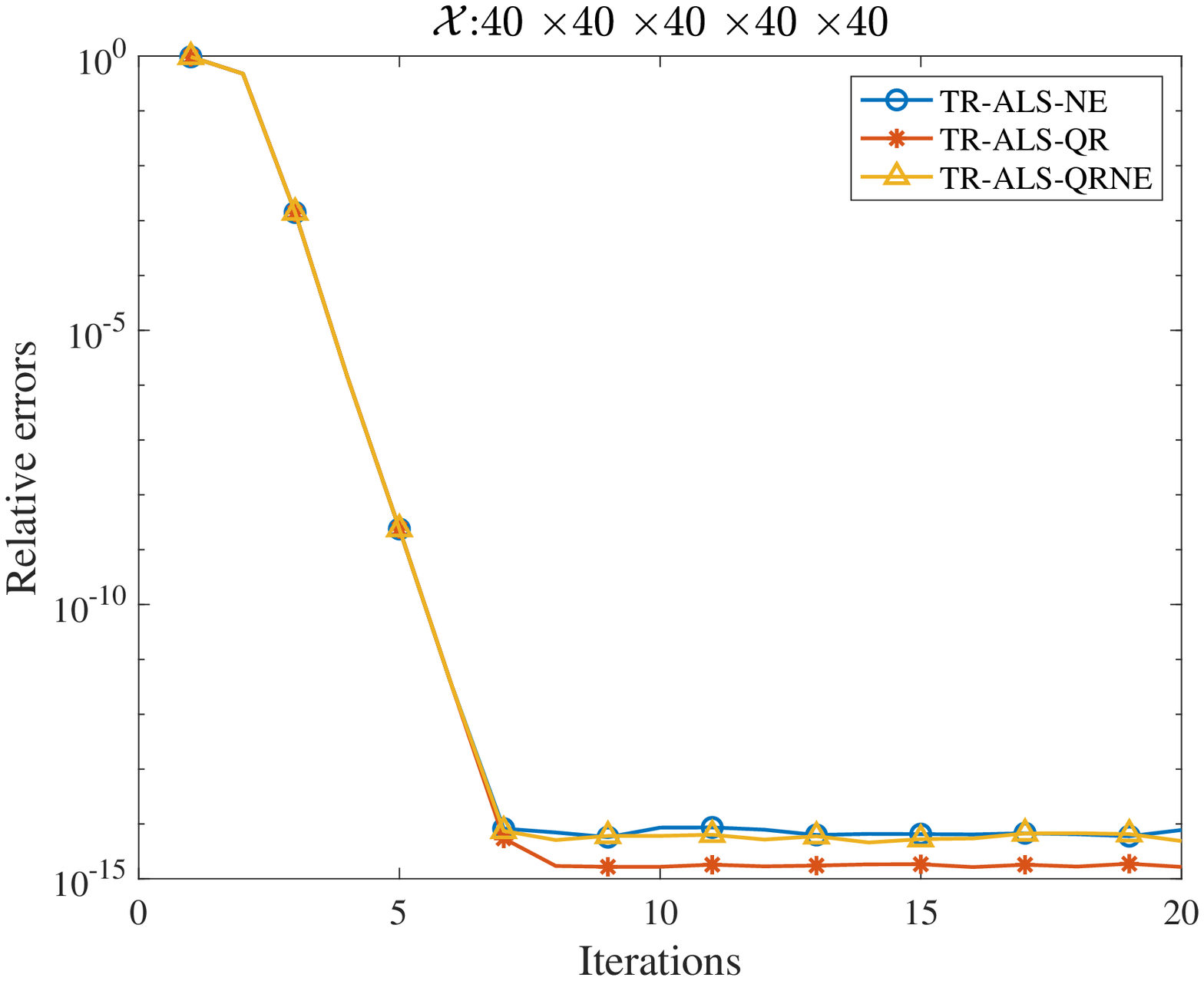}} 
	\subfloat[$\tensor{X}: 40 \times 40 \times 40 \times 40 \times 40$, $R_{true}$ = $R$ = 5]{\includegraphics[scale=0.15]{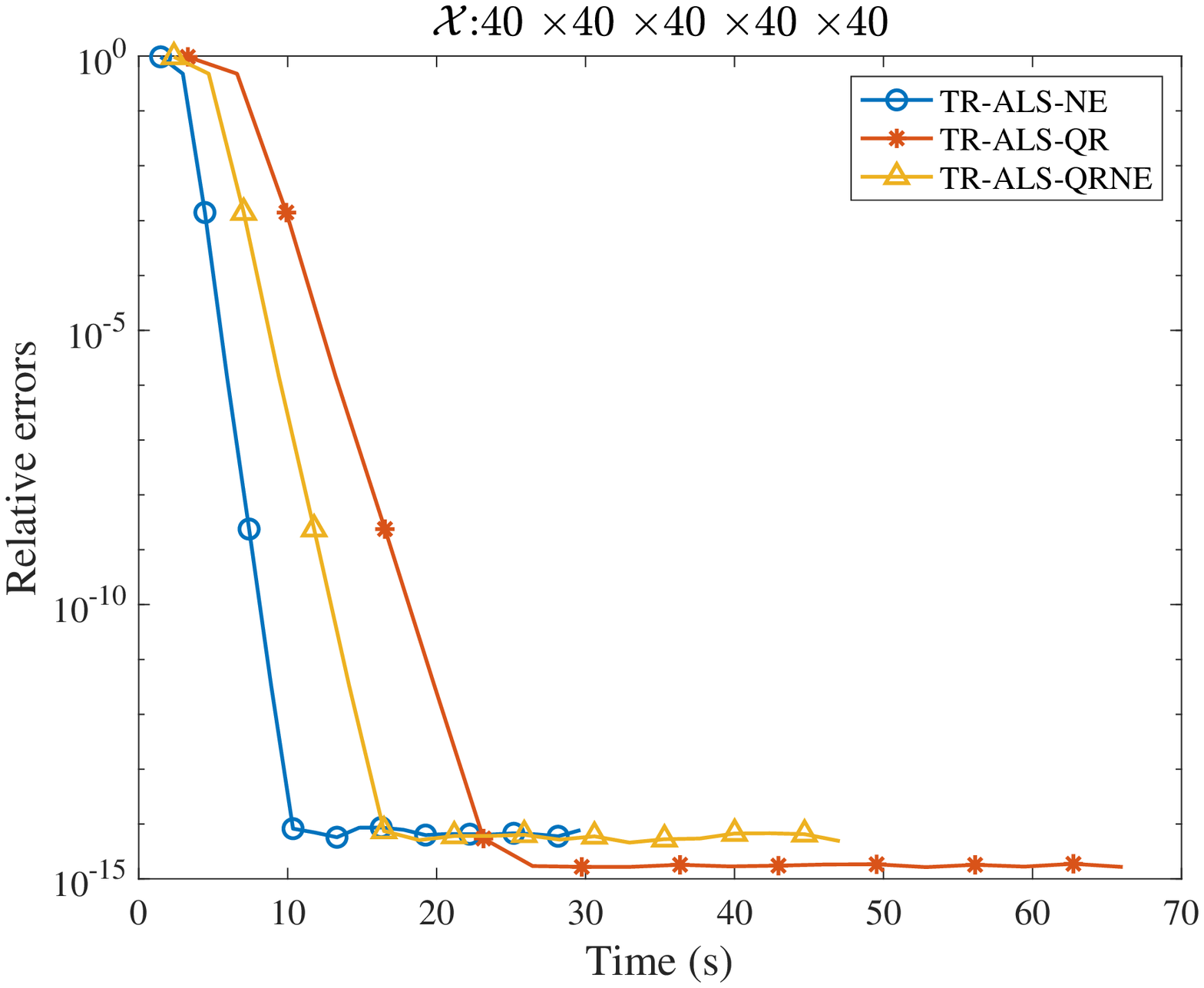}} 
	\quad
	\subfloat[$\tensor{X}: 500 \times 500 \times 500$, $R_{true}$ = $R$ = 10]{\includegraphics[scale=0.15]{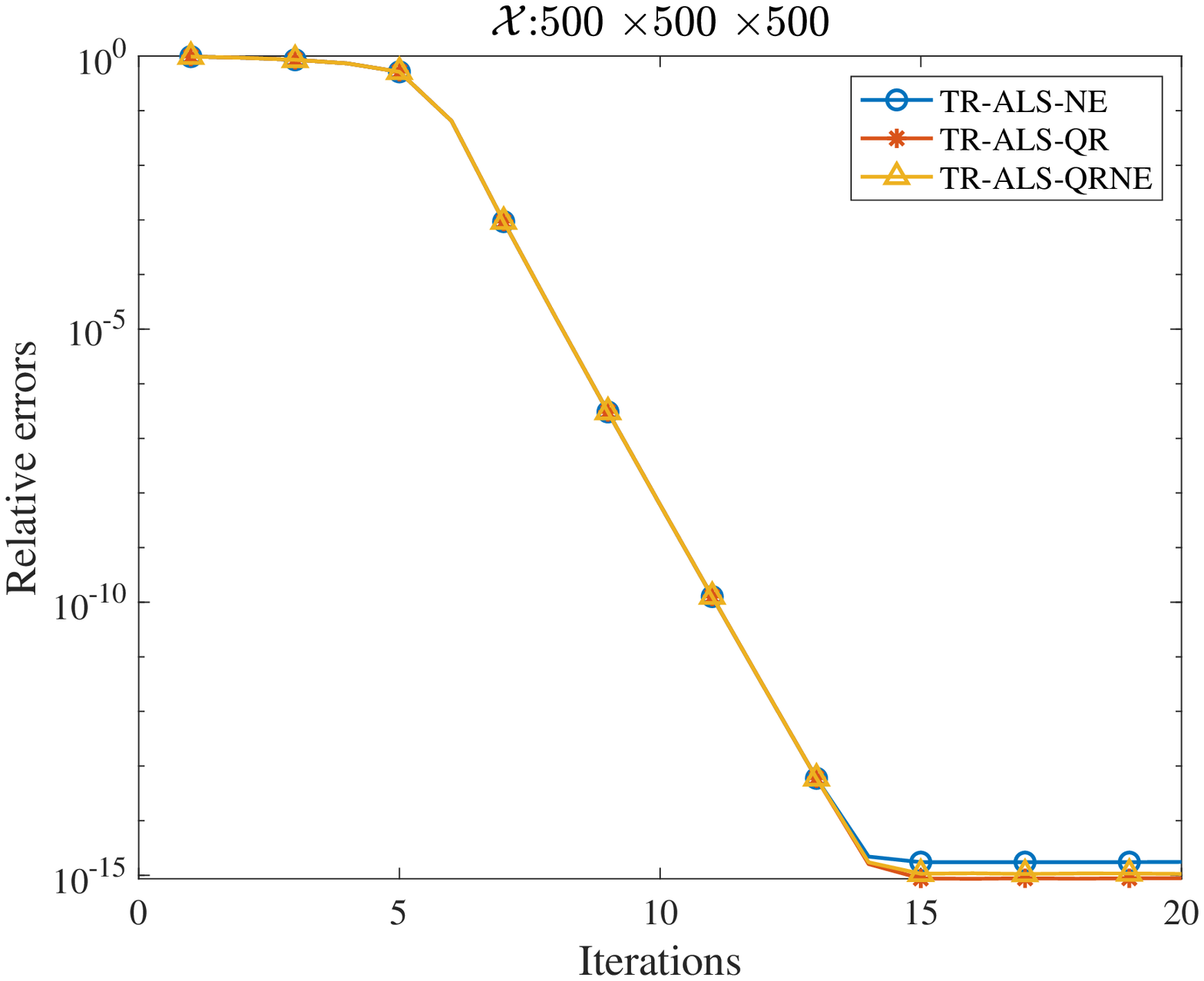}} 
	\subfloat[$\tensor{X}: 500 \times 500 \times 500$, $R_{true}$ = $R$ = 10]{\includegraphics[scale=0.15]{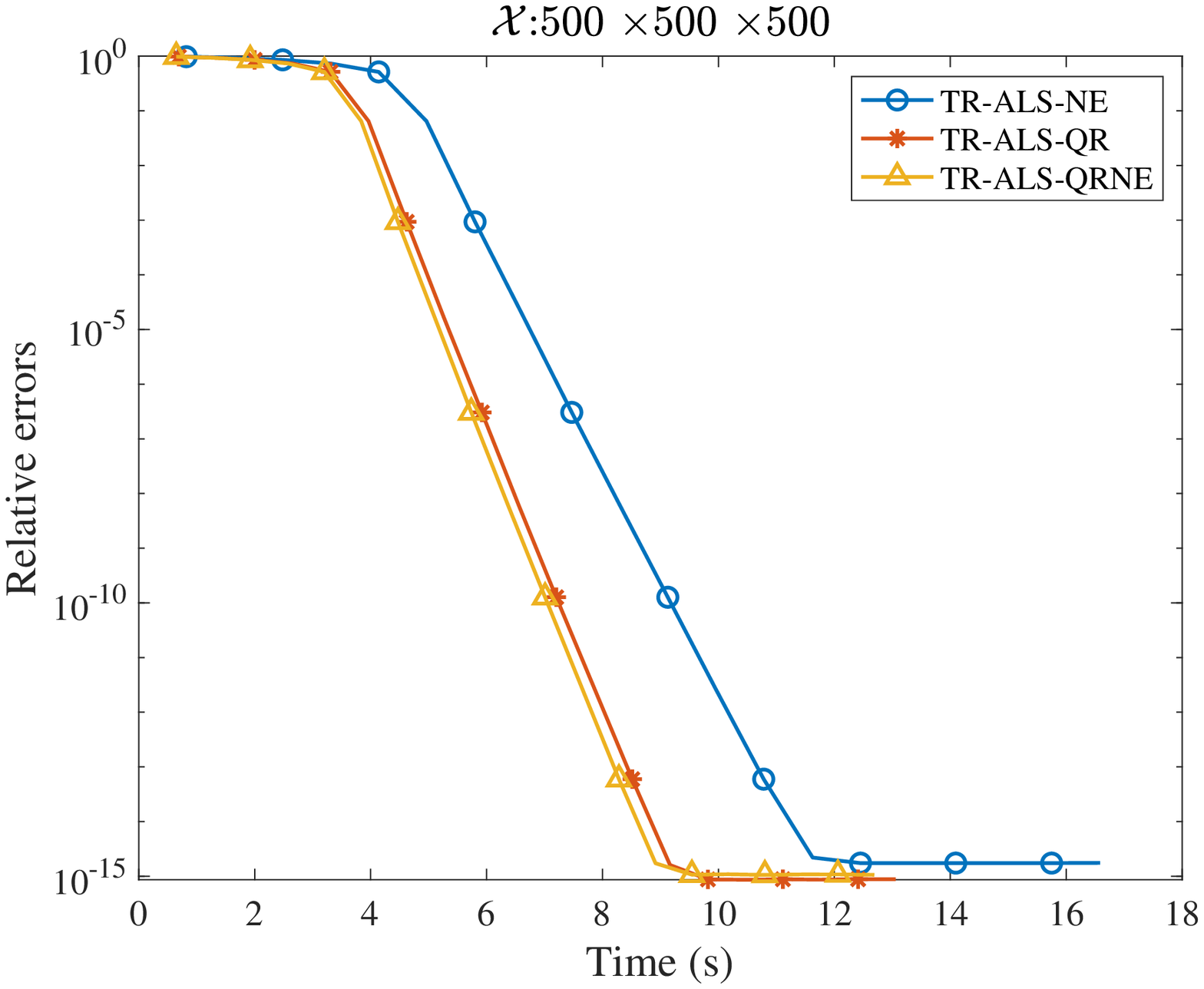}} 
	\subfloat[$\tensor{X}: 60 \times 60 \times 60 \times 60 \times 60$, $R_{true}$ = $R$ = 5]{\includegraphics[scale=0.15]{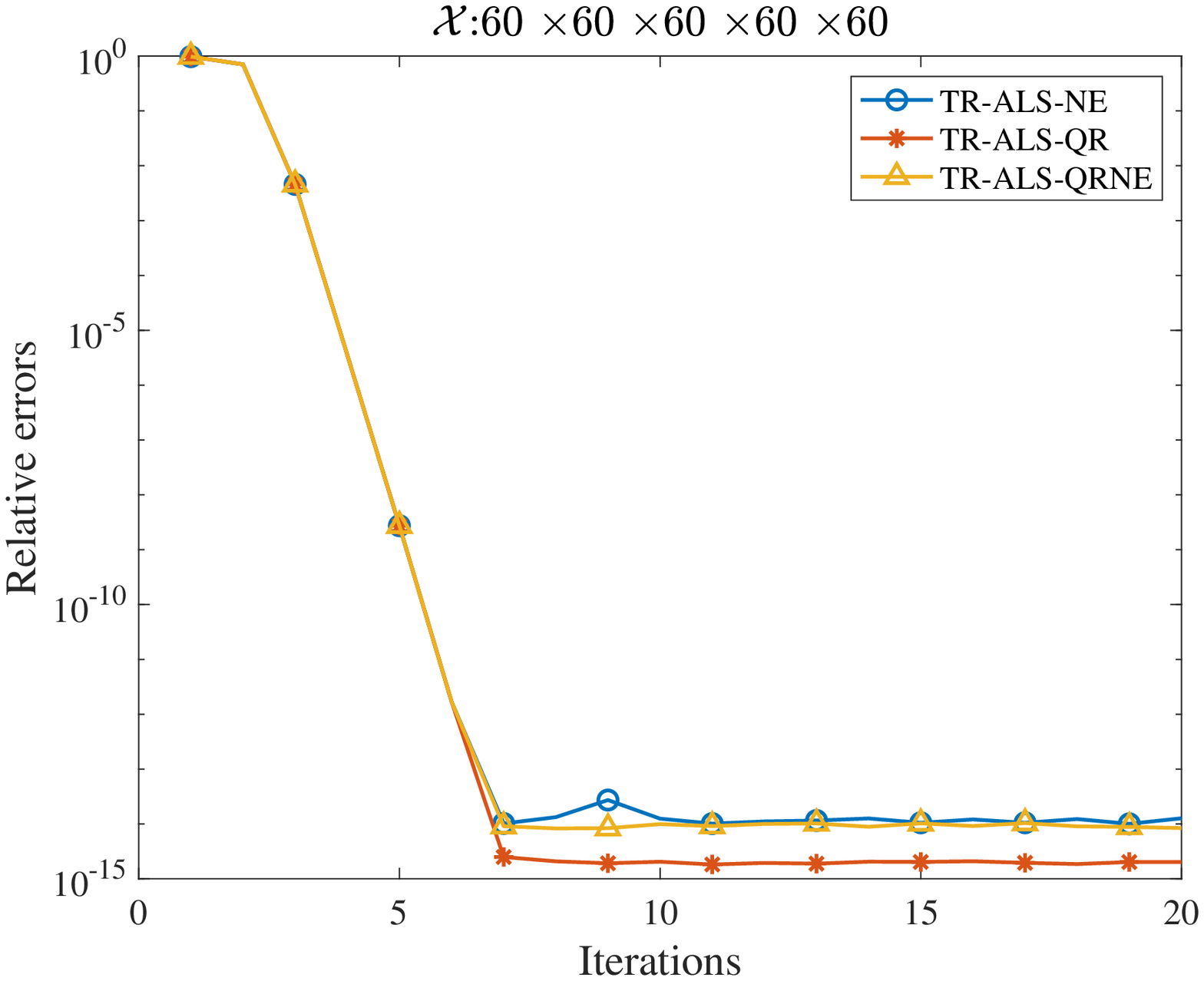}} 
	\subfloat[$\tensor{X}: 60 \times 60 \times 60 \times 60 \times 60$, $R_{true}$ = $R$ = 5]{\includegraphics[scale=0.15]{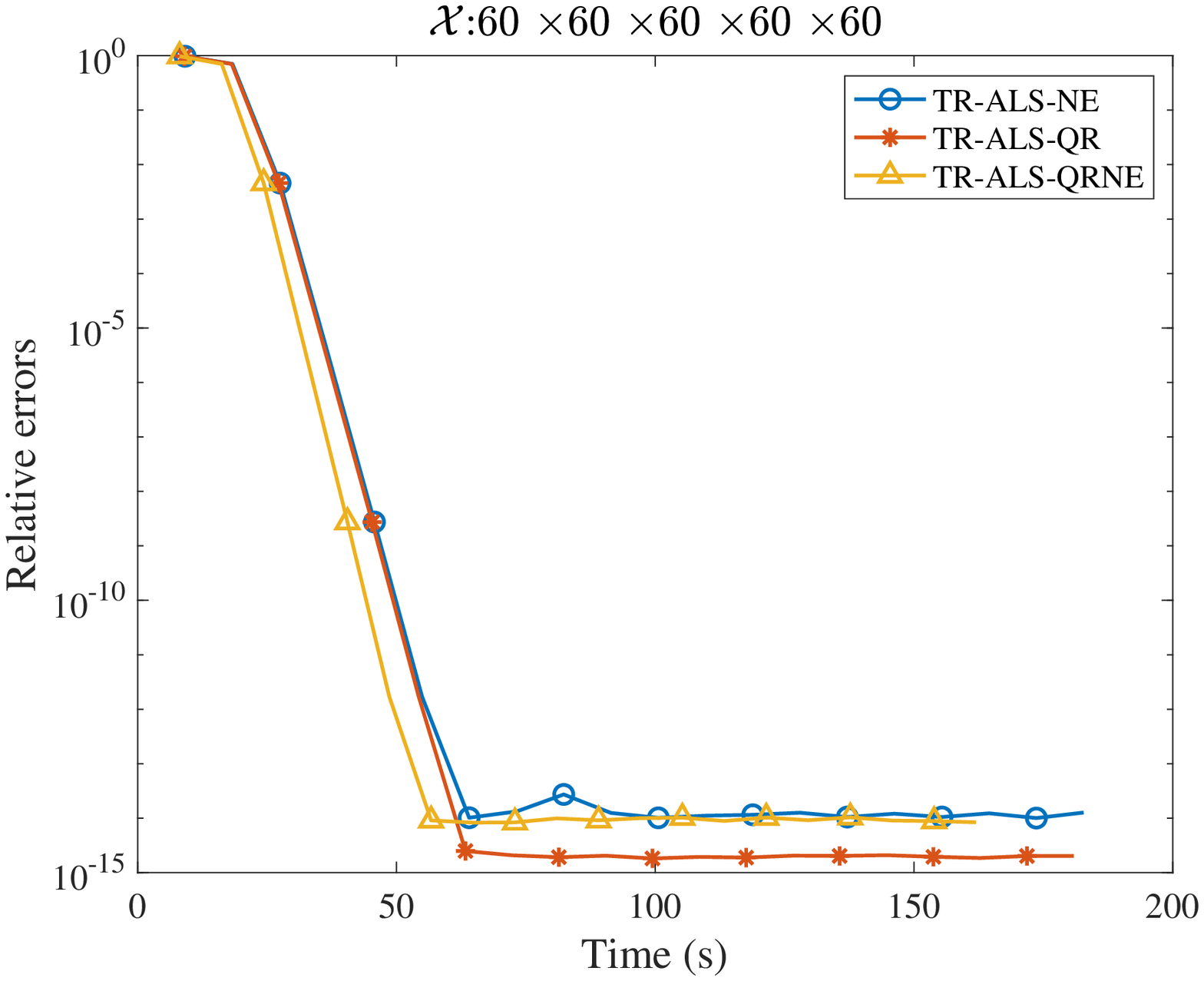}} 
	\caption{Number of iterations v.s. Relative errors and Time v.s. Relative errors output by algorithms for 3rd- and 5th-order tensors.}
	\label{fig:b1_conv}
\end{figure}

\paragraph{Experiment B-\uppercase\expandafter{\romannumeral2}.}
In this experiment, we generate the TR-cores as done in \cite{tomasi2006ComparisonAlgorithms} to control their collinearity. Specifically, we first use
the Matlab function \textsc{Matrandcong}($100$, $25$, $\gamma$) in the MATLAB Tensor Toolbox \cite{kolda2006TensorToolbox}, which can create a matrix of size $100 \times 25$ such that its each column has norm 1 and any two columns have an inner product equal to $\gamma$, to generate three $100 \times 25$ matrices. 
Then, these matrices are reshaped to $5 \times 100 \times 5$ TR-cores. The parameter $\gamma$ mentioned above is used to control the congruence of matrices and hence the collinearity of TR-cores. Note that for the congruence 0.5, the level of collinearity is relatively low. Actually, even for the congruence 0.9, the underlying cores may be only mildly collinear. For further details on the method and the function, see \cite{tomasi2006ComparisonAlgorithms,kolda2006TensorToolbox}.

We test all combinations of three different noise levels $10^{-4}$, $10^{-7}$, and $10^{-10}$, and three different collinearity levels $1-10^{-4}$, $1-10^{-7}$, and $1-10^{-10}$. 
And, we run 100 trials of each algorithm for each configuration. 
The numerical results are presented in \Cref{fig:b2_collinear}, where we see that the combination of noise and collinearity affects the ill-conditioning of the problem in different ways, and the QR-based algorithms have better performance than TR-ALS-NE in terms of relative error. Some specific discussions are in order. 
\begin{itemize}
	\item[1)] For the case of the noise level being $10^{-4}$ and the collinearity level being $1-10^{-4}$, i.e., the data is not ill-conditioned, all the algorithms have similar performance as expected. 
	However, from the first row of the figure, we see that, for the fixed noise level $10^{-4}$, the volatility of the behavior of TR-ALS-NE is rising as the collinearity level increases.
	
	\item[2)] From the first column of the figure, i.e., the case of fixed collinearity level $1-10^{-4}$ and changing noise level, we see a much more accurate solution output by TR-ALS-QR and TR-ALS-QRNE as the noise level decreases. This is mainly because the high levels of Gaussian noise can alleviate the ill-conditioning.
	
	\item[3)] In the remaining 4 cases, i.e., the combinations of higher collinearity ($1-10^{-7}$ and $1-10^{-10}$) and low noise ($10^{-7}$ and $10^{-10}$), which lead to ill-conditioned subproblems, TR-ALS-QR and TR-ALS-QRNE are always robust. That is, they obtain the lowest relative errors and have stable performances in all scenarios. Whereas, TR-ALS-NE tends not to converge quickly and also suffers from higher forward or backward errors.
\end{itemize}
The above numerical findings are consistent with the explanations in \Cref{rem:tr-als-qr} and the theoretical discussions on computational complexities in \Cref{ssec:tr_als_qene}.

\begin{figure}[htbp] 
	\centering 
	\subfloat[$\eta = 10^{-4}$, $\gamma = 1-10^{-4}$]{\includegraphics[scale=0.207]{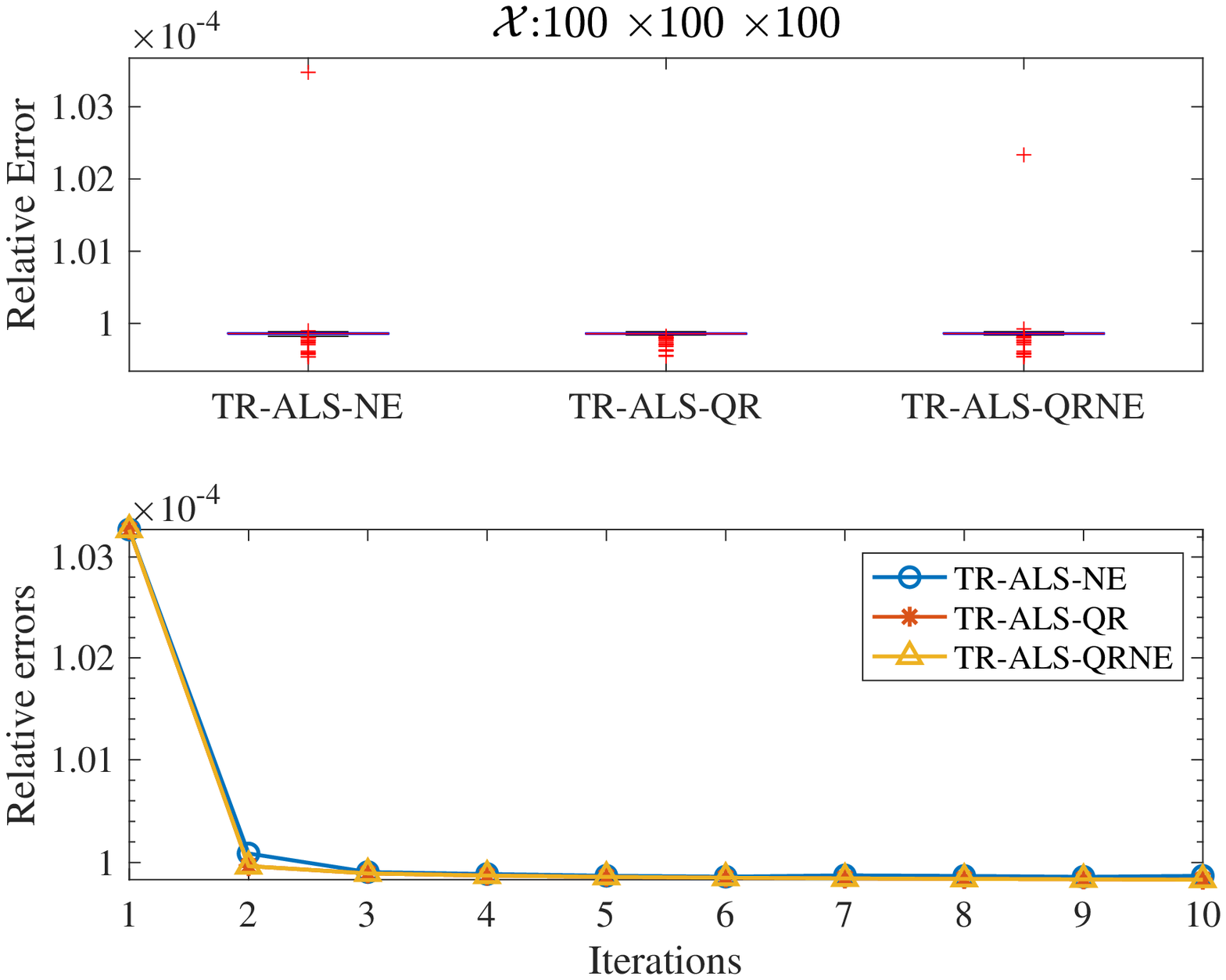}} 
	\subfloat[$\eta = 10^{-4}$, $\gamma = 1-10^{-7}$]{\includegraphics[scale=0.207]{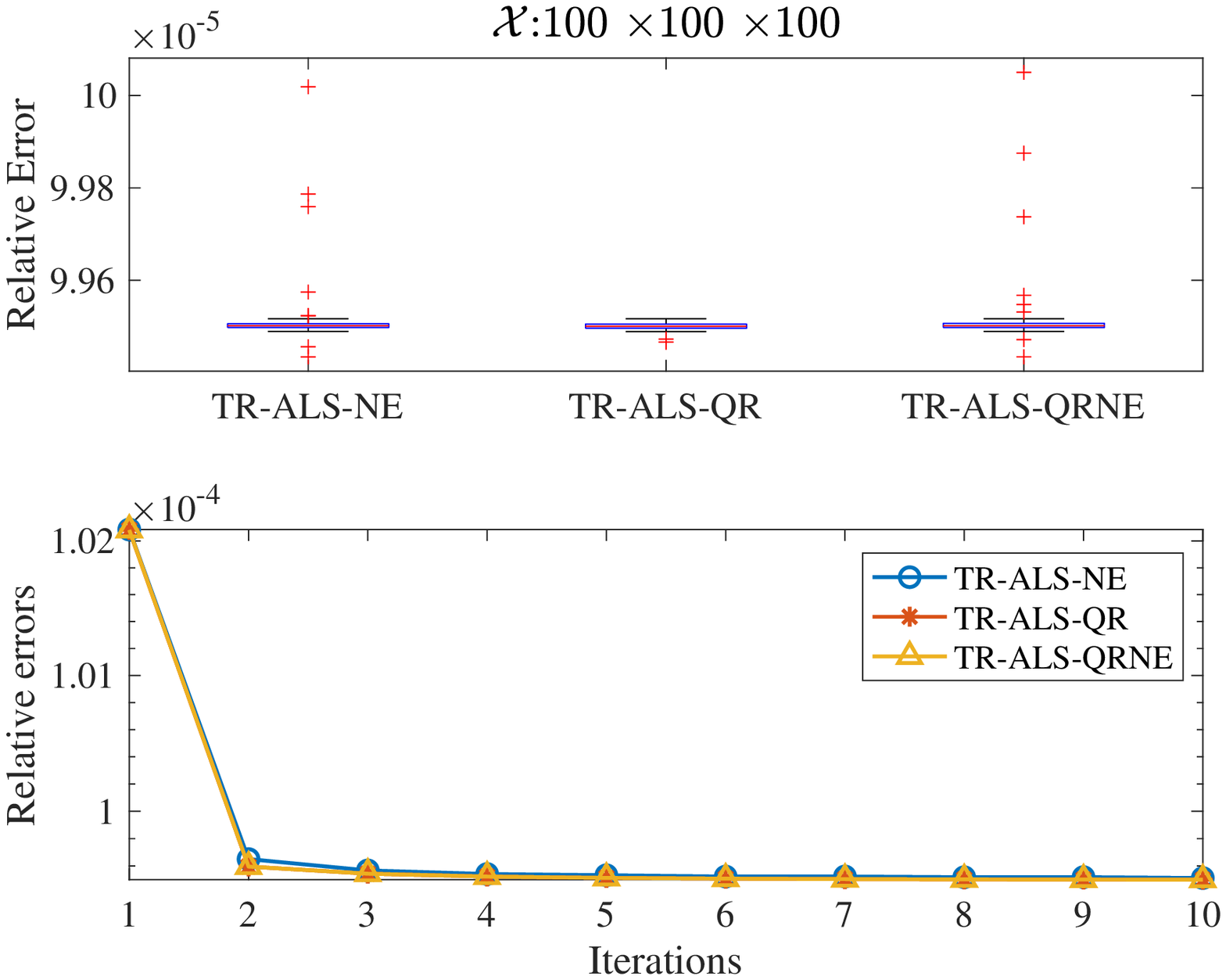}} 
	\subfloat[$\eta = 10^{-4}$, $\gamma = 1-10^{-10}$]{\includegraphics[scale=0.207]{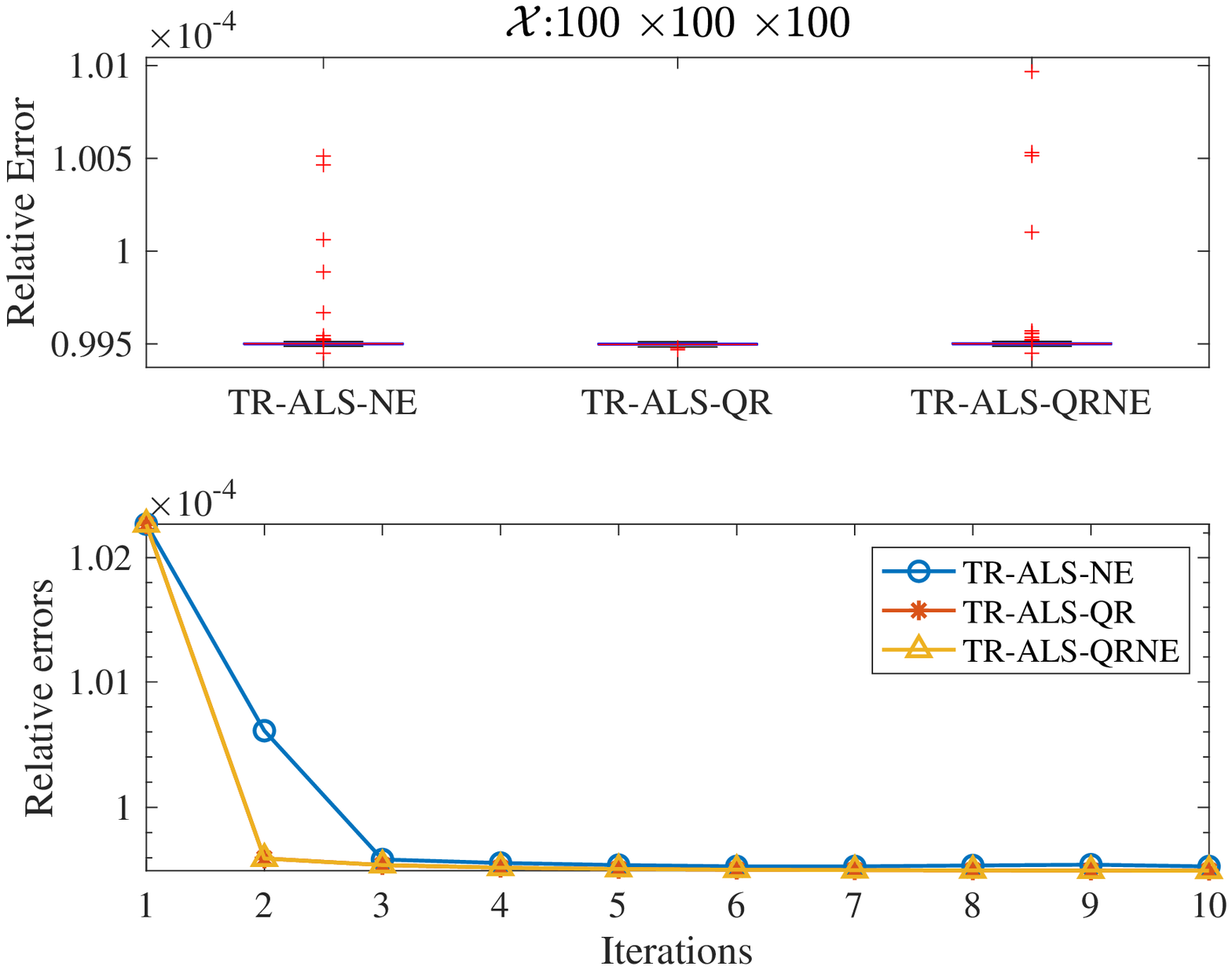}} 
	\quad
	\subfloat[$\eta = 10^{-7}$, $\gamma = 1-10^{-4}$]{\includegraphics[scale=0.207]{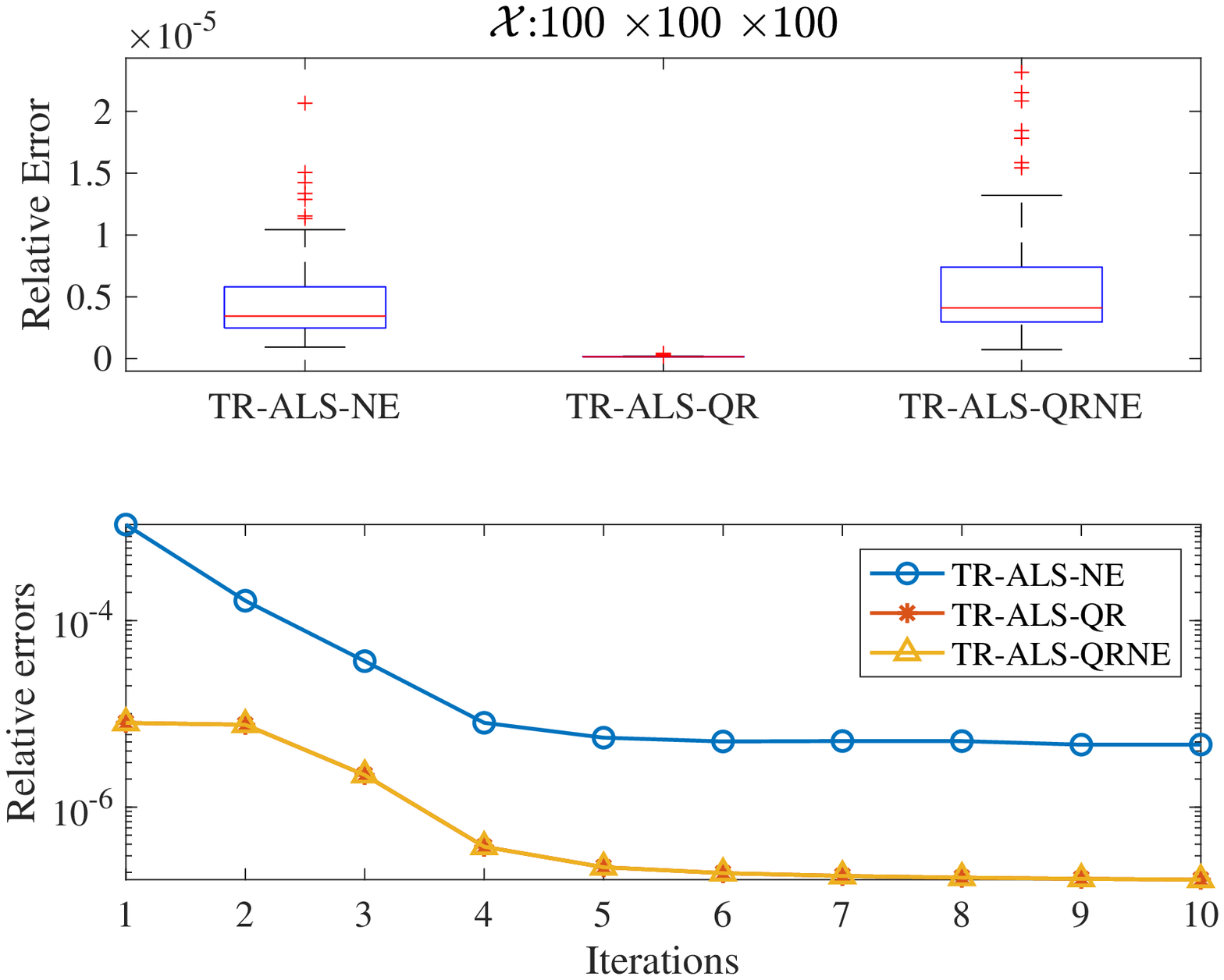}} 
	\subfloat[$\eta = 10^{-7}$, $\gamma = 1-10^{-7}$]{\includegraphics[scale=0.207]{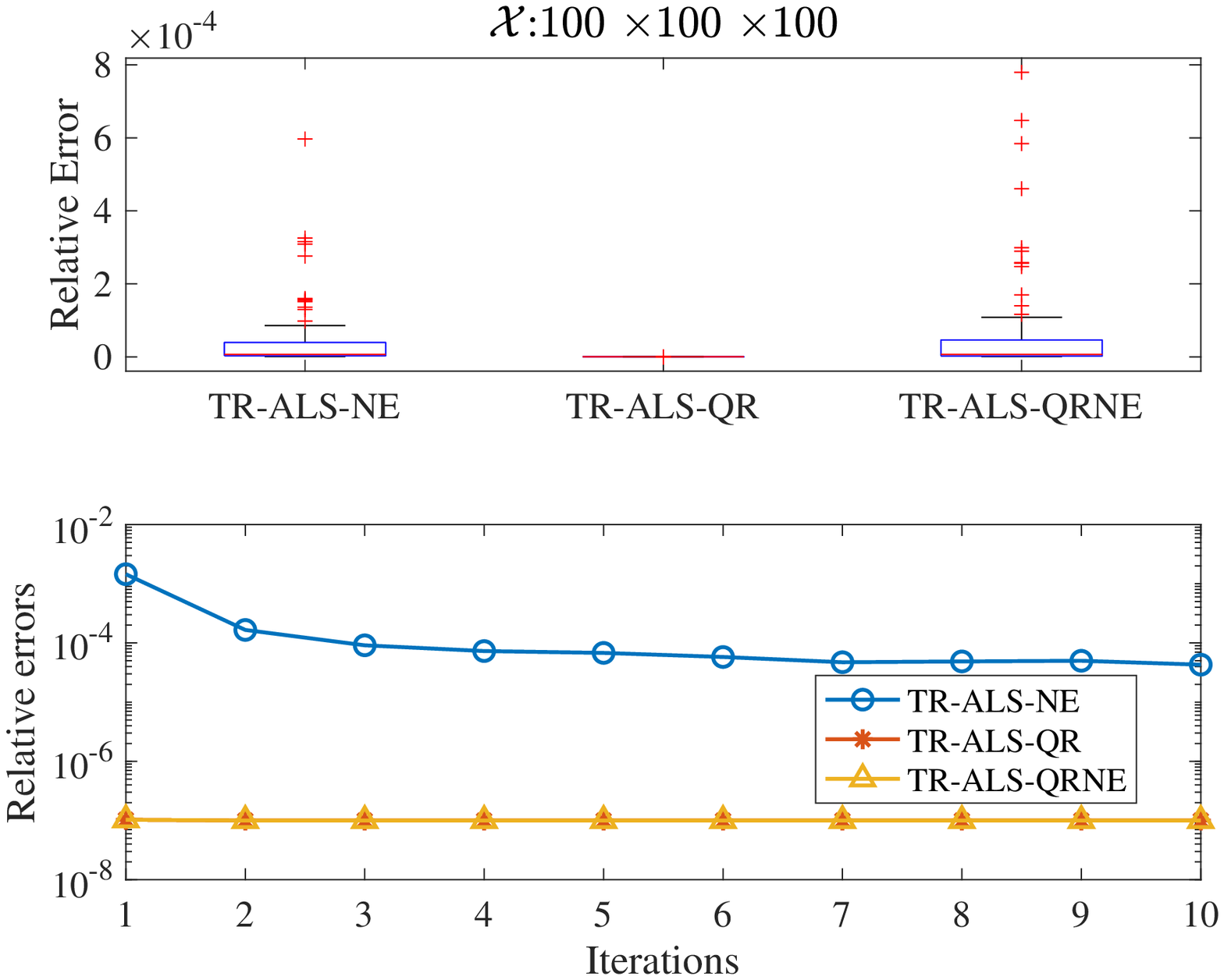}} 
	\subfloat[$\eta = 10^{-7}$, $\gamma = 1-10^{-10}$]{\includegraphics[scale=0.207]{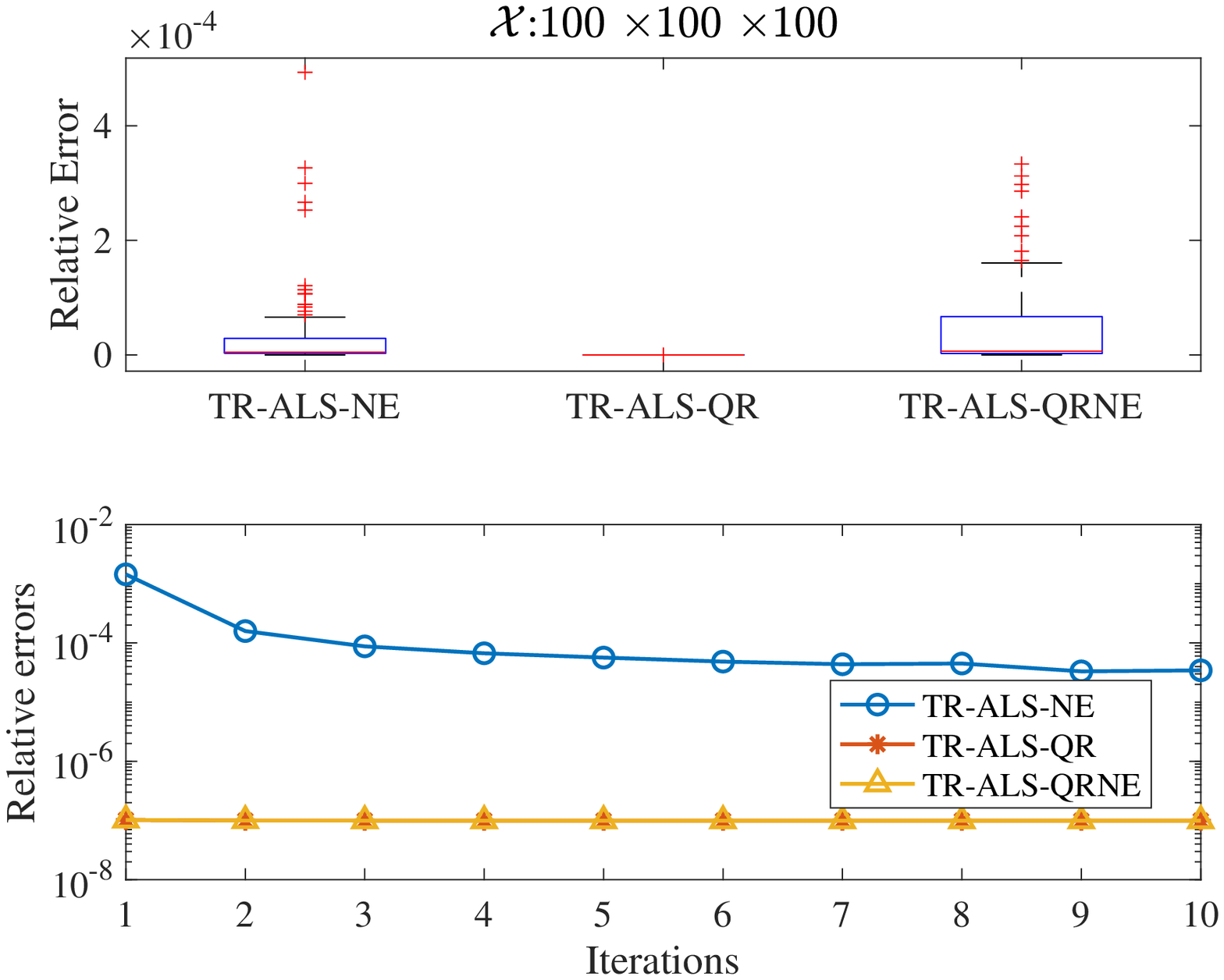}} 
	\quad
	\subfloat[$\eta = 10^{-10}$, $\gamma = 1-10^{-4}$]{\includegraphics[scale=0.207]{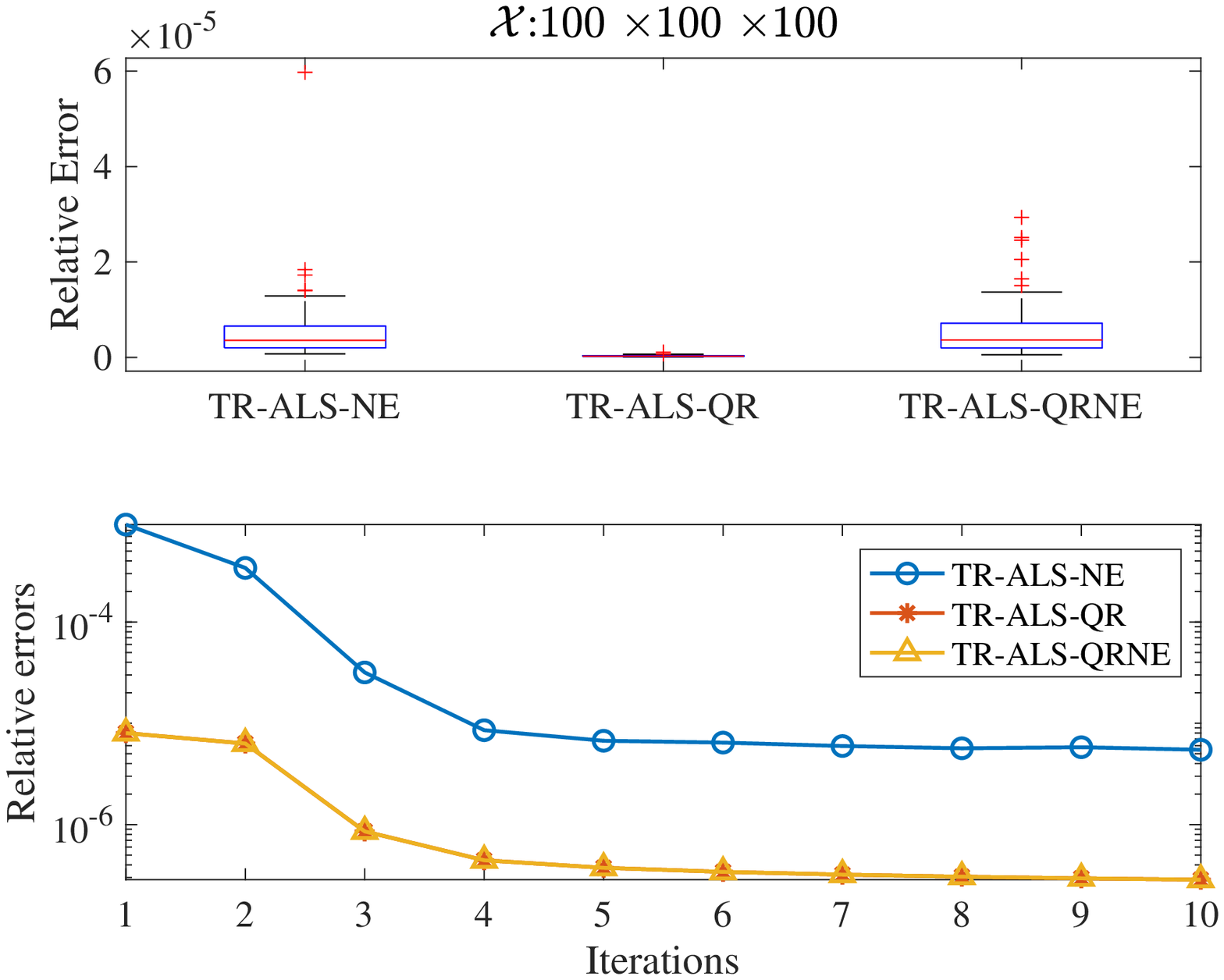}} 
	\subfloat[$\eta = 10^{-10}$, $\gamma = 1-10^{-7}$]{\includegraphics[scale=0.207]{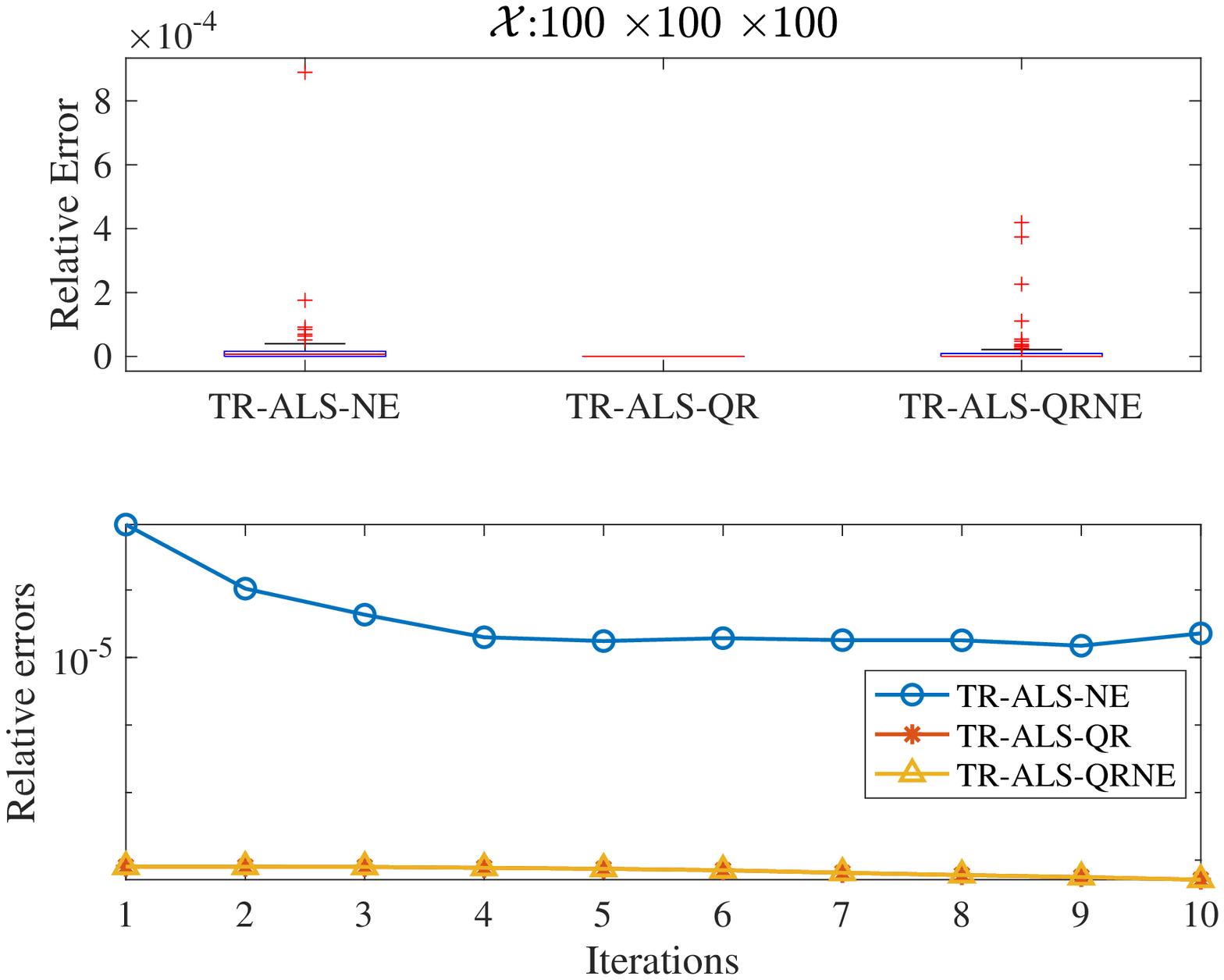}} 
	\subfloat[$\eta = 10^{-10}$, $\gamma = 1-10^{-10}$]{\includegraphics[scale=0.207]{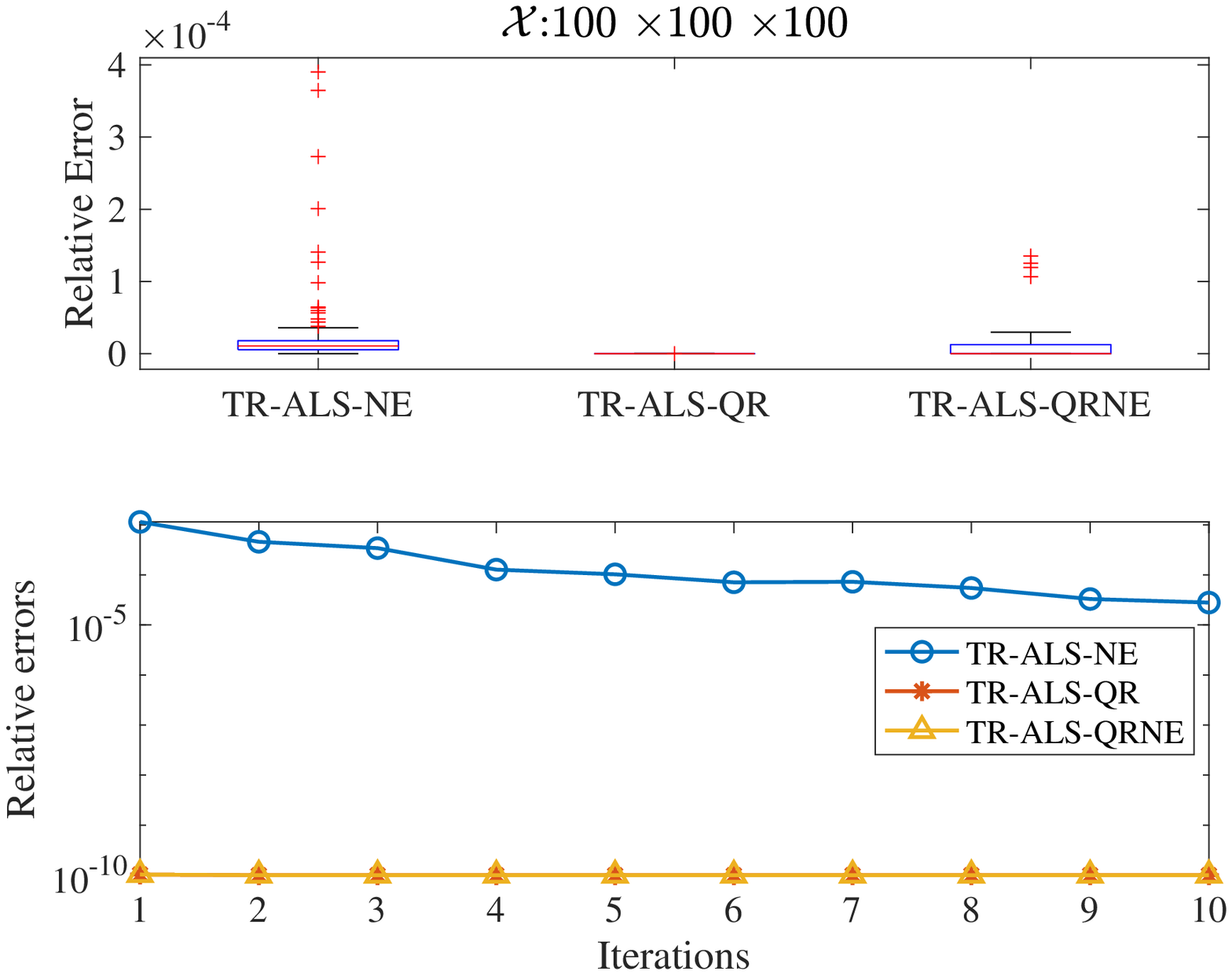}} 
	\caption{Boxplots and line chart of relative errors for TR-ALS-NE, TR-ALS-QR, and TR-ALS-QRNE on a $100 \times 100 \times 100$ synthetic tensor of rank 5 with three different levels of collinearity for the true TR-cores and three different levels of Gaussian noise added. Each algorithm is run 100 trails.}
	\label{fig:b2_collinear}
\end{figure}

\paragraph{Experiment B-\uppercase\expandafter{\romannumeral3}.}
In this experiment, the $5 \times 100 \times 5$ TR-cores are also constructed by reshaping some special $100 \times 25$ random matrices. The difference is that the entries of matrices are drawn from the multivariate $t$-distribution. Specifically,
we use the Matlab function \textsc{Mvtrnd}($\mat{C}$, $d$, 100) to implement this process, where $\mat{C} \in \bb{R}^{25 \times 25}$ is a correlation matrix whose ($i$, $j$)-th element is equal to $\theta^{|i-j|}$ with $\theta$ describing the correlation level, and $d$ is the degrees of freedom. We always set $d=1$ in our specific experiment.

We test all combinations of three different noise levels $10^{-4}$, $10^{-7}$, and $10^{-10}$, and three different correlation levels $1-10^{-1}$, $1-10^{-4}$, and $1-10^{-7}$, and run 100 trials of each algorithm for each configuration. 
\Cref{fig:b3_colstu} reports the numerical results, which are similar to the ones for Experiment B-\uppercase\expandafter{\romannumeral2}. That is, for well-conditioned cases, e.g., the one of $\eta = 10^{-4}$ and $\theta = 1-10^{-1}$, all the algorithms have similar performance.  
However, for ill-conditioned cases, e.g., the combinations of $\eta = 10^{-7}, 10^{-10}$ and $\theta = 1-10^{-4}, 1-10^{-7}$, TR-ALS-QR and TR-ALS-QRNE are still robust, that is, they obtain lower relative errors with little variation compared with TR-ALS-NE.

\begin{figure}[htbp] 
	\centering 
	\subfloat[$\eta = 10^{-4}$, $\theta = 1-10^{-1}$]{\includegraphics[scale=0.207]{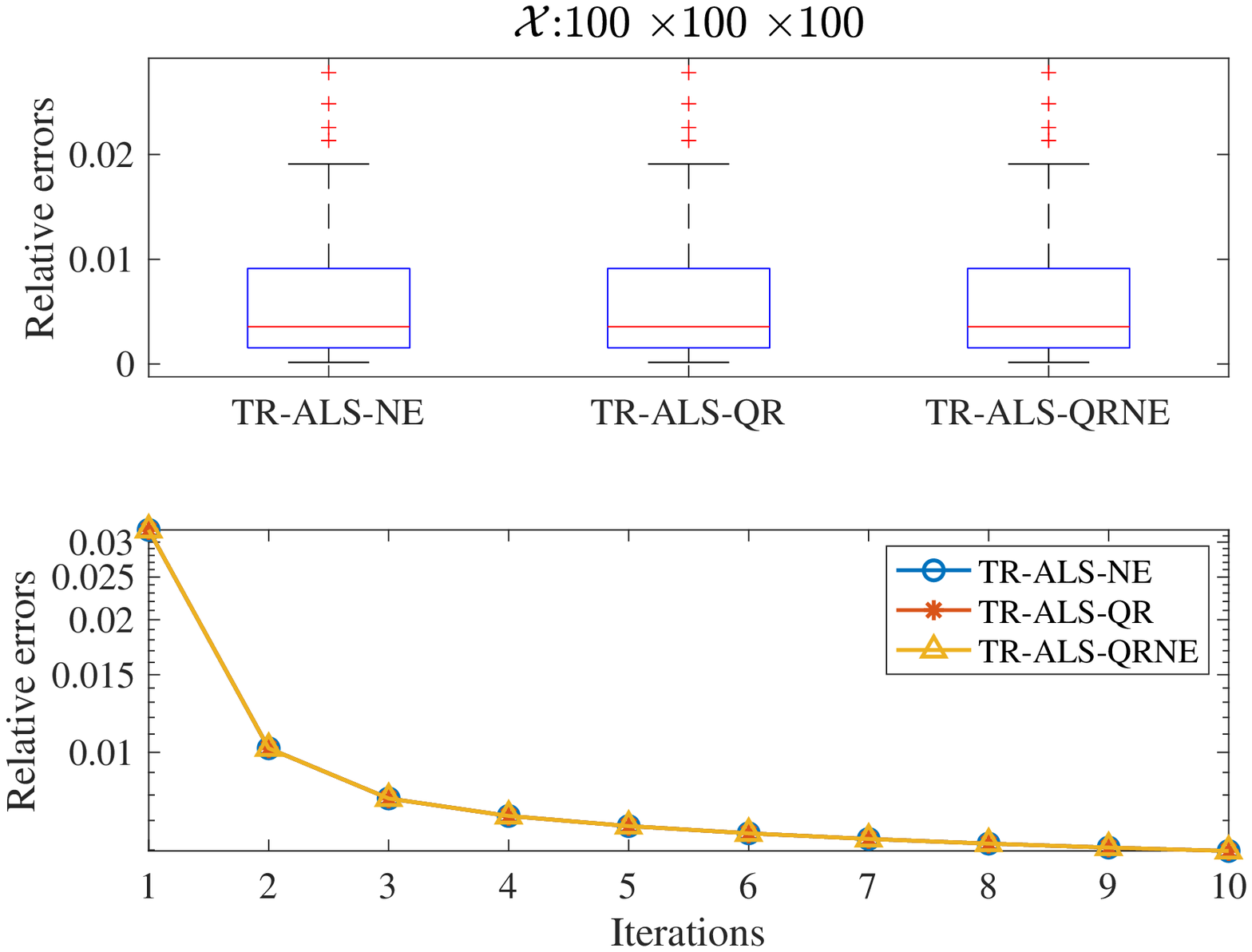}} 
	\subfloat[$\eta = 10^{-4}$, $\theta = 1-10^{-4}$]{\includegraphics[scale=0.207]{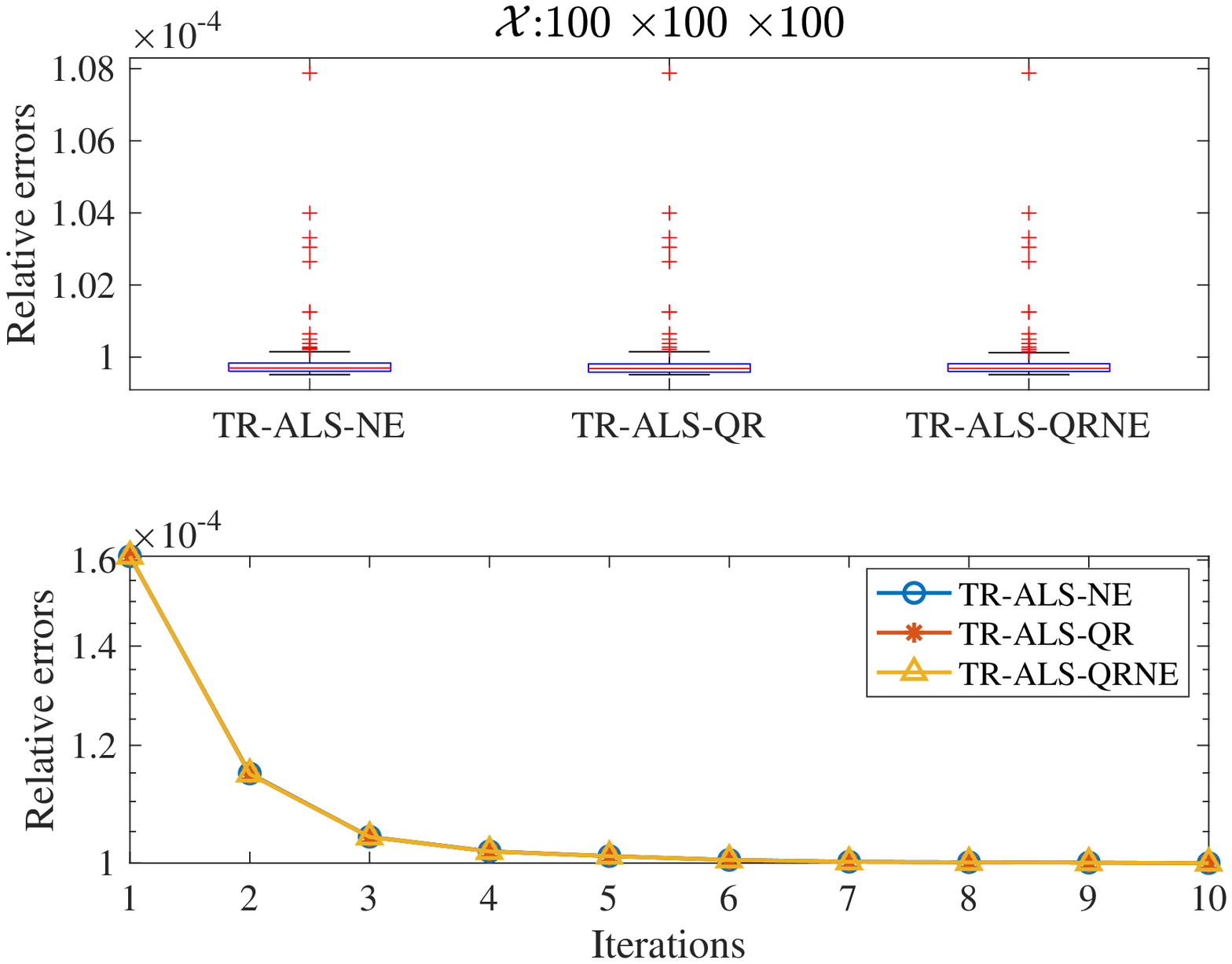}} 
	\subfloat[$\eta = 10^{-4}$, $\theta = 1-10^{-7}$]{\includegraphics[scale=0.207]{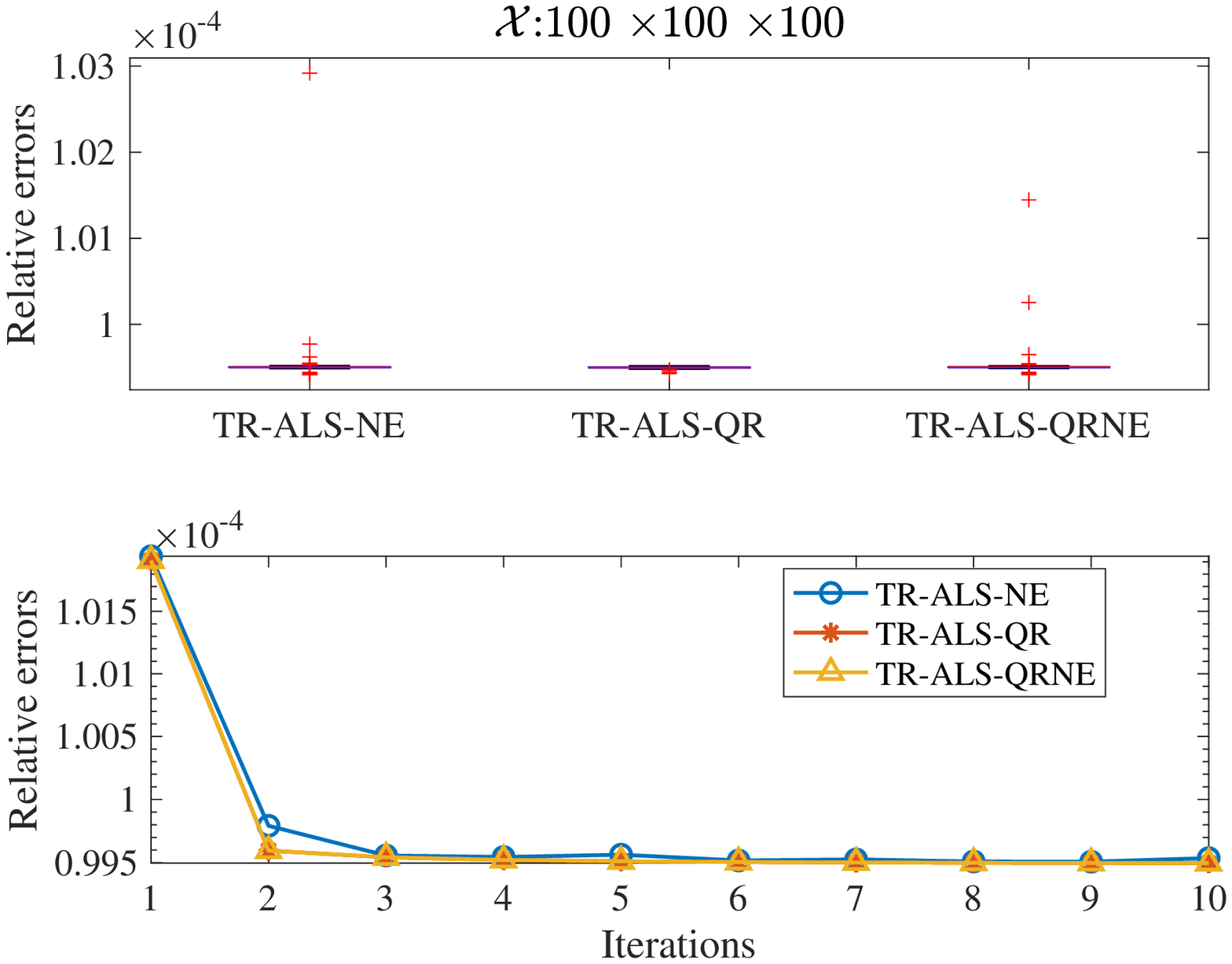}} 
	\quad
	\subfloat[$\eta = 10^{-7}$, $\theta = 1-10^{-1}$]{\includegraphics[scale=0.207]{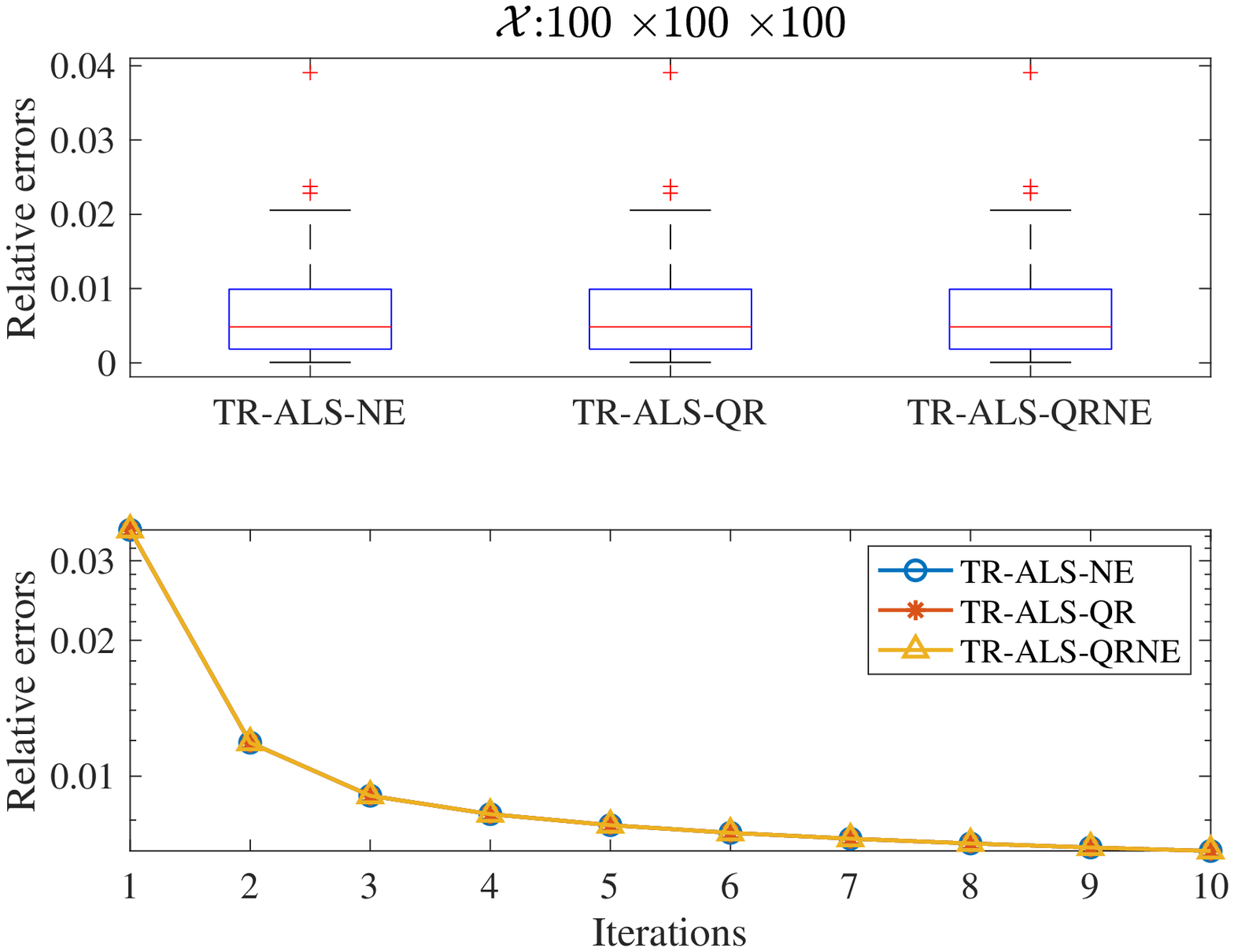}} 
	\subfloat[$\eta = 10^{-7}$, $\theta = 1-10^{-4}$]{\includegraphics[scale=0.207]{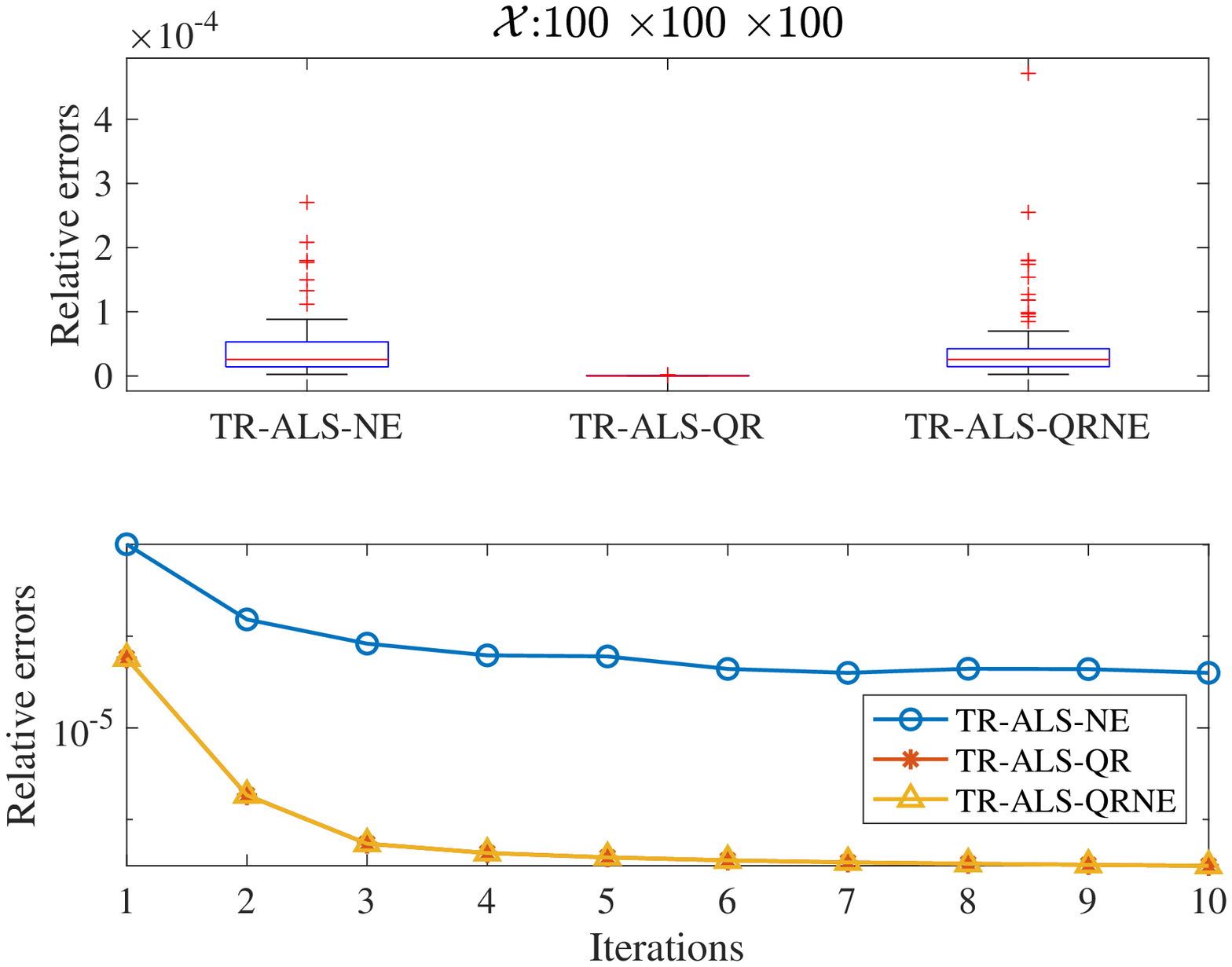}} 
	\subfloat[$\eta = 10^{-7}$, $\theta = 1-10^{-7}$]{\includegraphics[scale=0.207]{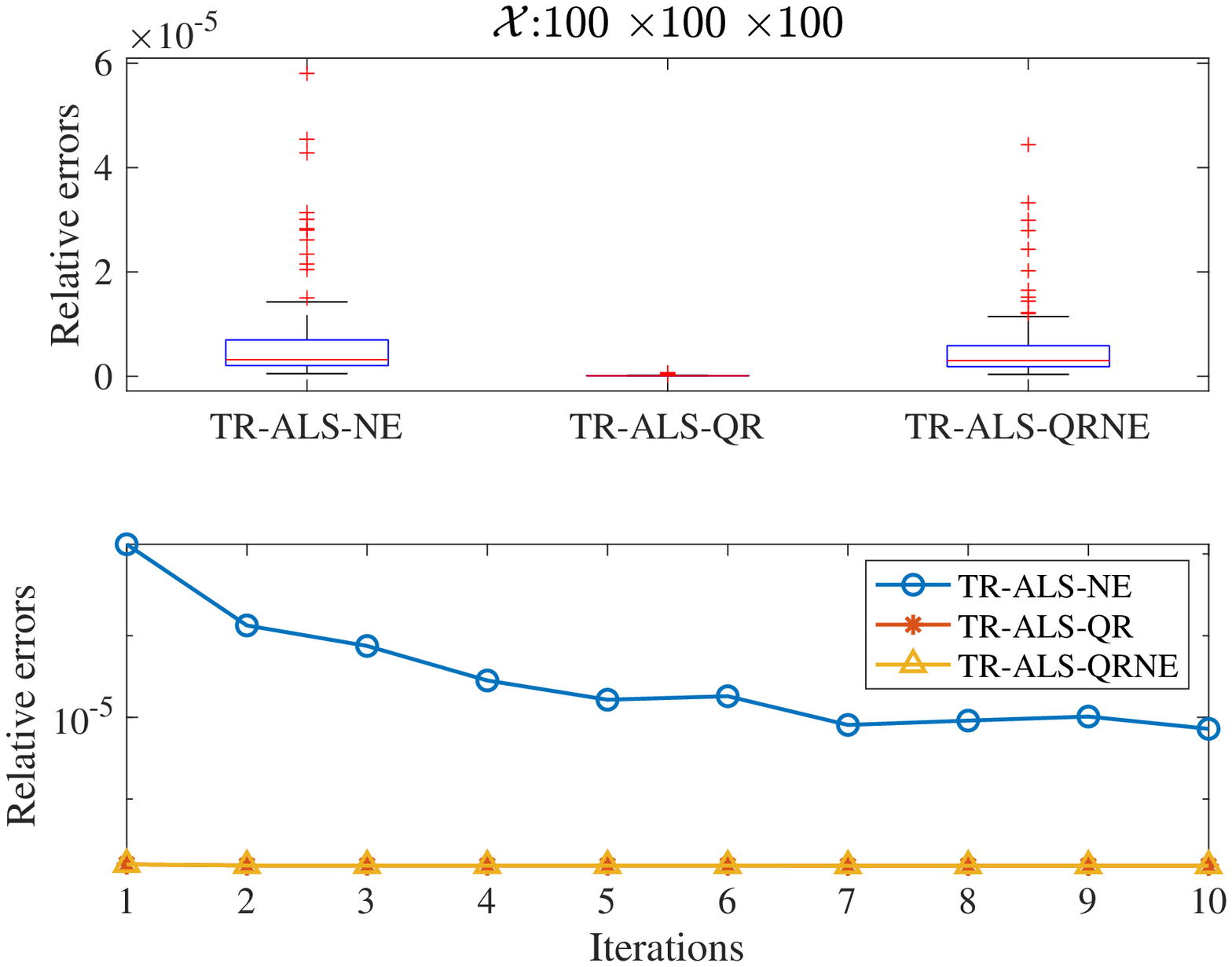}} 
	\quad
	\subfloat[$\eta = 10^{-10}$, $\theta = 1-10^{-1}$]{\includegraphics[scale=0.207]{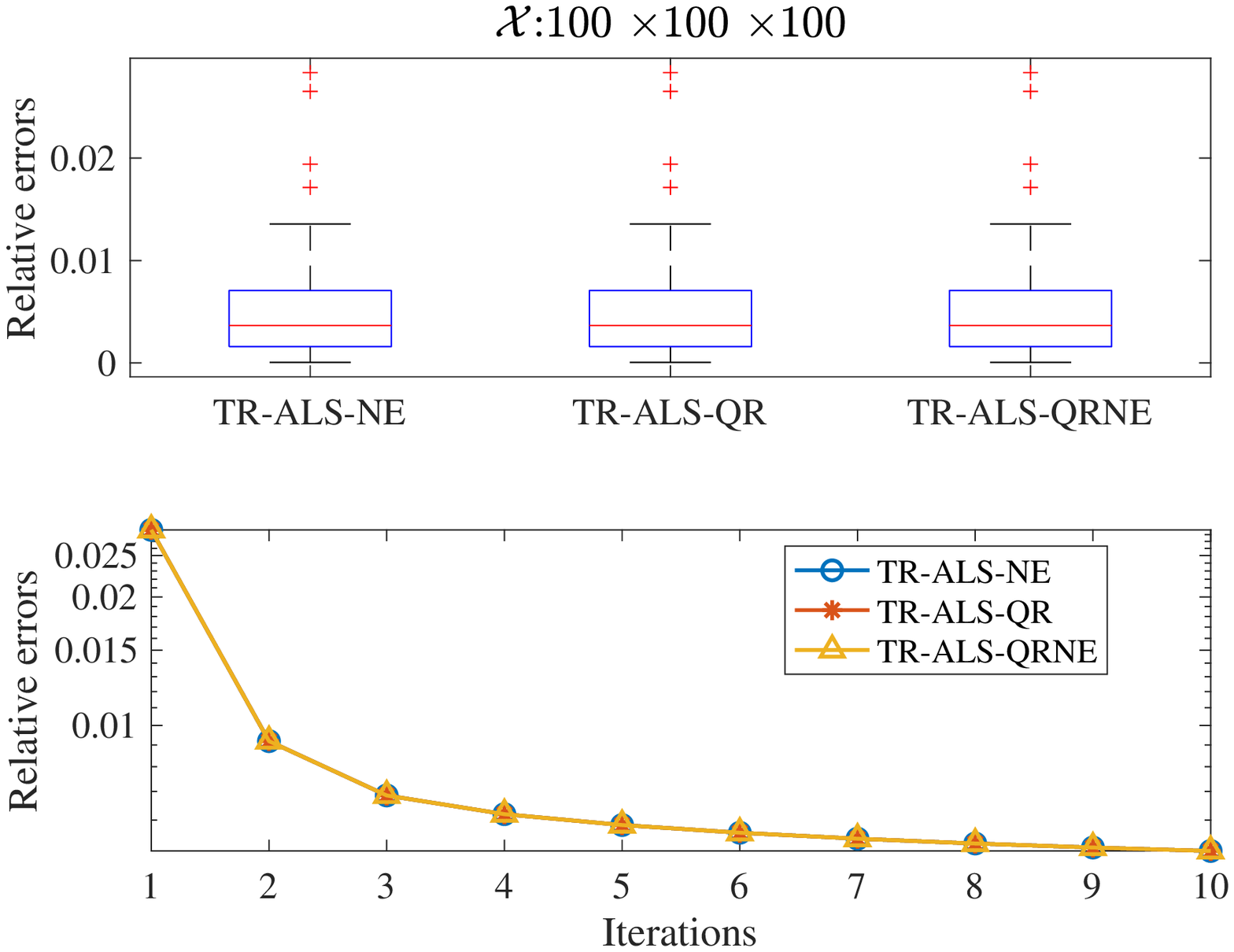}} 
	\subfloat[$\eta = 10^{-10}$, $\theta = 1-10^{-4}$]{\includegraphics[scale=0.207]{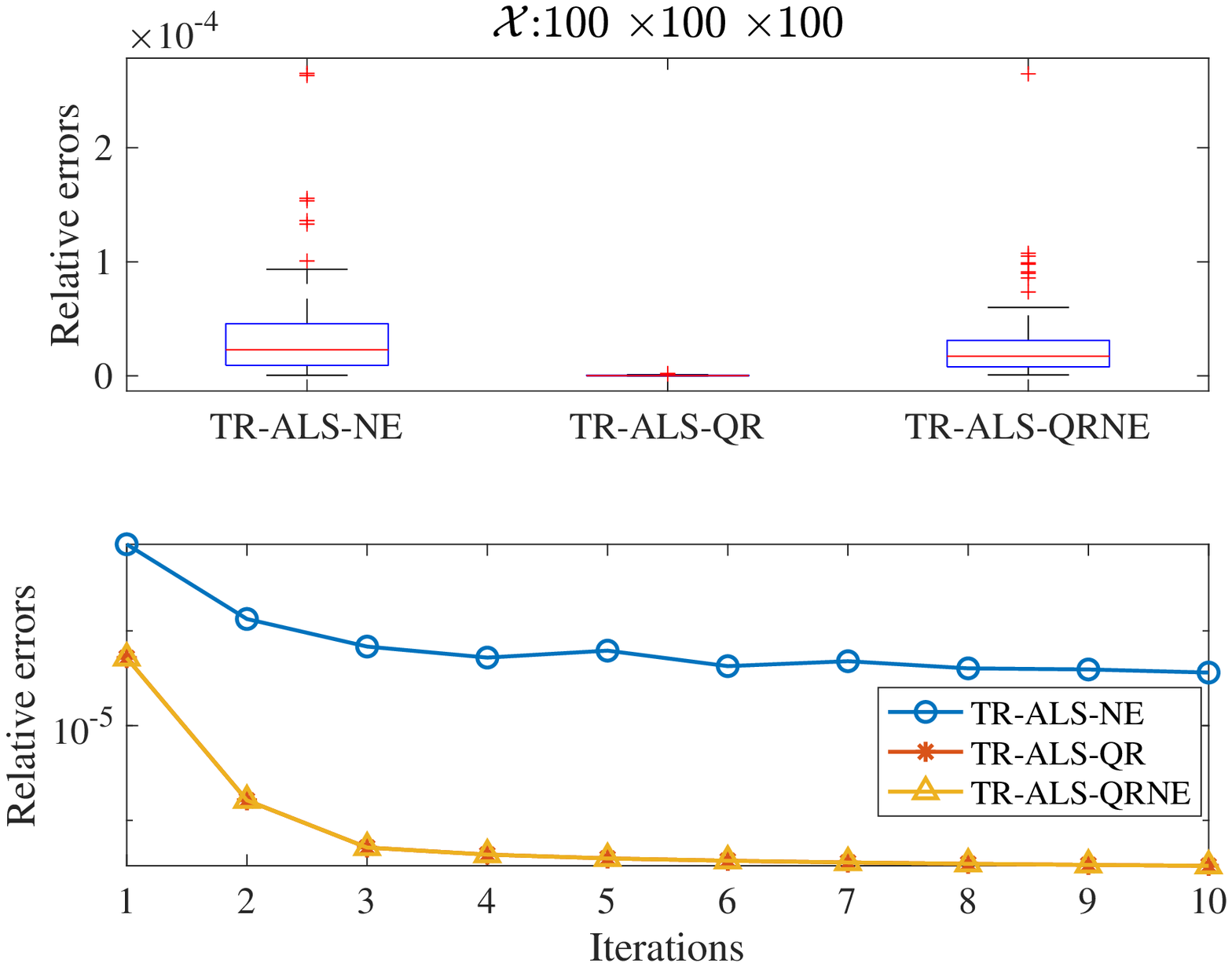}} 
	\subfloat[$\eta = 10^{-10}$, $\theta = 1-10^{-7}$]{\includegraphics[scale=0.207]{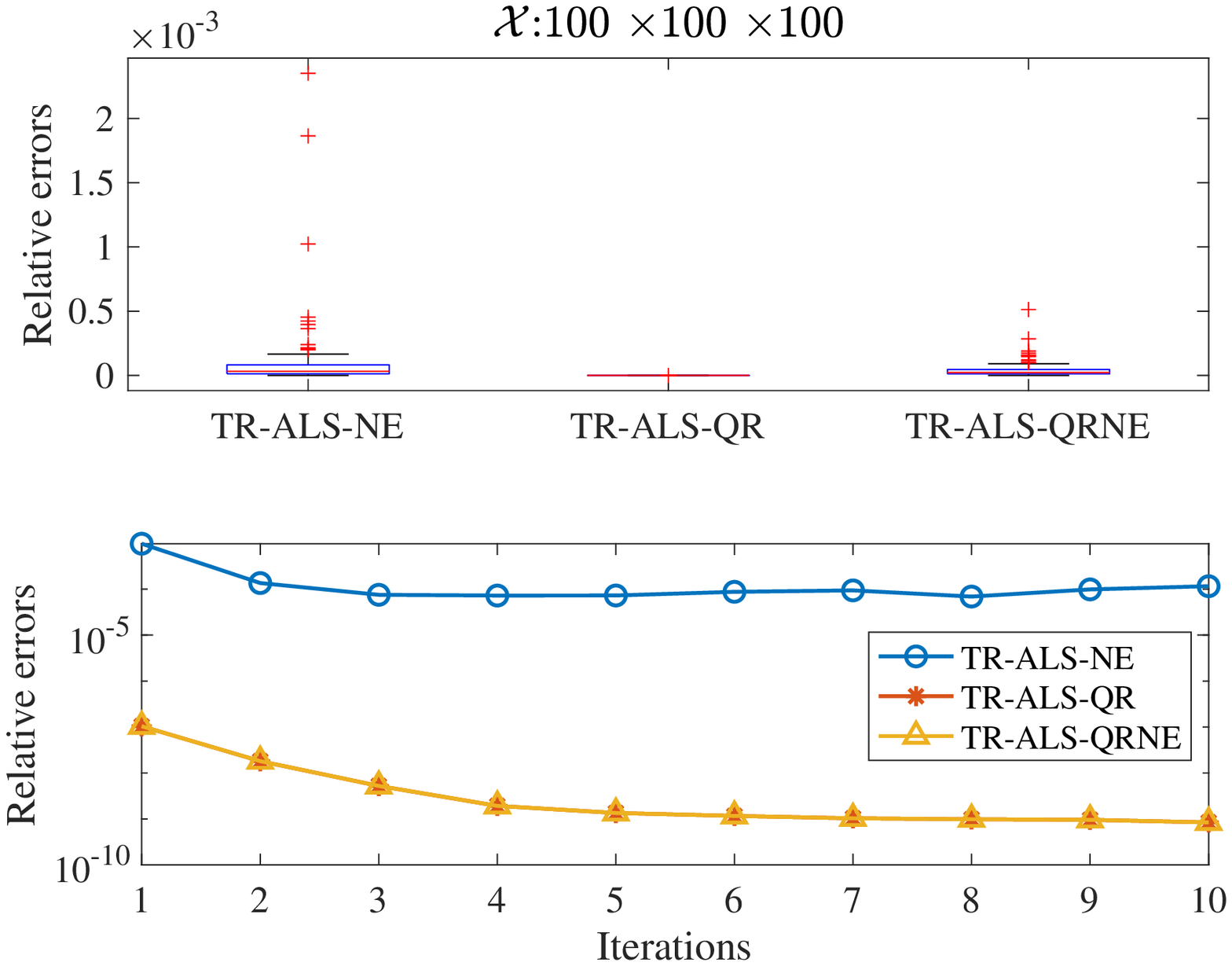}} 
	\caption{Boxplots and line chart of relative errors for TR-ALS-NE, TR-ALS-QR, and TR-ALS-QRNE on a $100 \times 100 \times 100$ synthetic tensor of rank 5 with three different kinds of the correlation matrices for the true TR-cores and three different levels of Gaussian noise added. Each algorithm is run 100 trails.}
	\label{fig:b3_colstu}
\end{figure}

\subsection{Performance on image and video data}

The brief information of three real image and video datasets is listed in \Cref{tab:real_data}.  
More specifically, \textbf{DC Mall} is a 3rd-order tensor containing the hyperspectral image, whose first two orders are the image height and width, and the third one is the number of spectral bands. 
\textbf{Park Bench} and \textbf{Tabby Cat} are 3rd-order tensors representing grayscale videos of a man sitting on a park bench and a tabby cat, respectively. The first two orders of them are the height and width of frames, and the third one is the number of frames. In addition, we also provide the links of the above data in footnotes.

\begin{table}[htbp]
	\centering
	\caption{Size and type of real datasets.} 
	\label{tab:real_data}
	\begin{tabular}{lll}
		\toprule
		Dataset & Size & Type \\
		\midrule
		DC Mall  \tablefootnote{\url{https://engineering.purdue.edu/~biehl/MultiSpec/}}    & $1280 \times 307 \times 191$   & Hyperspectral image \\
		Park Bench \tablefootnote{\url{https://www.pexels.com/video/man-sitting-on-a-bench-853751}}  & $1080 \times 1920 \times 364$ & Video \\
		Tabby Cat  \tablefootnote{\url{https://www.pexels.com/video/video-of-a-tabby-cat-854982/}}  & $720 \times 1280 \times 286$  & Video \\
		\bottomrule
	\end{tabular}
\end{table}

\paragraph{Lower-order tensors}
We implement our methods and TR-ALS on the data in \Cref{tab:real_data} directly with different target rank $R$. The numerical results are summarized in \Cref{tab:result_real_r3,tab:result_real_r10}, respectively. From these two tables, we can see that our three algorithms can achieve the same errors as TR-ALS but take much less time, and TR-ALS-QRNE is the fastest method. 
In addition, we can also find that when running a same algorithm with different target ranks, different errors are obtained. This is mainly because the closer the target rank is to the true rank, the more accurate the result of the decomposition is. On the other hand, the larger the target rank is, the longer time it takes to run the algorithm. 

\begin{table}[htbp]
	\centering
    \sisetup{round-mode=places,table-number-alignment=right,table-text-alignment=right}
	\caption{Decompositions for real datasets with target rank $R=3$.} 
	\label{tab:result_real_r3}
	\begin{tabular}{
			l 
			S[round-precision=3, table-figures-decimal=2, table-figures-integer=1]
			S[round-precision=1, table-figures-decimal=1, table-figures-integer=3]
			S[round-precision=3, table-figures-decimal=2, table-figures-integer=1]
			S[round-precision=1, table-figures-decimal=1, table-figures-integer=3]
			S[round-precision=3, table-figures-decimal=2, table-figures-integer=1]
			S[round-precision=1, table-figures-decimal=1, table-figures-integer=3]
		} 
		\toprule
		& \multicolumn{2}{c}{DC Mall} & \multicolumn{2}{c}{Park Bench} & \multicolumn{2}{c}{Tabby Cat} \\
		\cmidrule(lr){2-3}
		\cmidrule(lr){4-5}
		\cmidrule(lr){6-7}
		Method 		 & {Error} & {Time (s)} & {Error} & {Time (s)}  & {Error} & {Time (s)} \\
		\midrule
		TR-ALS 	     & 0.331144154752868 & 54.6150130400000   & 0.183328486150838 & 519.209623170000   & 0.189027280797860 & 188.651279860000 \\
		\midrule
		TR-ALS-NE    & 0.331144154752868 & 5.62500080000000   & 0.183328486150837 & 47.9150986600000   & 0.189027280797860 & 19.7737110600000 \\
		TR-ALS-QR    & 0.331144154752868 & 2.75662358000000   & 0.183328486150838 & 33.2959255000000   & 0.189027280797860 & 12.1747424200000 \\
		TR-ALS-QRNE  & 0.331144154752868 & 2.75732526000000   & 0.183328486150837 & 33.3094532100000   & 0.189027280797860 & 12.1206367100000 \\
		\bottomrule
	\end{tabular}
\end{table}	

\begin{table}[htbp]
	\centering
    \sisetup{round-mode=places,table-number-alignment=right,table-text-alignment=right}
	\caption{Decompositions for real datasets with target rank $R=10$.} 
	\label{tab:result_real_r10}
	\resizebox{1\linewidth}{!}{
	\begin{tabular}{
			l
			S[round-precision=3, table-figures-decimal=2, table-figures-integer=1]
			S[round-precision=1, table-figures-decimal=1, table-figures-integer=3]
			S[round-precision=3, table-figures-decimal=2, table-figures-integer=1]
			S[round-precision=1, table-figures-decimal=1, table-figures-integer=3]
			S[round-precision=3, table-figures-decimal=2, table-figures-integer=1]
			S[round-precision=1, table-figures-decimal=1, table-figures-integer=3]
		}  
		\toprule
		& \multicolumn{2}{c}{DC Mall} & \multicolumn{2}{c}{Park Bench} & \multicolumn{2}{c}{Tabby Cat} \\
		\cmidrule(lr){2-3}
		\cmidrule(lr){4-5}
		\cmidrule(lr){6-7}
		Method 		 & {Error} & {Time (s)} & {Error} & {Time (s)}  & {Error} & {Time (s)} \\
		\midrule
		TR-ALS 	     & 0.158708353596765 & 435.494658880000   & 0.0726437008272118 & 4114.35873270000   & 0.133891697950235 & 1475.19881578000 \\
		\midrule
		TR-ALS-NE    & 0.158708353596779 & 10.8201581300000   & 0.0726437008272114 & 75.7809652600000   & 0.133891697950235 & 31.2755252700000 \\
		TR-ALS-QR    & 0.158708353596765 & 8.54181705000000   & 0.0726437008272118 & 64.1277507200000   & 0.133891697950235 & 25.7211072400000 \\
		TR-ALS-QRNE  & 0.158708353596750 & 8.04869038000000   & 0.0726437008271361 & 63.2440317100000   & 0.133891697950234 & 25.3862390800000 \\
		\bottomrule
	\end{tabular}
}
\end{table}

\paragraph{Higher-order tensors}
The so-called higher-order tensors listed in \Cref{tab:real_data_reshape} are truncated and reshaped from the data in \Cref{tab:real_data}.  
For example, DC Mall is first truncated to size $1280 \times 306 \times 190$ and then reshaped into a $32 \times 40 \times 18 \times 17 \times 10 \times 19$ tensor.  
The numerical results of various algorithms with target rank $R=3$ on these higher-order tensors are summarized in \Cref{tab:result_real_reshape}, where we see a similar finding to lower-order cases, i.e., our algorithms have much better performance compared with TR-ALS. 
In addition, comparing \Cref{tab:result_real_reshape} with \Cref{tab:result_real_r3} shows that although the same target rank is used for the same dataset, the decompositions obtained from the reshaped tensors are less accurate than those obtained from the original tensors and the former also takes a longer time than the latter. This is mainly because the reshaped tensors change the structural information of the original data. 

\begin{table}[htbp]
	\centering
	\caption{Size of truncated and reshaped real datasets in \Cref{tab:real_data}.} 
	\label{tab:real_data_reshape}
	\begin{tabular}{ll}
		\toprule
		Dataset & Size \\ 
		\midrule
		DC Mall (reshaped)     & $32 \times 40 \times 18 \times 17 \times 10 \times 19$ \\
		Park Bench (reshaped)  & $24 \times 45 \times 32 \times 60 \times 28 \times 13$ \\
		Tabby Cat (reshaped)   & $16 \times 45 \times 32 \times 40 \times 13 \times 22$ \\
		\bottomrule
	\end{tabular}
\end{table}

\begin{table}[htbp]
	\centering
	\sisetup{round-mode=places,table-number-alignment=right,table-text-alignment=right}
	\caption{Decompositions for real datasets with target rank $R=3$.} 
	\label{tab:result_real_reshape}
	\resizebox{1\linewidth}{!}{
	\begin{tabular}{
			l
			S[round-precision=3, table-figures-decimal=2, table-figures-integer=1]
			S[round-precision=1, table-figures-decimal=1, table-figures-integer=3]
			S[round-precision=3, table-figures-decimal=2, table-figures-integer=1]
			S[round-precision=1, table-figures-decimal=1, table-figures-integer=3]
			S[round-precision=3, table-figures-decimal=2, table-figures-integer=1]
			S[round-precision=1, table-figures-decimal=1, table-figures-integer=3]
		}  
		\toprule
		& \multicolumn{2}{c}{DC Mall (reshaped)} & \multicolumn{2}{c}{Park Bench (reshaped)} & \multicolumn{2}{c}{Tabby Cat (reshaped)} \\
		\cmidrule(lr){2-3}
		\cmidrule(lr){4-5}
		\cmidrule(lr){6-7}
		Method 		 & {Error} & {Time (s)} & {Error} & {Time (s)}  & {Error} & {Time (s)} \\
		\midrule
		TR-ALS 		 & 0.383668315557150 & 114.892324890000   & 0.213643362405975 & 914.651634300000   & 0.196877559789126 & 351.144208150000 \\
		\midrule
		TR-ALS-NE    & 0.383668315557147 & 22.1979294400000   & 0.213643362405975 & 189.036449460000   & 0.196877559789130 & 74.2533063000000 \\
		TR-ALS-QR    & 0.383668315557150 & 13.3313218100000   & 0.213643362405975 & 107.425923470000   & 0.196877559789126 & 53.7369941600000 \\
		TR-ALS-QRNE  & 0.383668315557148 & 12.9254576200000   & 0.213643362405976 & 106.966414100000   & 0.196877559789122 & 53.3766223500000 \\
		\bottomrule
	\end{tabular}
}
\end{table}

\section{Concluding Remarks}
In this paper, we propose three practical ALS-based algorithms for TR decomposition, i.e., TR-ALS-NE, TR-ALS-QR, and TR-ALS-QRNE. They can make full use of the structure of the coefficient matrices of the TR-ALS subproblems. To achieve this, we present a new property of the subchain product of tensors and extend the QR factorization of the matrix to the 3rd-order tensor. Numerical results show that TR-ALS-NE can be much faster than the regular TR-ALS, and the QR-based methods are in turn more stable than TR-ALS-NE.

There are several potential performance improvements to pursue in future work.  
One is to exploit the structure of the subchain product of some 3rd-order tensors whose mode-2 unfolding matrices are upper triangular. 
Another one is to speed up MTTSP, whose counterpart in CP decomposition is MTTKRP which has been investigated extensively.  
Moreover, it is also interesting to combine our methods with randomized techniques to further reduce computational costs.


%
\section*{Acknowledgements}
The work is supported by the National Natural Science Foundation of China (no. 11671060,
11771099) and the Natural Science Foundation of Chongqing, China (no. cstc2019jcyj-msxmX0267).

\bibliographystyle{spmpsci}
\bibliography{jsc_ref}

\end{sloppypar}
\end{document}